\documentclass[11pt]{article}

\usepackage{amsmath, amsfonts, amssymb, amsthm,  graphicx, mathtools, enumerate}

\usepackage[blocks, affil-it]{authblk}

\usepackage{mathrsfs}

\usepackage[numbers, square]{natbib}
\usepackage[CJKbookmarks=true,
            bookmarksnumbered=true,  
			bookmarksopen=true,
			colorlinks=true,
			citecolor=red,
			linkcolor=blue,
			anchorcolor=red,
			urlcolor=blue]{hyperref}  
\usepackage[usenames]{xcolor}

\usepackage[letterpaper, left=1.2truein, right=1.2truein, top = 1.2truein, bottom = 1.2truein]{geometry}
\usepackage[ruled, vlined, lined, commentsnumbered]{algorithm2e}

\usepackage{prettyref,soul}
\usepackage[font=small,labelfont=it]{caption}

\newtheorem{lemma}{Lemma}[section]
\newtheorem{proposition}{Proposition}[section]
\newtheorem{thm}{Theorem}[section]

\newtheorem{corollary}{Corollary}[section]
\newtheorem{remark}{Remark}[section]

\newrefformat{eq}{(\ref{#1})}
\newrefformat{chap}{Chapter~\ref{#1}}
\newrefformat{sec}{Section~\ref{#1}}
\newrefformat{algo}{Algorithm~\ref{#1}}
\newrefformat{fig}{Fig.~\ref{#1}}
\newrefformat{tab}{Table~\ref{#1}}
\newrefformat{rmk}{Remark~\ref{#1}}
\newrefformat{clm}{Claim~\ref{#1}}
\newrefformat{def}{Definition~\ref{#1}}
\newrefformat{cor}{Corollary~\ref{#1}}
\newrefformat{lmm}{Lemma~\ref{#1}}
\newrefformat{lemma}{Lemma~\ref{#1}}
\newrefformat{prop}{Proposition~\ref{#1}}
\newrefformat{app}{Appendix~\ref{#1}}
\newrefformat{ex}{Example~\ref{#1}}
\newrefformat{exer}{Exercise~\ref{#1}}
\newrefformat{soln}{Solution~\ref{#1}}
\newrefformat{cond}{Condition~\ref{#1}}

\def\Var{\textsf{Var}} 

\def\text#1{\mbox{\rm #1}}

\newcommand{\ess}{\mathop{\rm ess}}

\newcommand{\indc}[1]{{\mathbf{1}_{\left\{{#1}\right\}}}}
\newcommand{\norm}[1]{\left\|{#1} \right\|}

\newcommand{\wh}{\widehat}
\newcommand{\wt}{\widetilde}

\newcommand{\fnorm}[1]{\|#1\|_{\rm F}}
\newcommand{\opnorm}[1]{\|#1\|_{\rm op}}

\newcommand{\sgn}{\mathop{\sf sign}}
\newcommand{\rt}{\mathop{\sf root}}

\newcommand{\Tr}{\mathop{\sf Tr}}

\newcommand{\iprod}[2]{\left \langle #1, #2 \right\rangle}

\newtheorem*{condb'}{Condition B'}

\title{Testing Equivalence of Clustering
}
\author[1]{Chao Gao}
\author[2]{Zongming Ma}
\affil[1]{
University of Chicago
}
\affil[2]{
University of Pennsylvania
}

\begin{document}
\maketitle

\begin{abstract}
 {In this paper, we test whether two datasets share a common clustering structure. 
As a leading example, we focus on comparing clustering structures in two independent random samples from two mixtures of multivariate Gaussian distributions. 
Mean parameters of these Gaussian distributions are treated as potentially unknown nuisance parameters and are allowed to differ.
Assuming knowledge of mean parameters, we first determine the phase diagram of the testing problem over the entire range of signal-to-noise ratios by providing both lower bounds and tests that achieve them. 
When nuisance parameters are unknown, we propose tests that achieve the detection boundary adaptively as long as ambient dimensions of the datasets grow at a sub-linear rate with the sample size.}
\smallskip

\end{abstract}


\section{Introduction}
 
Clustering analysis is one of the most important tasks in unsupervised learning.
In the context of Gaussian mixture model, we have independent observations $X_1,\cdots,X_n\sim N(z_i\theta,I_p)$, where $z_i\in\{-1,1\}$ for each $i\in[n]$ and $\theta\in\mathbb{R}^p$. 
In this setting,
clustering is equivalent to estimating the unknown label vector $z\in\{-1,1\}^n$. 
It is known that the minimax risk of estimating $z$ is given by
\begin{equation} 
\inf_{\wh{z}}\sup_{z\in\{-1,1\}^n}\mathbb{E}_{(z,\theta)}\ell(\wh{z},z)=\exp\left(-(1+o(1))\frac{\|\theta\|^2}{2}\right), \label{eq:yu-lu-diao}
\end{equation}
as long as $\|\theta\|^2\rightarrow\infty$.
See, for instance, \cite{lu2016statistical}.
Throughout the paper, the distance between two clustering structures is defined as
$$
\ell(\wh{z},z)=\left(\frac{1}{n}\sum_{i=1}^n\indc{\wh{z}_i\neq z_i}\right)\wedge \left(\frac{1}{n}\sum_{i=1}^n\indc{\wh{z}_i\neq -z_i}\right),
$$
the normalized Hamming distance up to a label switching.
Taking minimum over label switching is necessary since switching labels does not alter clustering structure.
Here and after, $a\wedge b = \min(a,b)$ for any real numbers $a$ and $b$.
Since the exponent is in the form of $\|\theta\|^2/2$, 
formula (\ref{eq:yu-lu-diao}) suggests that more covariates help to increase the clustering accuracy as they increase $\|\theta\|^2$. 
To be concrete, suppose we additionally have independent observations $Y_1,\cdots,Y_n\sim N(z_i\eta, I_q)$ for some $\eta\in\mathbb{R}^q$. By combining the datasets $X$ and $Y$, the error rate can be improved from (\ref{eq:yu-lu-diao}) to
\begin{equation}
\exp\left(-(1+o(1))\frac{\|\theta\|^2+\|\eta\|^2}{2}\right). \label{eq:improved-clustering-rate}
\end{equation}

In fact, integrating different sources of data to improve the performance of clustering analysis is a common practice in many areas. For example, in cancer genomics, researchers combine molecular features such as copy number variation and gene expression in integrative clustering to reveal novel tumor subtypes \citep{shen2009integrative,curtis2012genomic,mo2013pattern}. 
In collaborative filtering, combining different databases helps to better identify user types and thus makes better recommendations \citep{ma2008sorec,ma2011recommender,wang2013location}. 
In covariate-assisted network clustering \citep{binkiewicz2017covariate,deshpande2018contextual}, additional variables are collected to facilitate the clustering of nodes in a social network.

In the present paper, we investigate the hypothesis underpinning the foregoing practices: two datasets share a common clustering structure.
To be concrete, let us consider 
independent samples
$X_i\sim N(z_i\theta,I_p)$ and $Y_i\sim N(\sigma_i\eta,I_q)$ with some $z_i\in\{-1,1\}$ and $\sigma_i\in\{-1,1\}$ 
for all $i\in[n]$. 
If one uses $\{X_i\}_{1\le i\leq n}$ and $\{Y_i\}_{1\le i\leq n}$ to cluster the subjects $\{1,2,\dots,n\}$ into two disjoint subsets, it is implicitly assumed that $\ell(z,\sigma)=0$. 
From a statistical viewpoint, checking whether $\ell(z,\sigma)=0$ is equivalent to testing
\begin{equation}
H_0:\ell(z,\sigma)=0 \quad \mbox{vs.} \quad H_1:\ell(z,\sigma)> \epsilon \label{eq:problem}
\end{equation}
for some $\epsilon \geq 0$.
Let $P^{(n)}_{(\theta,\eta,z,\sigma)}$ be the joint distribution of the two datasets $\{X_i\}_{1\le i\leq n}$ and $\{Y_i\}_{1\le i\leq n}$. 
Given any testing procedure $\psi$, we define its worst-case testing error by
\begin{equation}
R_n(\psi,\theta,\eta,\epsilon)=\sup_{\substack{z\in\{-1,1\}^n\\\sigma\in\{-1,1\}^n\\\ell(z,\sigma)=0}}P^{(n)}_{(\theta,\eta,z,\sigma)}\psi + \sup_{\substack{z\in\{-1,1\}^n\\\sigma\in\{-1,1\}^n\\\ell(z,\sigma)> \epsilon}}P^{(n)}_{(\theta,\eta,z,\sigma)}(1-\psi). \label{eq:def-worst-risk}
\end{equation}
We call $\psi$ consistent if $R_n(\psi,\theta,\eta,\epsilon)\rightarrow 0$ as $n\rightarrow\infty$. The minimax testing error is defined by
$$R_n(\theta,\eta,\epsilon)=\inf_{\psi}R_n(\psi,\theta,\eta,\epsilon).$$
We will find necessary and sufficient conditions under appropriate calibrations of $\theta$, $\eta$ and $\epsilon$ such that $R_n(\theta,\eta,\epsilon)\rightarrow 0$ as $n\rightarrow\infty$.

An intuitive approach to testing \eqref{eq:problem} is to first estimate $z$ and $\sigma$ with $\wh{z}$ and $\wh{\sigma}$ by separately clustering $\{X_i\}_{1\le i\leq n}$ and $\{Y_i\}_{1\le i\leq n}$, 
and then one could reject the null hypothesis when $\ell(\wh{z},\wh{\sigma})>\epsilon/2$. 
With the known minimax optimal estimation error rate for $\wh{z}$ and $\wh{\sigma}$ in (\ref{eq:yu-lu-diao}), one can show that such a test is consistent as long as
\begin{equation}
\|\theta\|^2\wedge\|\eta\|^2>2\log\left(\frac{1}{\epsilon}\right). \label{eq:very-very-bad}
\end{equation}
However, we shall show that condition \eqref{eq:very-very-bad} is not optimal 
and that the required signal-to-noise ratio (SNR) for an optimal test to be consistent is much weaker.

Another natural test for (\ref{eq:problem}) can be based on a reduction to a well studied sparse signal detection problem. 
For convenience of discussion, let us suppose $\ell(z,\sigma)=\frac{1}{n}\sum_{i=1}^n\indc{z_i\neq \sigma_i}$ so that there is no ambiguity due to label switching. 
Note that 
$$D(X_i,Y_i)=\frac{(\|\eta\|/\|\theta\|)\theta^TX_i-(\|\theta\|/\|\eta\|)\eta^TY_i}{\sqrt{\|\theta\|^2+\|\eta\|^2}}\sim N\left(\frac{\|\theta\|\|\eta\|(z_i-\sigma_i)}{\sqrt{\|\theta\|^2+\|\eta\|^2}},1\right).$$
Then, we have $D(X_i,Y_i)\sim N(0,1)$ for all $i\in[n]$ under the null hypothesis and there are at least an $\epsilon$ fraction of coordinates distributed by $N\left(\pm\frac{2\|\theta\|\|\eta\|}{\sqrt{\|\theta\|^2+\|\eta\|^2}},1\right)$ under the alternative hypothesis. 
This setting is the sparse signal detection problem studied by \cite{ingster1997some,ingster2012nonparametric,donoho2004higher} under the asymptotic setting of $\epsilon=n^{-\beta}$ and $\frac{2\|\theta\|\|\eta\|}{\sqrt{\|\theta\|^2+\|\eta\|^2}}=\sqrt{2r\log n}$ with some constants $\beta,r>0$. 
It was shown in \cite{donoho2004higher} that the higher criticism test is consistent as long as
\begin{equation}
r > \begin{cases}
\beta-\frac{1}{2}, & \frac{1}{2}<\beta\leq \frac{3}{4}, \\
(1-\sqrt{1-\beta})^2, & \frac{3}{4} < \beta <1.
\end{cases} \label{eq:r-beta-IDJ}
\end{equation}
Moreover, the condition (\ref{eq:r-beta-IDJ}) cannot be improved if only the sequence $\{D(X_i,Y_i)\}_{1\leq i\leq n}$ is observed \citep{ingster1997some}. It can be checked that the condition (\ref{eq:r-beta-IDJ}) is always weaker than (\ref{eq:very-very-bad}) that is required by the test based on estimating $z$ and $\sigma$ first. 
However, we shall show later that one loses information by working only with the sequence $\{D(X_i,Y_i)\}_{1\leq i\leq n}$ for testing (\ref{eq:problem}), and we can further improve condition (\ref{eq:r-beta-IDJ}).

The main result of the present paper is an entire phase diagram of the testing problem (\ref{eq:problem}) under a natural asymptotic setting 
comparable to that used in \cite{donoho2004higher}. 
It turns out this seemingly simple testing problem (\ref{eq:problem}) 
has a complicated phase diagram parametrized by three parameters. 
The detection boundary of the diagram is characterized by five different functions over five disjoint regions in the space of signal-to-noise ratios. 
We also derive an asymptotically optimal test that achieves the detection boundary adaptively. 
At the heart of our construction of the optimal test is a precise likelihood ratio approximation. 
This leads to a sequence of asymptotically sufficient statistics, based on which a higher criticism type test can be proved to be optimal.

\paragraph{Related works}
In addition to the literature on integrative clustering that we have previously mentioned,
the testing problem \eqref{eq:problem} is related to feature selection in clustering analysis. 
In the literature, this has mainly been investigated in the context of sparse clustering \citep{azizyan2013minimax,jin2016influential,jin2017phase,verzelen2017detection,cai2019chime}, where it is assumed that only a small subset of covariates are useful for finding clusters, and so it is important to identify them. 
In comparison, the testing problem (\ref{eq:problem}) can be interpreted as testing whether inclusion of an additional set of covariates $\{Y_i\}_{1\leq i\leq n}$ can potentially lead to smaller clustering errors than using $\{X_i\}_{1\leq i\leq n}$ alone. 
The major difference is that the additional set of covariates may admit a completely different clustering structure in our setting, while in sparse clustering, covariates that are not useful have no clustering structure.

In addition to testing whether clustering structures in multiple datasets are equal, it is of interest to approach the problem from the opposite direction. In other word, one could also test whether the clustering structures share anything in common.
We refer the readers to \cite{gao2019clusterings} and \cite{gao2019testing} for studies along this line.

\paragraph{Paper organization} 
The rest of the paper is organized as the following. 
Section \ref{sec:balanced} studies (\ref{eq:problem}) with an additional equal SNR assumption. 
This simplified setting demonstrates essence of the problem while reducing a lot of technicalities.
The general version of the problem  {without equal SNR assumption} is studied in Section \ref{sec:general}. 
In Section \ref{sec:equality}, we consider (\ref{eq:problem}) with $\epsilon=0$, which is testing for exact equality. 
Optimal adaptive tests with unknown parameters for both $\epsilon > 0$ and $\epsilon = 0$ are discussed in Section \ref{sec:adaptive-t}. 
Finally, technical proofs are given in Section \ref{sec:all-pf}.

\paragraph{Notation}

For $d \in \mathbb{N}$, we write $[d] = \{1,\dotsc,d\}$.  Given $a,b\in\mathbb{R}$, we write $a\vee b=\max(a,b)$, $a\wedge b=\min(a,b)$ and $a_+ = \max(a,0)$.  For two positive sequences $\{a_n\}$ and $\{b_n\}$, we write $a_n\lesssim b_n$ to mean that there exists a constant $C > 0$ independent of $n$ such that $a_n\leq Cb_n$ for all $n$. 
Moreover, $a_n \asymp b_n$ means $a_n\lesssim b_n$ and $b_n\lesssim a_n$.  For a set $S$, we use $\indc{S}$ and $|S|$ to denote its indicator function and cardinality respectively. 
For any matrix $A$, $A^T$ stands for its transpose.
Any vector $v\in \mathbb{R}^d$ is by default a $d\times 1$ matrix.
For a vector $v = (v_1,\ldots,v_d)^T \in\mathbb{R}^d$, we define $\norm{v}^2=\sum_{\ell=1}^d v_\ell^2$.  The trace inner product between two matrices $A,B\in\mathbb{R}^{d_1\times d_2}$ is defined as $\iprod{A}{B} =\sum_{\ell=1}^{d_1}\sum_{\ell'=1}^{d_2}A_{\ell \ell'}B_{\ell \ell'}$, while the Frobenius and operator norms of $A$ are given by $\fnorm{A}=\sqrt{\iprod{A}{A}}$ and $\opnorm{A}=s_{\max}(A)$ respectively, where $s_{\max}(\cdot)$ denotes the largest singular value.  The notation $\mathbb{P}$ and $\mathbb{E}$ are generic probability and expectation operators whose distribution is determined by the context.

\section{ {Testing with Equal Signal-to-Noise Ratios}}
\label{sec:balanced}

Recall that we have two independent datasets $X_i\stackrel{ind}{\sim} N(z_i\theta,I_p)$ and $Y_i \stackrel{ind}{\sim} N(\sigma_i\eta,I_q)$ for $i\in[n]$. 
In this section, we {first} assume that  {SNRs} of the two datasets are  {equal}. In other words, $\|\theta\|=\|\eta\|$. 
The general case of potentially unequal SNRs
is more complicated and will be studied in Section \ref{sec:general}. 

First, we show that we can apply dimension reduction to both datasets without losing any information for testing \eqref{eq:problem}. 
Consider $\{X_i\}_{1\leq i\leq n}$.
Since the clustering structure only appears in the direction of $\theta$, 
we can project all $X_i$'s to the one-dimensional subspace spanned by 
the unit vector $\theta/\|\theta\|$. 
After projection, we obtain $\theta^TX_i/\|\theta\|\sim N(z_i\|\theta\|,1)$ for $i\in[n]$. 
Moreover, for any vector $u\in\mathbb{R}^p$ such that $u^T\theta=0$, 
we have $u^TX_i\sim N(0,1)$ for $i\in[n]$. 
Therefore, we conclude that the projected dataset 
$\{\theta^T X_i/\|\theta\|\}_{1\leq i\leq n}$ preserves all clustering information. 
The same argument also applies to $\{Y_i\}_{1\leq i\leq n}$. 
In the rest of this section, we write
\begin{equation}
	\label{eq:proj-1d}
	 \wt{X}_i=\theta^TX_i/\|\theta\|\quad 
	 \mbox{and} \quad 
	 \wt{Y}_i=\eta^TY_i/\|\eta\|
\end{equation}
for $i\in[n]$ and work with these one-dimensional random variables when constructing tests. 
On the other hand, we shall establish lower bounds of the testing problem directly in the original multi-dimensional setting.

\subsection{A Connection to Sparse Signal Detection}

\paragraph{A related sparse mixture detection problem}
For simplicity, let us suppose for now that $\frac{1}{n}\sum_{i=1}^n\indc{z_i\neq \sigma_i}\leq \frac{1}{2}$, so that $\ell(z,\sigma)=\frac{1}{n}\sum_{i=1}^n\indc{z_i\neq \sigma_i}$. 
Under  {$H_0$ in \eqref{eq:problem}}, we have both $\wt{X}_i\sim N(z_i\|\theta\|,1)$ and $\wt{Y}_i\sim N(z_i\|\theta\|,1)$ for all $i\in[n]$, which motivates us to compute  {scaled differences} $(\wt{X}_i-\wt{Y}_i)/{\sqrt{2}}$, $i\in [n]$. 

Note that the distributions of $\{(\wt{X}_i-\wt{Y}_i)/\sqrt{2}\}_{1\leq i\leq n}$  {under $H_0$ and under $H_1$ in \eqref{eq:problem}} are the same as  {those in} a sparse signal detection problem. 
Indeed, $(\wt{X}_i-\wt{Y}_i)/\sqrt{2} \stackrel{iid}{\sim} N(0,1)$ for $i\in[n]$ under the null, and at least an $\epsilon$ fraction of the statistics follow either $N(\sqrt{2}\|\theta\|,1)$ or $N(-\sqrt{2}\|\theta\|,1)$ under the alternative. 
A well studied  {Bayesian} version of the sparse signal detection problem is given by the following form:
\begin{eqnarray}
\label{eq:equal-sparse-null} H_0: && U_1,\cdots,U_n\stackrel{iid}{\sim} N(0,1), \\
\label{eq:equal-sparse-alt} H_1: && U_1,\cdots,U_n\stackrel{iid}{\sim} (1-\epsilon)N(0,1)+\frac{\epsilon}{2}N(-\sqrt{2}\|\theta\|,1)+\frac{\epsilon}{2}N(\sqrt{2}\|\theta\|,1).
\end{eqnarray}
 {In what follows, we refers to \eqref{eq:equal-sparse-null}--\eqref{eq:equal-sparse-alt} (and any such Bayesian version of the problem) as a sparse mixture detection problem.}
There are two  {noticeable} differences between (\ref{eq:equal-sparse-alt}) and the distribution of $\{(\wt{X}_i-\wt{Y}_i)/\sqrt{2}\}_{1\leq i\leq n}$ under  {$H_1$ in \eqref{eq:problem}}:
\begin{enumerate}
	\item The number of  {non-null} signals in (\ref{eq:equal-sparse-alt}) is a binomial random variable  {while it is deterministic in \eqref{eq:problem}};
	
	\item The probabilities that a  {non-null} signal is from $N(\sqrt{2}\|\theta\|,1)$ and from $N(-\sqrt{2}\|\theta\|,1)$ are equal  {in \eqref{eq:equal-sparse-alt} while there is no restriction on how many non-null signals follow either of the two distributions in \eqref{eq:problem}.}
\end{enumerate}
However, 
these differences are inconsequential  {as long as our focus is on the phase diagrams of these testing problems with the calibration we now introduce}.

For either the hypothesis testing problem (\ref{eq:equal-sparse-null})-(\ref{eq:equal-sparse-alt})  {or \eqref{eq:problem} with $\|\theta\| = \|\eta\|$},
 {introduce the calibration}
\begin{equation}
\epsilon=n^{-\beta}\quad \mbox{and} \quad  {\sqrt{2}}\,\|\theta\|=\sqrt{ {2} r\log n}. \label{eq:equal-cali}
\end{equation}
 {For \eqref{eq:equal-sparse-null}-\eqref{eq:equal-sparse-alt}}\footnote{The non-Bayesian version of the problem has also been studied in \cite{ingster1997some,ingster2012nonparametric}.},
it was proved in \cite{ingster1997some,ingster2012nonparametric} for that the likelihood ratio test is consistent 
when $\beta<\beta_{\rm IDJ}^*(r)$\footnote{Following \cite{cai2014optimal}, we call $\beta_{\rm IDJ}^*(r)$ the Ingster--Donoho--Jin threshold.} and no test is consistent when $\beta>\beta_{\rm IDJ}^*(r)$,  where the threshold function is 
\begin{equation}
\beta_{\rm IDJ}^*(r)=\begin{cases}
\frac{1}{2}+r, & 0<r\leq \frac{1}{4}, \\ 
1-(1-\sqrt{r})_+^2, & r>\frac{1}{4}.
\end{cases}\label{eq:equal-thresh}
\end{equation}
Note that $\beta<\beta^*_{\rm IDJ}(r)$ is equivalent to (\ref{eq:r-beta-IDJ}).
Moreover, Donoho and Jin \citep{donoho2004higher} proposed a higher-criticism (HC) test that rejects $H_0$ when
$$\sup_{t>0}\frac{\left|\sum_{i=1}^n\indc{|U_i|^2> t}-n\mathbb{P}(\chi_1^2>t)\right|}{\sqrt{n\mathbb{P}(\chi_1^2>t)(1-\mathbb{P}(\chi_1^2>t))}}>\sqrt{2(1+\delta)\log\log n},$$
where $\chi_m^2$ denotes a chi-square distribution with $m$ degrees of freedom and $\delta>0$ is some arbitrary fixed constant.
They proved that the HC test adaptively achieves consistency when $\beta<\beta_{\rm IDJ}^*(r)$. 
We refer interested readers to \cite{donoho2015higher,jin2016rare} for more discussions on HC tests.

\paragraph{ {Result for testing equivalence of clustering}}
 {Turn to} (\ref{eq:problem}). 
We need to slightly modify the HC test to accommodate the possibility of label switching in the clustering context. Define
\begin{eqnarray*}
T_n^- &=& \sup_{t>0}\frac{\left|\sum_{i=1}^n\indc{|\wt{X}_i-\wt{Y}_i|^2/2> t}-n\mathbb{P}(\chi_1^2>t)\right|}{\sqrt{n\mathbb{P}(\chi_1^2>t)(1-\mathbb{P}(\chi_1^2>t))}}, \\
T_n^+ &=& \sup_{t>0}\frac{\left|\sum_{i=1}^n\indc{|\wt{X}_i+\wt{Y}_i|^2/2> t}-n\mathbb{P}(\chi_1^2>t)\right|}{\sqrt{n\mathbb{P}(\chi_1^2>t)(1-\mathbb{P}(\chi_1^2>t))}}.
\end{eqnarray*}
Based on these two statistics, we define
\begin{equation}
\psi=\indc{T_n^-\wedge T_n^+>\sqrt{2(1+\delta)\log\log n}}, \label{eq:equal-test-diff}
\end{equation}
where $\delta>0$ is an arbitrary fixed constant. 
 {Taking the minimum of $T_n^+$ and $T_n^-$ makes the test invariant to label switching.}

\begin{proposition}\label{prop:equal-diff}
 {For testing \eqref{eq:problem}}
with the assumption that $\|\theta\|=\|\eta\|$ and the calibration in (\ref{eq:equal-cali}), 
the test (\ref{eq:equal-test-diff}) satisfies $\lim_{n\rightarrow\infty}R_n(\psi,\theta,\eta,\epsilon)=0$ as long as $\beta<\beta_{\rm IDJ}^*(r)$.
\end{proposition}

Proposition \ref{prop:equal-diff} shows that the test (\ref{eq:equal-test-diff}) consistently distinguishes two clustering structures under the same condition that implies  {consistency in
the sparse mixture detection problem \eqref{eq:equal-sparse-null}-\eqref{eq:equal-sparse-alt}}.
 {This being said, it is not clear at this point whether $\beta^*_{\rm IDJ}(r)$ is the detection boundary for \eqref{eq:problem} under the equal SNR assumption and the calibration \eqref{eq:equal-cali}, which, if were true, would require that no consistent test exists when $\beta > \beta^*_{\rm IDJ}(r)$.
}

\begin{remark}\label{rem:est}
Another straightforward way to testing (\ref{eq:problem}) is to first estimate $z$ and $\sigma$ and then reject $H_0$ if the two estimators are not sufficiently close. 
Let $\wh{z}$ and $\wh{\sigma}$ be minimax optimal estimators of $z$ and $\sigma$ that satisfy the error bounds (\ref{eq:yu-lu-diao}). 
A natural test is then $\psi_{\rm estimation}=\indc{\ell(\wh{z},\wh{\sigma})>\epsilon/2}$. 
It can be shown that $\lim_{n\rightarrow\infty}R_n(\psi_{\rm estimation},\theta,\eta)=0$ when $\beta<r/2$ under the calibration  {(\ref{eq:equal-cali})}. 
Compared with the condition $\beta<\beta_{\rm IDJ}^*(r)$ required by the test (\ref{eq:equal-test-diff}), $\psi_{\rm estimation}$ 
requires a stronger SNR to achieve consistency.
\end{remark}

\subsection{The Lost Information}

The natural  {follow-up} question is whether the condition $\beta<\beta_{\rm IDJ}^*(r)$ in Proposition \ref{prop:equal-diff} is 
necessary for consistently testing (\ref{eq:problem})  {with the equal SNR assumption and the calibration \eqref{eq:equal-cali}}.
In order to address this lower bound question, let us continue  {to suppose} $\frac{1}{n}\sum_{i=1}^n\indc{z_i\neq \sigma_i}\leq \frac{1}{2}$ so that we ignore label switching temporarily. 
 {A key observation is that} by reducing the data from $(\wt{X}_i,\wt{Y}_i)$ to $(\wt{X}_i-\wt{Y}_i)/ {\sqrt{2}} $, we  {have thrown} away all the information in $(\wt{X}_i+\wt{Y}_i) / {\sqrt{2}}$. 
 {Therefore, we now study the sequence $\{(\wt{X}_i+\wt{Y}_i)/\sqrt{2}\}_{1\leq i\leq n}$.}

We note that whether $z_i=\sigma_i$ not only changes the distribution of $(\wt{X}_i-\wt{Y}_i) / {\sqrt{2}}$, but also the distribution of $(\wt{X}_i+\wt{Y}_i) / {\sqrt{2}}$. 
In fact, we have
$$\frac{1}{\sqrt{2}}(\wt{X}_i+\wt{Y}_i)\sim \begin{cases}
N(\pm\sqrt{2}\|\theta\|,1), & z_i=\sigma_i, \\
N(0,1), & z_i\neq \sigma_i.
\end{cases}$$
Since there is at least  {an} $\epsilon$ fraction of clustering labels that do not match, 
 {a natural corresponding sparse mixture detection problem is the following:}
\begin{eqnarray}
\label{eq:equal-sum-null} H_0: && V_1,\cdots,V_n\stackrel{iid}{\sim} \frac{1}{2}N(-\sqrt{2}\|\theta\|,1)+\frac{1}{2}N(\sqrt{2}\|\theta\|,1), \\
\label{eq:equal-sum-alt} H_1: && V_1,\cdots,V_n\stackrel{iid}{\sim} \frac{1-\epsilon}{2}N(-\sqrt{2}\|\theta\|,1)+\frac{1-\epsilon}{2}N(\sqrt{2}\|\theta\|,1) + \epsilon N(0,1).
\end{eqnarray}
Compared with (\ref{eq:equal-sparse-null})-(\ref{eq:equal-sparse-alt}), the roles of $N(0,1)$ and $\frac{1}{2}N(-\sqrt{2}\|\theta\|,1)+\frac{1}{2}N(\sqrt{2}\|\theta\|,1)$ are switched in (\ref{eq:equal-sum-null})-(\ref{eq:equal-sum-alt}). 
 {To our limited knowledge,} the testing problem (\ref{eq:equal-sum-null})-(\ref{eq:equal-sum-alt}) has not been studied in the literature before. With the same calibration (\ref{eq:equal-cali}), its fundamental limit is given by the following theorem.

\begin{thm}\label{thm:V-equal}
Consider testing (\ref{eq:equal-sum-null})-(\ref{eq:equal-sum-alt}) with calibration (\ref{eq:equal-cali}). 
Define
\begin{equation}
\bar{\beta}^*(r)=1\wedge\frac{r+1}{2}. \label{eq:threshold-V-equal}
\end{equation}
When $\beta<\bar{\beta}^*(r)$, the likelihood ratio test is consistent. When $\beta>\bar{\beta}^*(r)$, no test is consistent.
\end{thm}

Theorem \ref{thm:V-equal} shows that the optimal threshold  {(in terms of the calibration \eqref{eq:equal-cali})} for the testing problem (\ref{eq:equal-sum-null})-(\ref{eq:equal-sum-alt}) is $\bar{\beta}^*(r)$. 
It is easy to check that
$$
\bar{\beta}^*(r)\leq\beta_{\rm IDJ}^*(r),
\qquad \mbox{for all $r>0$}.
$$
This  {indicates that} the sequence $\{(\wt{X}_i+\wt{Y}_i)/\sqrt{2}\}_{1\leq i\leq n}$ does contain information, but not as much as that in $\{(\wt{X}_i-\wt{Y}_i)/\sqrt{2}\}_{1\leq i\leq n}$. 
Similar to (\ref{eq:equal-test-diff}), we can also design an HC-type test  {as motivated by} (\ref{eq:equal-sum-null})-(\ref{eq:equal-sum-alt}). 
Define
\begin{eqnarray}
\nonumber \bar{T}_n^+ &=& \sup_{t>0}\frac{\left|\sum_{i=1}^n\indc{(\wt{X}_i+\wt{Y}_i)^2/2\leq t}-\mathbb{P}(\chi_{1,2\|\theta\|^2}^2\leq t)\right|}{\sqrt{n\mathbb{P}(\chi_{1,2\|\theta\|^2}^2\leq t)(1-\mathbb{P}(\chi_{1,2\|\theta\|^2}^2\leq t))}}, \\
\nonumber  \bar{T}_n^- &=& \sup_{t>0}\frac{\left|\sum_{i=1}^n\indc{(\wt{X}_i-\wt{Y}_i)^2/2\leq t}-\mathbb{P}(\chi_{1,2\|\theta\|^2}^2\leq t)\right|}{\sqrt{n\mathbb{P}(\chi_{1,2\|\theta\|^2}^2\leq t)(1-\mathbb{P}(\chi_{1,2\|\theta\|^2}^2\leq t) )}}.
\end{eqnarray}
In addition to $\bar{T}_n^+$, we need $\bar{T}_n^-$ to accommodate the possibility of $\frac{1}{n}\sum_{i=1}^n\indc{z_i\neq \sigma_i}> \frac{1}{2}$. The overall test for our original problem is then
\begin{equation}
\bar{\psi}=\indc{\bar{T}_n^-\wedge \bar{T}_n^+>\sqrt{2(1+\delta)\log\log n}}, \label{eq:equal-test-sum}
\end{equation}
where $\delta>0$ is an arbitrary fixed constant. 

\begin{thm}\label{thm:equal-sum}
 {For testing \eqref{eq:problem}}
with the assumption that $\|\theta\|=\|\eta\|$ and the calibration in (\ref{eq:equal-cali}), 
the test (\ref{eq:equal-test-sum}) satisfies $\lim_{n\rightarrow\infty}R_n(\bar{\psi},\theta,\eta,\epsilon)=0$ as long as $\beta<\bar{\beta}^*(r)$.
\end{thm}

\subsection{Combining the Two Views}\label{sec:combine-equal}

Proposition \ref{prop:equal-diff} and Theorem \ref{thm:equal-sum} show that the original testing problem (\ref{eq:problem})  {is connected} to  {both} (\ref{eq:equal-sparse-null})-(\ref{eq:equal-sparse-alt})
 {and} (\ref{eq:equal-sum-null})-(\ref{eq:equal-sum-alt}). 
 {The two views are complementary.}
Both are  {non-trivial} and lead to tests  {for the original problem \eqref{eq:problem}} that achieve consistency under appropriate conditions. 
However, to achieve optimality in the original testing problem (\ref{eq:problem})  {under equal SNR assumption with the calibration \eqref{eq:equal-cali}}, we need to combine the two views.
 {In what follows, we first explain how this can be done in sparse mixture detection. 
This is followed by our main result for testing equivalence of clustering as in \eqref{eq:problem} with equal SNR assumption.}
Interestingly, \citet{tony2019covariate} discovered a similar phenomenon that one achieves additional power by using a complementary sequence in a different context, namely two sample multiple testing.

\paragraph{Sparse mixture detection}
We now study the combination of the two views (\ref{eq:equal-sparse-null})-(\ref{eq:equal-sparse-alt}) and (\ref{eq:equal-sum-null})-(\ref{eq:equal-sum-alt}), which can be formulated  {as testing}
\begin{align}
\label{eq:equal-comb-null} H_0:~& (U_i,V_i)\stackrel{iid}{\sim} \frac{1}{2}N(0,1)\otimes N(-\sqrt{2}\|\theta\|,1)+\frac{1}{2}N(0,1)\otimes N(\sqrt{2}\|\theta\|,1),\,\, i\in[n], \,\, \mbox{vs.} \\
\label{eq:equal-comb-alt} H_1:~& (U_i,V_i)\stackrel{iid}{\sim} \frac{1-\epsilon}{2}N(0,1)\otimes N(-\sqrt{2}\|\theta\|,1)+\frac{1-\epsilon}{2}N(0,1)\otimes N(\sqrt{2}\|\theta\|,1) \\
\nonumber & \quad\quad\qquad+\frac{\epsilon}{2}N(-\sqrt{2}\|\theta\|,1)\otimes N(0,1)+\frac{\epsilon}{2}N(\sqrt{2}\|\theta\|,1)\otimes N(0,1),\,\, i\in[n].
\end{align}
 {The 
critical values \eqref{eq:equal-thresh} and \eqref{eq:threshold-V-equal}
can now be viewed as detection boundaries for testing \eqref{eq:equal-comb-null}-\eqref{eq:equal-comb-alt} when only $\{U_i\}_{1\leq i\in n}$ and $\{V_i\}_{1\leq i\in n}$ are used, respectively.}
The two components $U_i$ and $V_i$ behave very differently under  {null and alternative}. 
The value of $|U_i|$  {tends to be smaller} under $H_0$ and  {larger} under $H_1$, while the value of $|V_i|$  {behaves in the opposite way}.
 {This motivates us to} combine the two pieces of information by  {working with} $|U_i|-|V_i|$, which  {tends to be smaller} under $H_0$ and  {larger} under $H_1$.
Since there is  {on average} an $\epsilon$ fraction of  {non-nulls} under $H_1$, we  {may} reject $H_0$  {if} $\sum_{i=1}^n\indc{|U_i|-|V_i|>t}$ is above some threshold. 
This intuition motivates us to consider the following HC-type test. 
Define the survival function
$$S_{\|\theta\|}(t)=\mathbb{P}_{(U^2,V^2)\sim \chi_1^2\otimes \chi_{1,2\|\theta\|^2}^2}\left(|U|-|V|>t\right).$$
We reject $H_0$ when
\begin{equation}
\sup_{t\in\mathbb{R}}\frac{\left|\sum_{i=1}^n\indc{|U_i|-|V_i|>t}-nS_{\|\theta\|}(t)\right|}{\sqrt{nS_{\|\theta\|}(t)(1-S_{\|\theta\|}(t))}}>\sqrt{2(1+\delta)\log\log n},\label{eq:HC-comb}
\end{equation}
where $\delta>0$ is an arbitrary fixed constant.

\begin{thm}\label{thm:U-V-equal}
Consider testing (\ref{eq:equal-comb-null})-(\ref{eq:equal-comb-alt}) with calibration (\ref{eq:equal-cali}). 
Define
\begin{equation}
\beta^*(r)=\begin{cases}
\frac{1}{2}({1+3r}), & 0<r\leq \frac{1}{5}, \\
\sqrt{1-(1-2r)_+^2}, & r>\frac{1}{5}.
\end{cases}\label{eq:freestyle}
\end{equation}
When $\beta<\beta^*(r)$, the likelihood ratio test and the HC-type test (\ref{eq:HC-comb}) are consistent. When $\beta>\beta^*(r)$, no test is consistent.
\end{thm}

We plot the three threshold functions (a.k.a.~detection boundaries) {$\bar{\beta}^*(r)$  {(red)}, $\beta_{\rm IDJ}^*(r)$  {(orange)} and $\beta^*(r)$  {(blue)}} in Figure \ref{fig:phase-equal}.
\begin{figure}[!tb]
\begin{center}
\includegraphics[width=0.6\textwidth]{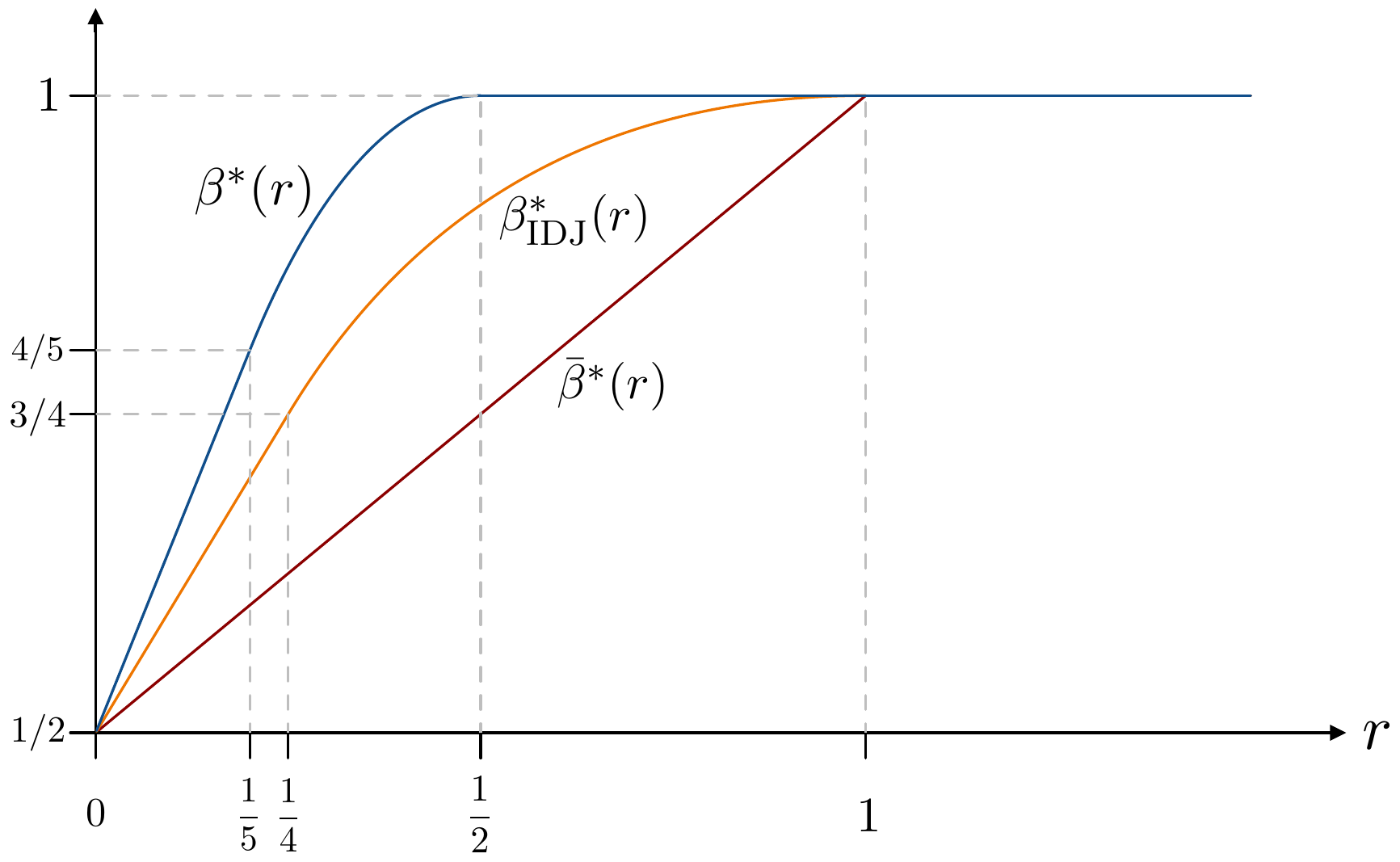}
\caption{ {Comparison of three detection boundaries.}
}
\label{fig:phase-equal}
\end{center}
\end{figure}
Since $\bar{\beta}^*(r)\leq\beta_{\rm IDJ}^*(r)\leq\beta^*(r)$ for all $r>0$, 
 {in view of the discussion following (\ref{eq:equal-comb-null})-(\ref{eq:equal-comb-alt}) we can conclude that pooling information in $\{U_i\}_{1\leq i\in n}$ and $\{V_i\}_{1\leq i\in n}$ leads to a more powerful test than using either single sequence.}

\paragraph{Testing equivalence of clustering}
 {We are now in a position to show that $\beta^*(r)$ in \eqref{eq:freestyle} is also the detection boundary for testing \eqref{eq:problem} under the equal SNR assumption and the calibration \eqref{eq:equal-cali}.}
 {Motivated by \eqref{eq:HC-comb} and taking into account possible label switching, we define}
\begin{eqnarray*}
\check{T}_n^- &=& \sup_{t\in\mathbb{R}}\frac{\left|\sum_{i=1}^n\indc{|\wt{X}_i-\wt{Y}_i|-|\wt{X}_i+\wt{Y}_i|>\sqrt{2}t}-nS_{\|\theta\|}(t)\right|}{\sqrt{nS_{\|\theta\|}(t)(1-S_{\|\theta\|}(t))}}, \\
\check{T}_n^+ &=& \sup_{t\in\mathbb{R}}\frac{\left|\sum_{i=1}^n\indc{|\wt{X}_i+\wt{Y}_i|-|\wt{X}_i-\wt{Y}_i|>\sqrt{2}t}-nS_{\|\theta\|}(t)\right|}{\sqrt{nS_{\|\theta\|}(t)(1-S_{\|\theta\|}(t))}},
\end{eqnarray*}
and
\begin{equation}
\check{\psi}=\indc{\check{T}_n^-\wedge \check{T}_n^+>\sqrt{2(1+\delta)\log\log n}}, \label{eq:equal-test-comb}
\end{equation}
where $\delta>0$ is an arbitrary fixed constant.
\begin{thm}\label{thm:main-equal}
 {For testing \eqref{eq:problem}}
with the assumption that $\|\theta\|=\|\eta\|$ and the calibration in (\ref{eq:equal-cali}), 
the test (\ref{eq:equal-test-comb}) satisfies $\lim_{n\rightarrow\infty}R_n(\check{\psi},\theta,\eta,\epsilon)=0$ as long as $\beta<{\beta}^*(r)$. Moreover, we have $\liminf_{n\rightarrow\infty}R_n(\theta,\eta,\epsilon)>0$,  {that is no test is consistent}, when $\beta>{\beta}^*(r)$.
\end{thm}

 {We conclude this section with} three remarks on Theorem \ref{thm:main-equal}. 
First, the theorem shows that the two-dimensional sparse mixture testing problem (\ref{eq:equal-comb-null})-(\ref{eq:equal-comb-alt}) contains the mathematical essence of the original  {testing equivalence of clustering} problem (\ref{eq:problem}), because  {they share the same detection boundary}. 
 {In addition}, it shows that either the view of (\ref{eq:equal-sparse-null})-(\ref{eq:equal-sparse-alt}) or (\ref{eq:equal-sum-null})-(\ref{eq:equal-sum-alt}) results in a suboptimal solution (see Figure \ref{fig:phase-equal}). 
The testing problem (\ref{eq:problem}) is fundamentally different from the sparse mixture detection problem  {(\ref{eq:equal-sparse-null})-(\ref{eq:equal-sparse-alt})} that  {has been} well studied in the literature. 
 {Furthermore}, it  {suffices} to work with the one-dimensional  {projected datasets} $\{(\theta^TX_i/\|\theta\|,\eta^TY_i/\|\eta\|)\}_{1\leq i\leq n}$  {when constructing tests}, as  {the upper and the lower bounds match} in Theorem \ref{thm:main-equal}.

\section{The General Phase Diagram}
\label{sec:general}

In this section, we study the general case  {testing \eqref{eq:problem}} where $\|\theta\|$ and $\|\eta\|$ are not necessarily equal. 
This is a more complicated problem than the  {equal SNR} case studied in Section \ref{sec:balanced}. 
 {For the general case,} we adopt the following calibration:
\begin{equation}
\epsilon=n^{-\beta},\quad \frac{2\|\theta\|\|\eta\|}{\sqrt{\|\theta\|^2+\|\eta\|^2}} = \sqrt{2r\log n}, \quad \frac{|\|\theta\|^2-\|\eta\|^2|}{\sqrt{\|\theta\|^2+\|\eta\|^2}} = \sqrt{2s\log n}. \label{eq:general-cali}
\end{equation}
 {With this calibration, $(r,s)$ can take any value in $(0,\infty)\times [0,\infty)$.}
 {Although} there are other ways to parametrize $\|\theta\|$ and $\|\eta\|$,  {we find (\ref{eq:general-cali}) convenient and interpretable}. 
In \eqref{eq:general-cali}, $r$  {characterizes} overall signal strength and $s$ {quantifies the level of difference in SNRs  {of the two samples}}.
When $s=0$, (\ref{eq:general-cali})  {reduces to} (\ref{eq:equal-cali}). 
 {With this natural reduction, all results in Section \ref{sec:balanced} can be obtained by setting $s = 0$ in results for the general case which we shall derive in this section.}
Furthermore, the following expressions can be derived from (\ref{eq:general-cali}):
\begin{eqnarray}
 {\|\theta\|^2+\|\eta\|^2} &=&  {2(r+s)\log n}, \\
\|\theta\|^2\vee\|\eta\|^2 &=& \left(r+s+\sqrt{s}\sqrt{r+s}\right)\log n, \\
\label{eq:min-sig} \|\theta\|^2\wedge\|\eta\|^2 &=& \left(r+s-\sqrt{s}\sqrt{r+s}\right)\log n.
\end{eqnarray}

\subsection{A Related Sparse Mixture Detection Problem}

With $\wt{X}_i\sim N(z_i\|\theta\|,1)$ and $\wt{Y}_i\sim N(\sigma_i\|\eta\|,1)$  {as defined in \eqref{eq:proj-1d}}, it is natural to consider
\begin{equation}
\frac{\|\eta\|\wt{X}_i-\|\theta\|\wt{Y}_i}{\sqrt{\|\theta\|^2+\|\eta\|^2}}\sim N\left(\frac{\|\theta\|\|\eta\|(z_i-\sigma_i)}{\sqrt{\|\theta\|^2+\|\eta\|^2}},1\right). \label{eq:seq1}
\end{equation}
Moreover, to  {avoid} information loss, we also consider the following complementary sequence to (\ref{eq:seq1}),
\begin{equation}
\frac{\|\theta\|\wt{X}_i+\|\eta\|\wt{Y}_i}{\sqrt{\|\theta\|^2+\|\eta\|^2}}\sim N\left(\frac{\|\theta\|^2z_i+\|\eta\|^2\sigma_i}{\sqrt{\|\theta\|^2+\|\eta\|^2}},1\right). \label{eq:seq2}
\end{equation}
The sequences (\ref{eq:seq1}) and (\ref{eq:seq2}) are  {mutually} independent. 
Since $(\wt{X}_i,\wt{Y}_i)$  {and} $\left(\frac{\|\eta\|\wt{X}_i-\|\theta\|\wt{Y}_i}{\sqrt{\|\theta\|^2+\|\eta\|^2}},\frac{\|\theta\|\wt{X}_i+\|\eta\|\wt{Y}_i}{\sqrt{\|\theta\|^2+\|\eta\|^2}}\right)$  {have one-to-one correspondence}, there is no information loss.

Without loss of generality\footnote{We only use $\|\theta\|\geq\|\eta\|$ to motivate the testing problem (\ref{eq:general-comb-null})-(\ref{eq:general-comb-alt}). All the theorems in the paper hold with general $\theta$ and $\eta$ that admit the calibration (\ref{eq:general-cali}).},  {let us further} {assume} $\|\theta\|\geq\|\eta\|$.
We note that when $z_i=\sigma_i$, the two sequences have means $0$ and $\pm\sqrt{\|\theta\|^2+\|\eta\|^2}$, respectively. 
When $z_i\neq \sigma_i$,  {they} have means $\pm \frac{2\|\theta\|\|\eta\|}{\sqrt{\|\theta\|^2+\|\eta\|^2}}$ and $\pm \frac{|\|\theta\|^2-\|\eta\|^2|}{\sqrt{\|\theta\|^2+\|\eta\|^2}}$, respectively. 
Therefore, a  {natural corresponding} sparse mixture detection problem to (\ref{eq:problem}) is
\begin{eqnarray}
\label{eq:general-comb-null} H_0: \quad (U_i,V_i) &\stackrel{iid}{\sim}& \frac{1}{2}N(0,1)\otimes N(-\sqrt{2(r+s)\log n},1) \\
\nonumber && +\frac{1}{2}N(0,1)\otimes N(\sqrt{2(r+s)\log n},1),\quad i\in[n], \\
\label{eq:general-comb-alt} H_1: \quad (U_i,V_i) &\stackrel{iid}{\sim}& \frac{1-\epsilon}{2}N(0,1)\otimes N(-\sqrt{2(r+s)\log n},1) \\
\nonumber && +\frac{1-\epsilon}{2}N(0,1)\otimes N(\sqrt{2(r+s)\log n},1) \\
\nonumber && +\frac{\epsilon}{2}N(\sqrt{2r\log n},1)\otimes N(\sqrt{2s\log n},1) \\
\nonumber && +\frac{\epsilon}{2}N(-\sqrt{2r\log n},1)\otimes N(-\sqrt{2s\log n},1), \quad i\in[n].
\end{eqnarray}
When $s=0$, the testing problem (\ref{eq:general-comb-null})-(\ref{eq:general-comb-alt}) 
 {reduces to}
(\ref{eq:equal-comb-null})-(\ref{eq:equal-comb-alt}).

 {Similar to Section \ref{sec:balanced}, as a first step,}
we derive the detection boundaries of tests that only use $\{U_i\}_{1\leq i\leq n}$ or $\{V_i\}_{1\leq i\leq n}$.

\begin{thm}\label{thm:general-separate}
Consider testing (\ref{eq:general-comb-null})-(\ref{eq:general-comb-alt}) with $\epsilon=n^{-\beta}$. Define
$$\bar{\beta}^*(r,s)=\begin{cases}
\frac{1}{2}+r-2\sqrt{s}(\sqrt{r+s}-\sqrt{s}), & 3s>r {~and~} (\sqrt{r+s}-\sqrt{s})^2\leq\frac{1}{4}, \\
\frac{1+r-s}{2}, &  3s\leq r {~and~} r+s\leq 1, \\
r-2(\sqrt{r+s}-\sqrt{s})(\sqrt{r+s}-1), & r+s > 1 {~and~} \frac{1}{4}<(\sqrt{r+s}-\sqrt{s})^2\leq 1, \\
1, & (\sqrt{r+s}-\sqrt{s})^2 > 1.
\end{cases}$$
 {For any fixed constant $\delta>0$,}
we have the following two conclusions:
\begin{enumerate}
\item When $\beta<\beta^*_{\rm IDJ}(r)$, the test with rejection region
$$\sup_{t>0}\frac{\left|\sum_{i=1}^n\indc{|U_i|^2>t}-n\mathbb{P}(\chi_1^2>t)\right|}{\sqrt{n\mathbb{P}(\chi_1^2>t) (1-\mathbb{P}(\chi_1^2>t)  )}}>\sqrt{2(1+\delta)\log\log n}$$
is consistent. When $\beta>\beta^*_{\rm IDJ}(r)$, no test that only uses $\{U_i\}_{1\leq i\leq n}$ is consistent.
\item When $\beta<\bar{\beta}^*(r,s)$, the test with rejection region
$$\sup_{t>0}\frac{\left|\sum_{i=1}^n\indc{|V_i|^2\leq t}-n\mathbb{P}(\chi_{1,2(r+s)\log n}^2\leq t )\right|}{\sqrt{n\mathbb{P}(\chi_{1,2(r+s)\log n}^2\leq t ) (1-\mathbb{P}(\chi_{1,2(r+s)\log n}^2\leq t ) )}}>\sqrt{2(1+\delta)\log\log n}$$
is consistent. When $\beta>\bar{\beta}^*(r,s)$, no test that only uses $\{V_i\}_{1\leq i\leq n}$ is consistent.
\end{enumerate}
\end{thm}

The first conclusion of Theorem \ref{thm:general-separate} is obvious, since the marginal distributions of $\{U_i\}_{1\leq i\leq n}$ under  {\eqref{eq:general-comb-null} and \eqref{eq:general-comb-alt}} 
are exactly the same as  {those under (\ref{eq:equal-sparse-null}) and (\ref{eq:equal-sparse-alt}), respectively.} 
In contrast, the second conclusion shows an intricate behavior of the two-dimensional threshold function $\bar{\beta}^*(r,s)$. We note that $\bar{\beta}^*(r,s)$ can be viewed as an extension of $\bar{\beta}^*(r)$ defined in (\ref{eq:threshold-V-equal}) in the sense that 
 {setting $s = 0$ in $\bar{\beta}^*(r,0)$ gives \eqref{eq:threshold-V-equal}.} 
The definition of $\bar{\beta}^*(r,s)$ involves four  {disjoint regions in $(0,\infty) \times [0,\infty)$}. 
When $s=0$, the second and the third  {cases} become degenerate. 
Moreover, we also have the relation $\bar{\beta}^*(r,s)\leq \bar{\beta}^*(r)$ for all $r,s>0$, which suggests that the testing problem becomes harder as $\|\theta\|$ and $\|\eta\|$ 
 {become more different}.
 {Last but not least}, as $s\rightarrow\infty$, we have $\bar{\beta}^*(r,s)\rightarrow\frac{1}{2}$.

\subsection{Which Event Shall We Count?}\label{sec:which-count}

Now let us try to solve the testing problem (\ref{eq:general-comb-null})-(\ref{eq:general-comb-alt}) by considering both $\{U_i\}_{1\leq i\leq n}$ and $\{V_i\}_{1\leq i\leq n}$. 
In order to derive the sharp detection boundary of (\ref{eq:general-comb-null})-(\ref{eq:general-comb-alt}) and also of the original problem (\ref{eq:problem}), we need to first find the optimal testing statistic. 
 {By} Theorem \ref{thm:general-separate}, 
 {the detection boundary of either single sequence can be achieved by an appropriate HC-type test.}
 {For $\{U_i\}_{1\leq i\leq n}$ the test counts the number of large $|U_i|$'s by $\sum_{i=1}^n\indc{|U_i|^2>t}$, and for $\{V_i\}_{1\leq i\leq n}$ the corresponding test counts the number of small $|V_i|$'s by $\sum_{i=1}^n\indc{|V_i|^2\leq t}$.} 
 {These} tests suggest that for testing (\ref{eq:general-comb-null})-(\ref{eq:general-comb-alt}) we should count the event that either $|U_i|$ is large or $|V_i|$ is small. 
When the  {SNRs are equal},
we have used $\sum_{i=1}^n\indc{|U_i|-|V_i|>t}$ in Section \ref{sec:combine-equal}  {for this purpose}. 
However,  {such an event may no longer be} appropriate when $\|\theta\|\neq \|\eta\|$. 

In order to find out the  {appropriate} event to count, we present the following heuristic argument from a more general perspective. 
Let consider  {the following} abstract sparse mixture testing problem:
\begin{eqnarray}
\label{eq:abstract-null} H_0: && W_1,\cdots,W_n\stackrel{iid}{\sim} P, \quad \mbox{vs.}\\
\label{eq:abstract-alt} H_1: && W_1,\cdots,W_n\stackrel{iid}{\sim} (1-\epsilon)P+\epsilon Q,
\end{eqnarray}
where $\epsilon=n^{-\beta}$ for some constant $\beta\in(0,1)$.
Then, the general HC-type testing statistic can be written as
\begin{equation}
\sup_{A\in\mathcal{A}}\frac{\left|\sum_{i=1}^n\indc{W_i\in A}-nP(A)\right|}{\sqrt{nP(A)(1-P(A))}}, \label{eq:HC-abstract}
\end{equation}
where $\mathcal{A}$ is some collection of events. 
 {As we shall show,} the reason to take supreme over  {the collection} $\mathcal{A}$ is  {mostly} for the sake of adaptation. 
When  {one has knowledge of $P$ and $Q$,}
 {let us start with} the statistic
$$T_n(A)=\frac{\sum_{i=1}^n\indc{W_i\in A}-nP(A)}{\sqrt{nP(A)(1-P(A))}}.$$
Now the question becomes how to choose $A$. 
Since the mean and the variance of $T_n(A)$ are $0$ and $1$ under $H_0$, the test $\indc{|T_n(A)|>c_n}$ for some slowly diverging sequence $c_n$ will be consistent if $\frac{(\mathbb{E}_{H_1}T_n(A))^2}{\Var_{H_1}(T_n(A))}\rightarrow\infty$ by  {applying} Chebyshev's inequality. 
A direct calculation gives
$$\frac{(\mathbb{E}_{H_1}T_n(A))^2}{\Var_{H_1}(T_n(A))}=\frac{\left(n\epsilon(Q(A)-P(A))\right)^2}{n\left((1-\epsilon)P(A)+\epsilon Q(A)\right)\left(1-\left((1-\epsilon)P(A)+\epsilon Q(A)\right)\right)}.$$
By  {symmetry of the righthand side}, we may consider $(1-\epsilon)P(A)+\epsilon Q(A)\leq \frac{1}{2}$ without loss of generality. This leads to the simplification
\begin{equation}
\frac{(\mathbb{E}_{H_1}T_n(A))^2}{\Var_{H_1}(T_n(A))}\asymp \frac{\left(n\epsilon(Q(A)-P(A))\right)^2}{nP(A)+n\epsilon Q(A)}. \label{eq:ratio-mean-var}
\end{equation}
In order that this ratio tends to infinity, we require either $\frac{(n\epsilon P(A))^2}{n P(A)+n\epsilon Q(A)}\rightarrow\infty$ or $\frac{(n\epsilon Q(A))^2}{n P(A)+n\epsilon Q(A)}\rightarrow\infty$. 
Suppose $\frac{(n\epsilon P(A))^2}{n P(A)+n\epsilon Q(A)}\rightarrow\infty$ holds, and then we have $n\epsilon^2 P(A)\rightarrow\infty$, which requires $\beta<\frac{1}{2}$, 
 {which is too strong a condition to be of our interest.}
Therefore, we require $\frac{(n\epsilon Q(A))^2}{n P(A)+n\epsilon Q(A)}\rightarrow\infty$, which can be equivalently written as two conditions
$$\frac{n\epsilon^2 Q(A)^2}{P(A)}\rightarrow\infty\quad\text{and}\quad n\epsilon Q(A)\rightarrow\infty.$$
 {With the calibration $\epsilon=n^{-\beta}$,} 
these two conditions  {are equivalent to}
\begin{equation}
\beta<\frac{1}{2}+\frac{\log Q(A)}{\log n}+\frac{1}{2}\min\left(1,\frac{\log\frac{1}{P(A)}}{\log n}\right). \label{eq:beta-abstract}
\end{equation}
To maximize the detection region, we shall consider some event $A$ that makes the righthand side of (\ref{eq:beta-abstract}) as large as possible. Since the righthand side of (\ref{eq:beta-abstract}) is increasing in $Q(A)$ and decreasing in $P(A)$, the maximum is achieved by $A=\{\frac{dQ}{dP}(W)>t \}$ for some  {appropriate choice of} $t$
according to the Neyman--Pearson lemma.
This fact naturally motivates the choice
$$\mathcal{A}=\left\{ \Big\{\frac{dQ}{dP}(W)>t \Big\}: t>0\right\}$$
in (\ref{eq:HC-abstract}),  {which} results in the HC-type statistic
\begin{equation}
\sup_{t>0}\frac{\left|\sum_{i=1}^n\indc{(dQ/dP)(W_i)>t}-nP((dQ/dP)(W)>t)\right|}{\sqrt{nP((dQ/dP)(W)>t)P((dQ/dP)(W)\leq t)}}. \label{eq:HC-abstract-optimal}
\end{equation}

\subsection{Likelihood Ratio Approximation}
The heuristic argument in Section \ref{sec:which-count} suggests  {that we use} the statistic $\sum_{i=1}^n\indc{(dQ/dP)(W_i)>t}$. 
We specify $P$ and $Q$ to the setting of (\ref{eq:general-comb-null})-(\ref{eq:general-comb-alt})
 {to obtain that}
$$\frac{dQ}{dP}(W_i)=\frac{q(U_i,V_i)}{p(U_i,V_i)},$$
where
\begin{eqnarray}
\label{eq:general-P-den} p(u,v) &=& \frac{1}{2}\phi(u)\phi(v-\sqrt{2(r+s)\log n}) + \frac{1}{2}\phi(u)\phi(v+\sqrt{2(r+s)\log n}), \\
\label{eq:general-Q-den}  q(u,v) &=& \frac{1}{2}\phi(u-\sqrt{2r\log n})\phi(v-\sqrt{2s\log n})  \\
\nonumber && + \frac{1}{2}\phi(u+\sqrt{2r\log n})\phi(v+\sqrt{2s\log n}).
\end{eqnarray}
 {Here $\phi(\cdot)$ is the probability density function of $N(0,1)$.}
The following  {key} lemma greatly simplifies the calculation of the likelihood ratio statistic.

\begin{lemma}\label{lem:LR-approx}
For $p(u,v)$ and $q(u,v)$ defined above, we have
$$\sup_{r,s>0}\sup_{u,v\in\mathbb{R}}\left|\log\frac{q(u,v)}{p(u,v)}-\sqrt{2\log n}\left(|\sqrt{r}u+\sqrt{s}v|-\sqrt{r+s}|v|\right)\right|\leq \log 2.$$
\end{lemma}

By Lemma \ref{lem:LR-approx}, 
$\sqrt{2\log n}\left(|\sqrt{r}u+\sqrt{s}v|-\sqrt{r+s}|v|\right)$ is the leading term of $\log\frac{q(u,v)}{p(u,v)}$ as $n\rightarrow\infty$. 
Therefore,  {from an asymptotic viewpoint, we could simply focus on the sequence} 
$$\{|\sqrt{r}U_i+\sqrt{s}V_i|-\sqrt{r+s}|V_i|\}_{1\leq i\leq n}$$ 
which combines the information of $\{U_i\}_{1\leq i\leq n}$ and $\{V_i\}_{1\leq i\leq n}$. 
When $s=0$,  {it reduces to} $\{\sqrt{r}(|U_i|-|V_i|)\}_{1\leq i\leq n}$, which  {further} justifies the optimality of the test (\ref{eq:HC-comb}) when $\|\theta\|=\|\eta\|$. 
As $s\rightarrow\infty$, we have $\sqrt{r+s}-\sqrt{s}=\frac{r}{\sqrt{r+s}+\sqrt{s}}\rightarrow 0$, and  {it can shown that}
the sequence becomes $\{\sqrt{r}U_i\sgn(V_i)\}_{1\leq i\leq n}$.
 {In other word,} asymptotically only the sign information of the sequence $\{V_i\}_{1\leq i\leq n}$  {matters as $s\to\infty$}. 

\subsection{The Three-Dimensional Phase Diagram} 

 {We now move on to determine detection boundaries for \eqref{eq:general-comb-null}-\eqref{eq:general-comb-alt} and for \eqref{eq:problem} in general.}

\paragraph{Sparse mixture detection}
 {Consider the sparse mixture detection problem \eqref{eq:general-comb-null}-\eqref{eq:general-comb-alt} first.}
Inspired by 
Lemma \ref{lem:LR-approx}, we consider the following HC-type test with  rejection region
\begin{equation}
\sup_{t\in\mathbb{R}}\frac{\left|\sum_{i=1}^n\indc{|\sqrt{r}U_i+\sqrt{s}V_i|-\sqrt{r+s}|V_i|>t}-nS_{(r,s)}(t)\right|}{\sqrt{n S_{(r,s)}(t)(1-S_{(r,s)}(t))}}>\sqrt{2(1+\delta)\log\log n}, \label{eq:HC-comb-general}
\end{equation}
where $\delta>0$ is some arbitrary fixed constant, and 
$S_{(r,s)}(t)$ is the survival function of $|\sqrt{r}U_i+\sqrt{s}V_i|-\sqrt{r+s}|V_i|$ under the null distribution, defined by 
\begin{equation}
S_{(r,s)}(t)=
\mathbb{P}_{ {H_0}}
\left(|\sqrt{r}U+\sqrt{s}V|-\sqrt{r+s}|V|>t\right), \label{eq:survival-r-s}
\end{equation}
 {where $H_0$ is defined in \eqref{eq:general-comb-null}.}
By Lemma \ref{lem:LR-approx} and our heuristic argument in Section \ref{sec:which-count}, the  {test} statistic in (\ref{eq:HC-comb-general}) is asymptotically  {equivalent to} (\ref{eq:HC-abstract-optimal}). 
Indeed, the test  {with rejection region} (\ref{eq:HC-comb-general}) achieves the optimal detection boundary of the testing problem (\ref{eq:general-comb-null})-(\ref{eq:general-comb-alt}), 
 {which is summarized as the following theorem.}

\begin{thm}\label{thm:general-HC}
Consider  {testing} (\ref{eq:general-comb-null})-(\ref{eq:general-comb-alt}) with $\epsilon=n^{-\beta}$. Define
$$\beta^*(r,s)=\begin{cases}
\frac{1}{2}+2(r+s-\sqrt{s}\sqrt{r+s}), &   3s>r \mathrm{~and~} r+s-\sqrt{s}\sqrt{r+s}\leq \frac{1}{8}, \\
\frac{1}{2}({1+3r-s}), &   3s\leq r \mathrm{~and~} 5r+s\leq 1, \\
2\sqrt{r}\sqrt{1-r-s}, &   5r+s>1, ~\frac{1}{8}<r+s-\sqrt{s}\sqrt{r+s}\leq \frac{1}{2}, \\
& ~~\mathrm{~and~}2(1-r-s)(r+s-\sqrt{s}\sqrt{r+s})>r, \\
\big[2\sqrt{2(r+s-\sqrt{s}\sqrt{r+s})} &   5r+s>1, ~\frac{1}{8}<r+s-\sqrt{s}\sqrt{r+s}\leq \frac{1}{2}, \\
~~~ -2(r+s-\sqrt{s}\sqrt{r+s})\big], & ~~\mathrm{~and~} 2(1-r-s)(r+s-\sqrt{s}\sqrt{r+s})\leq r, \\
1, &   r+s-\sqrt{s}\sqrt{r+s} > \frac{1}{2}.
\end{cases}$$
When $\beta<\beta^*(r,s)$, the test  {with rejection region} (\ref{eq:HC-comb-general}) is consistent. When $\beta>\beta^*(r,s)$, no test is consistent.
\end{thm}

\paragraph{Testing equivalence of clustering}
{Turn to} the original testing problem (\ref{eq:problem}). 
Note that the two sequences (\ref{eq:seq1}) and (\ref{eq:seq2}) play the same roles as $\{U_i\}_{1\leq i\leq n}$ and $\{V_i\}_{1\leq i\leq n}$  {do in sparse mixture detection}. 
 {In view of the parameterization in} (\ref{eq:general-cali})-(\ref{eq:min-sig}), 
we define
\begin{eqnarray}
C^-(X_i,Y_i,\theta,\eta) &=& |\theta^TX_i-\eta^TY_i|-|\theta^TX_i+\eta^TY_i|,
\label{eq:cminus}
 \\
C^+(X_i,Y_i,\theta,\eta) &=& |\theta^TX_i+\eta^TY_i|-|\theta^TX_i-\eta^TY_i|.
\label{eq:cplus}
\end{eqnarray}
 {For testing \eqref{eq:problem},}
we need both $\{C^-(X_i,Y_i,\theta,\eta)\}_{1\leq i\leq n}$ and $\{C^+(X_i,Y_i,\theta,\eta)\}_{1\leq i\leq n}$ to accommodate the possibility of label switching.
Then, the HC-type statistics for testing (\ref{eq:problem})  {can be} defined as
\begin{eqnarray}
\label{eq:stat-t-dot--} \dot{T}_n^- &=& \sup_{t\in\mathbb{R}}\frac{\left|\sum_{i=1}^n\indc{C^-(X_i,Y_i,\theta,\eta)>t\sqrt{2\log n}}-nS_{(r,s)}(t)\right|}{\sqrt{n S_{(r,s)}(t)(1-S_{(r,s)}(t))}}, \\
\label{eq:stat-t-dot-+} \dot{T}_n^+ &=& \sup_{t\in\mathbb{R}}\frac{\left|\sum_{i=1}^n\indc{C^+(X_i,Y_i,\theta,\eta)>t\sqrt{2\log n}}-nS_{(r,s)}(t)\right|}{\sqrt{n S_{(r,s)}(t)(1-S_{(r,s)}(t))}}.
\end{eqnarray}
They lead to the test
\begin{equation}
\dot{\psi}=\indc{\dot{T}_n^-\wedge \dot{T}_n^+>\sqrt{2(1+\delta)\log\log n}}, \label{eq:general-test-comb}
\end{equation}
for some arbitrary fixed constant $\delta>0$.

\begin{thm}\label{thm:main-general}
 {For testing} (\ref{eq:problem}) with calibration (\ref{eq:general-cali}), the test (\ref{eq:general-test-comb}) satisfies $\lim_{n\rightarrow\infty}R_n(\dot{\psi},\theta,\eta,\epsilon)=0$ as long as $\beta<{\beta}^*(r,s)$. Moreover, when $\beta>{\beta}^*(r,s)$, we have $\liminf_{n\rightarrow\infty}R_n(\theta,\eta,\epsilon)>0$.
\end{thm}

\begin{figure}[!tbh]
\begin{center}
\includegraphics[width=\textwidth]{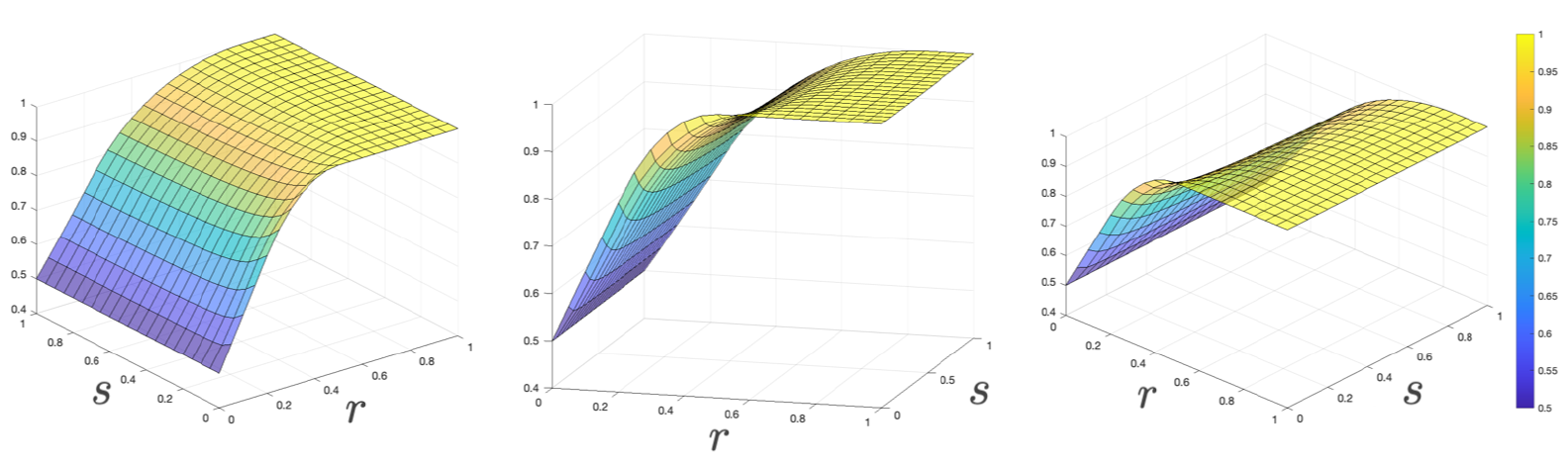}
\caption{ {3D plot of the detection boundary $\beta^*(r,s)$.}}\label{fig:3d}
\end{center}
\end{figure}

With Theorem \ref{thm:main-general}, we completely characterize the detection boundary of the testing problem (\ref{eq:problem}) by the function $\beta^*(r,s)$. 
To  {help understanding} the behavior of $\beta^*(r,s)$, Figure \ref{fig:3d}  {demonstrates its 3D plot} from various angles.
In addition,
we plot the five regions  {that divide the domain of} $\beta^*(r,s)$,  {that is $(0,\infty)\times [0,\infty)$}, on the left panel of Figure \ref{fig:combined}. 
 {Furthermore,
we fix $s$} and study the behavior of the function $\beta_s^*(r)=\beta^*(r,s)$  {as a function of $r$ at some fixed $s$ value}.
\begin{figure}[!tb] 
\begin{center}
\includegraphics[width=\textwidth]{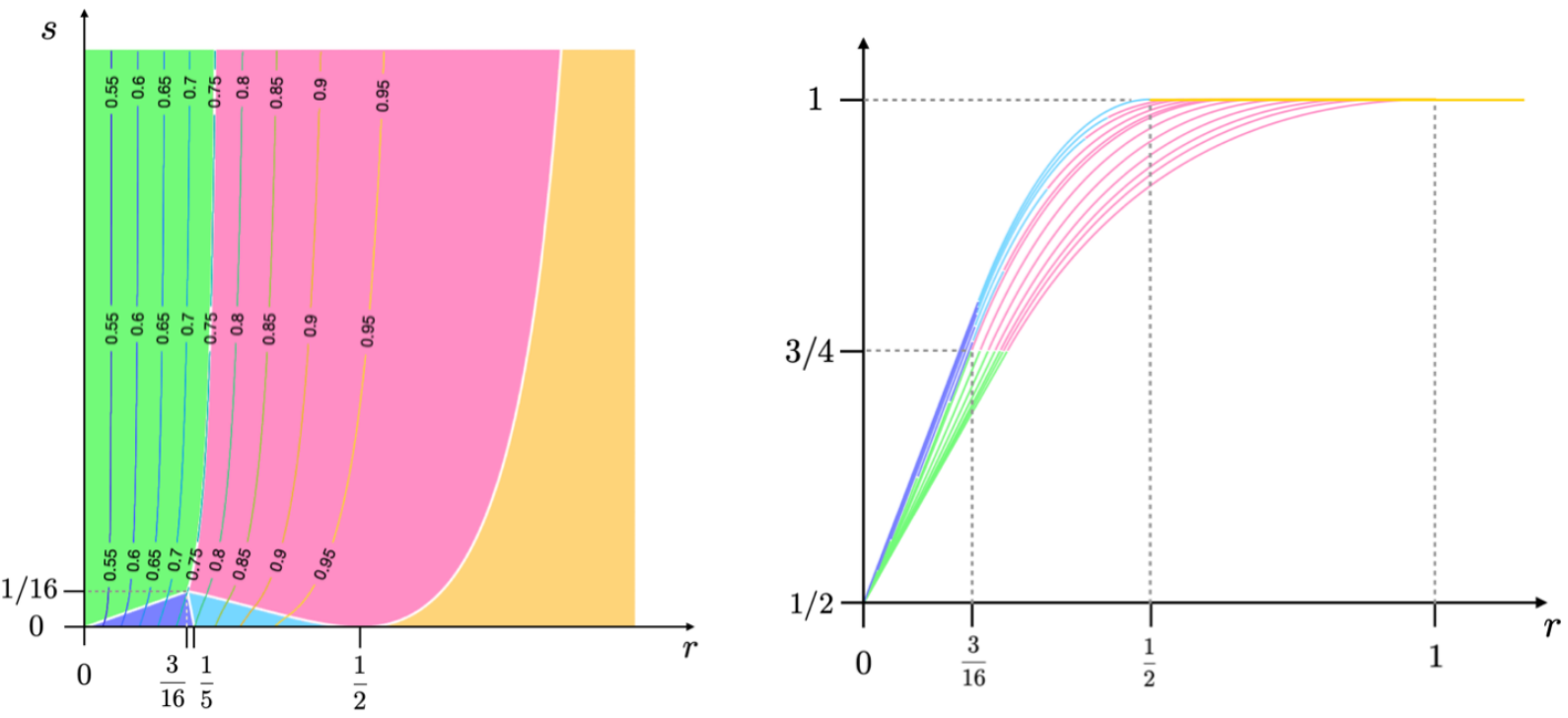}
\caption{The five regions of the $(r,s)$-plane with the contour of $\beta^*(r,s)$ (Left Panel). The detection boundaries $\beta_s^*(r)=\beta^*(r,s)$ with $s$ fixed (Right Panel). The curve moves to the right as the fixed value of $s$ increases. 
The five colors of the two plots correspond to the five regions of $\beta^*(r,s)$ in the order of green, blue, cyan, magenta, and yellow.}\label{fig:combined}
\end{center}
\end{figure}
We start with $s=0$. In this case,  {the problem reduces to}
the  {equal SNR}
situation, and we are able to recover $\beta_s^*(r)=\beta^*(r)$, where $\beta^*(r)$ is defined in (\ref{eq:freestyle}). 
 {For any fixed}
$s\in \left(0,\frac{1}{16}\right)$, the definition of $\beta_s^*(r)$ involves all the five areas in  {the left panel of} Figure \ref{fig:combined}, and we have
$$\beta_s^*(r)=\begin{cases}
\frac{1}{2}+2(r+s-\sqrt{s}\sqrt{r+s}), &   0<r<3s, \\
\frac{1}{2}(1+3r-s), &   3s\leq r< \frac{1-s}{5}, \\
2\sqrt{r}\sqrt{1-r-s}, &   \frac{1-s}{5}\leq r < \rt(s), \\
2\sqrt{2(r+s-\sqrt{s}\sqrt{r+s})}-2(r+s-\sqrt{s}\sqrt{r+s}), &   \rt(s) \leq r <\left(\sqrt{\frac{1}{2}+\frac{s}{4}}+\sqrt{\frac{s}{4}}\right)^2-s, \\
1, &   r>\left(\sqrt{\frac{1}{2}+\frac{s}{4}}+\sqrt{\frac{s}{4}}\right)^2-s.
\end{cases}$$
Here $r=\rt(s)$ is a root of the equation $2(1-r-s)(r+s-\sqrt{s}\sqrt{r+s})=r$. We note that when $s\in \left(0,\frac{1}{16}\right)$, the equation has a unique  {real} root between $\frac{3}{16}$ and $\frac{1}{2}$.
Next, we consider  {any fixed} $s\geq\frac{1}{16}$. 
 {In this case, two regions become degenerate}, and we have
$$\beta^*_s(r)=\begin{cases}
\frac{1}{2}+2(r+s-\sqrt{s}\sqrt{r+s}), &   0<r< \left(\sqrt{\frac{1}{8}+\frac{s}{4}}+\sqrt{\frac{s}{4}}\right)^2-s,\\
\big[2\sqrt{2(r+s-\sqrt{s}\sqrt{r+s})} &   \left(\sqrt{\frac{1}{8}+\frac{s}{4}}+\sqrt{\frac{s}{4}}\right)^2-s \leq r < \left(\sqrt{\frac{1}{2}+\frac{s}{4}}+\sqrt{\frac{s}{4}}\right)^2-s, \\
~~~ -2(r+s-\sqrt{s}\sqrt{r+s})\big], \\
1, &   r\geq \left(\sqrt{\frac{1}{2}+\frac{s}{4}}+\sqrt{\frac{s}{4}}\right)^2-s.
\end{cases}$$
Last but not least, we would like to point out that when $s=\infty$, we obtain the Ingster--Donoho--Jin threshold $\beta_s^*(r)=\beta^*_{\rm IDJ}(r)$. This agrees with the intuition that the sequence $\{V_i\}_{1\leq i\leq n}$ is asymptotically non-informative for the testing problem (\ref{eq:general-comb-null})-(\ref{eq:general-comb-alt}) as $s\rightarrow\infty$. The functions $\{\beta_s^*(r)\}$ with various choices of $s$ are shown on the right panel of Figure \ref{fig:combined}, and all the curves are between $\beta^*(r)$ and $\beta_{\rm IDJ}^*(r)$ (also see Figure \ref{fig:phase-equal}). It is clear that for a fixed $s$, a larger $r$ makes the testing problem easier. On the other hand, increasing $s$ always makes the problem harder in the sense that $\beta^*_{s_1}(r)\geq\beta^*_{s_2}(r)$  {for all $r>0$} when $s_1<s_2$.

\section{Testing for Exact Equality}\label{sec:equality}

 {The most stringent}
version of the testing problem (\ref{eq:problem}) is whether or not the two clustering structures are exactly equal. This can be formulated into the following hypothesis testing problem:
\begin{equation}
H_0:\ell(z,\sigma)=0\quad \mbox{ {vs.}}\quad
H_1:\ell(z,\sigma)> 0. \label{eq:problem-exact}
\end{equation}
Since the loss function $\ell(z,\sigma)$ only takes value in the set $\{0,n^{-1},2n^{-1},\cdots\}$, the alternative hypothesis of (\ref{eq:problem-exact}) is equivalent to $\ell(z,\sigma)\geq n^{-1}$. Therefore, the testing problem (\ref{eq:problem-exact}) is 
a special case of (\ref{eq:problem}) with $\beta=1$. However, Theorem \ref{thm:main-general} only covers $\beta<1$. Since the lower bound proof of Theorem \ref{thm:main-general} is based on the connection between (\ref{eq:problem}) and (\ref{eq:general-comb-null})-(\ref{eq:general-comb-alt}), 
which requires $n\epsilon\rightarrow\infty$, the  {boundary case of} $\beta=1$ is thus excluded.

In this section, we rigorously study the testing problem (\ref{eq:problem-exact}). Given a testing procedure $\psi$, we define its worst-case testing error by
$$R_n^{\rm exact}(\psi,\theta,\eta)=\sup_{\substack{z\in\{-1,1\}^n\\ \sigma\in\{z,-z\}~\,}}P^{(n)}_{(\theta,\eta,z,\sigma)}\psi + \sup_{\substack{z\in\{-1,1\}^n\\\sigma\in\{-1,1\}^n\\\ell(z,\sigma)> 0}}P^{(n)}_{(\theta,\eta,z,\sigma)}(1-\psi).$$
The minimax testing error is  {then} defined by
$$R_n^{\rm exact}(\theta,\eta)=\inf_{\psi}R_n^{\rm exact}(\psi,\theta,\eta).$$
Our first result gives a necessary and sufficient condition for the existence of a consistent test.

\begin{thm}\label{thm:exact}
Consider  {testing} (\ref{eq:problem-exact}) with calibration (\ref{eq:general-cali}). 
When $r+s-\sqrt{s}\sqrt{r+s}<\frac{1}{2}$, we have $\liminf_{n\rightarrow\infty}R_n^{\rm exact}(\theta,\eta)>0$. When $r+s-\sqrt{s}\sqrt{r+s}>\frac{1}{2}$, the HC-type test $\dot{\psi}$ defined in (\ref{eq:general-test-comb}) satisfies $\lim_{n\rightarrow\infty}R_n^{\rm exact}(\dot{\psi},\theta,\eta)=0$.
\end{thm}

Theorem \ref{thm:exact} shows that whether 
 {$r+s-\sqrt{s}\sqrt{r+s}$ is above or below $\frac{1}{2}$}
determines the existence of a consistent test. This is compatible with the last regime of the threshold function $\beta^*(r,s)$. 
See the yellow area in the left panel of Figure \ref{fig:combined}. 
Given the relation (\ref{eq:min-sig}), it is required that both $\|\theta\|^2$ and $\|\eta\|^2$  {are} 
greater than $\frac{1}{2}\log n$  {for separating} the null and the alternative hypotheses. 
Moreover, the same optimal HC-type test in Theorem \ref{thm:main-general} continues to work for testing exact equality.

In addition to the HC-type test, we introduce a 
Bonferroni-type test that is also optimal for (\ref{eq:problem-exact}). 
 {To this end, define}
$$
t^*(r,s)=\begin{cases}
\sqrt{r(1-r-s)}, & 2(r+s)(r+s+\sqrt{s}\sqrt{r+s})\leq r, \\
\sqrt{2(r+s-\sqrt{s}\sqrt{r+s})}-(r+s-\sqrt{s}\sqrt{r+s}), & 2(r+s)(r+s+\sqrt{s}\sqrt{r+s})> r.
\end{cases}
$$
 {The following lemma shows that}
it characterizes the largest element of the sequence $\{|\sqrt{r}U_i+\sqrt{s}V_i|-\sqrt{r+s}|V_i|\}_{1\leq i\leq n}$ under the null distribution.
\begin{lemma}\label{lem:max-order}
Suppose $\{(U_i,V_i)\}_{1\leq i\leq n}$ are generated according to (\ref{eq:general-comb-null}). Then,
we have
$$\frac{\max_{1\leq i\leq n}\left(|\sqrt{r}U_i+\sqrt{s}V_i|-\sqrt{r+s}|V_i|\right)}{\sqrt{2\log n}}\rightarrow t^*(r,s),$$
in probability.
\end{lemma}
Lemma \ref{lem:max-order} shows that the largest element of the sequence $\{|\sqrt{r}U_i+\sqrt{s}V_i|-\sqrt{r+s}|V_i|\}_{1\leq i\leq n}$ is asymptotically $t^*(r,s)\sqrt{2\log n}$ 
under $H_0$. It is therefore natural to reject $H_0$ when the random variable $\max_{1\leq i\leq n}\left(|\sqrt{r}U_i+\sqrt{s}V_i|-\sqrt{r+s}|V_i| \right)$ is larger than $t^*(r,s)\sqrt{2\log n}$. 
 {In view of the connection between sparse mixture detection and testing clustering equivalence,}
applying the result to the sequences $\{C^-(X_i,Y_i,\theta,\eta)\}_{1\leq i\leq n}$ and $\{C^+(X_i,Y_i,\theta,\eta)\}_{1\leq i\leq n}$, we obtain the following 
testing procedure, 
\begin{equation}
\psi_{\rm Bonferroni}=\indc{\left(\max_{1\leq i\leq n}C^-(X_i,Y_i,\theta,\eta)\right)\wedge \left(\max_{1\leq i\leq n}C^+(X_i,Y_i,\theta,\eta)\right)>2t^*(r,s)\log n}. \label{eq:test-exact-Bonf}
\end{equation}

\begin{thm}\label{thm:Bonf}
Consider testing (\ref{eq:problem-exact}) with calibration (\ref{eq:general-cali}). When $r+s-\sqrt{s}\sqrt{r+s}>\frac{1}{2}$, we have $\lim_{n\rightarrow\infty}R_n^{\rm exact}(\psi_{\rm Bonferroni},\theta,\eta)=0$.
\end{thm}

\section{Adaptive Tests}\label{sec:adaptive-t}
In this section, we  {investigate} 
how to  {test}
(\ref{eq:problem}) and (\ref{eq:problem-exact}) when the model parameters $\theta\in\mathbb{R}^p$ and $\eta\in\mathbb{R}^q$ are unknown. 
We will show that both the HC-type test and the Bonferroni test can be modified into adaptive procedures, as long as  {some} mild {growth rate} conditions on the dimensions $p$ and $q$ are satisfied.

\subsection{Adaptive Bonferroni Test}

 {We start with testing \eqref{eq:problem-exact}.}
When designing the adaptive procedures, we adopt a random data splitting 
 {scheme}.
We first draw $d_1,\cdots, d_n\stackrel{iid}{\sim}\text{Bernoulli}(\frac{1}{2})$, and then define $\mathcal{D}_0=\{i\in[n]: d_i=0\}$ and $\mathcal{D}_1=\{i\in[n]:d_i=1\}$. 
Then, $\{\mathcal{D}_0,\mathcal{D}_1\}$  {forms} a random partition of $[n]$. 
Given some algorithms $\wh{\theta}(\cdot)$ and $\wh{\eta}(\cdot)$ that compute estimators of $\theta$ and $\eta$, we define 
 {
$\wh{\theta}^{(m)}=\wh{\theta}(\{(X_i,Y_i)\}_{i\in\mathcal{D}_m})$ and 
$\wh{\eta}^{(m)}=\wh{\eta}(\{(X_i,Y_i)\}_{i\in\mathcal{D}_m})$
for $m = 0$ and $1$.}
{For $m = 0$ and $1$,
by plugging $\wh{\theta}^{(m)}$ and $\wh{\eta}^{(m)}$ into the relation (\ref{eq:general-cali}), we obtain $\wh{r}^{(m)}$ and $\wh{s}^{(m)}$.  }

Given these estimators of $\theta$ and $\eta$, we can modify (\ref{eq:test-exact-Bonf}) into an adaptive procedure. We replace $\max_{1\leq i\leq n}C^-(X_i,Y_i,\theta,\eta)$ and $\max_{1\leq i\leq n}C^+(X_i,Y_i,\theta,\eta)$ by 
\begin{eqnarray*}
\wh{C}_m^- &=& \max_{i\in\mathcal{D}_m}C^-(X_i,Y_i,\wh{\theta}^{(1-m)},\wh{\eta}^{(1-m)}),\quad m=0,1, \\
\wh{C}_m^+ &=& \max_{i\in\mathcal{D}_m}C^+(X_i,Y_i,\wh{\theta}^{(1-m)},\wh{\eta}^{(1-m)}),\quad m=0,1.
\end{eqnarray*}
Then, we combine these statistics by
\begin{eqnarray}
\label{eq:levivsbeast} \wh{C}^- &=& \begin{cases}
\wh{C}_0^- \vee \wh{C}_1^-, & \indc{\|\wh{\theta}^{(0)}-\wh{\theta}^{(1)}\|\leq 1, \|\wh{\eta}^{(0)}-\wh{\eta}^{(1)}\|\leq 1} + \indc{\|\wh{\theta}^{(0)}-\wh{\theta}^{(1)}\|> 1, \|\wh{\eta}^{(0)}-\wh{\eta}^{(1)}\|> 1}=1, \\
\wh{C}_0^- \vee \wh{C}_1^+, & \indc{\|\wh{\theta}^{(0)}-\wh{\theta}^{(1)}\|> 1, \|\wh{\eta}^{(0)}-\wh{\eta}^{(1)}\|\leq 1} + \indc{\|\wh{\theta}^{(0)}-\wh{\theta}^{(1)}\|\leq 1, \|\wh{\eta}^{(0)}-\wh{\eta}^{(1)}\|> 1}=1,
\end{cases} \\
\label{eq:elvinisgreat} \wh{C}^+ &=& \begin{cases}
\wh{C}_0^+ \vee \wh{C}_1^+, & \indc{\|\wh{\theta}^{(0)}-\wh{\theta}^{(1)}\|\leq 1, \|\wh{\eta}^{(0)}-\wh{\eta}^{(1)}\|\leq 1} + \indc{\|\wh{\theta}^{(0)}-\wh{\theta}^{(1)}\|> 1, \|\wh{\eta}^{(0)}-\wh{\eta}^{(1)}\|> 1}=1, \\
\wh{C}_0^+ \vee \wh{C}_1^-, & \indc{\|\wh{\theta}^{(0)}-\wh{\theta}^{(1)}\|> 1, \|\wh{\eta}^{(0)}-\wh{\eta}^{(1)}\|\leq 1} + \indc{\|\wh{\theta}^{(0)}-\wh{\theta}^{(1)}\|\leq 1, \|\wh{\eta}^{(0)}-\wh{\eta}^{(1)}\|> 1}=1.
\end{cases}
\end{eqnarray}
The adaptive Bonferroni test is defined by 
$$\psi_{\rm ada-Bonferroni}=\indc{\wh{C}^-\wedge \wh{C}^+>2\left(1+\frac{1}{\sqrt{\log n}}\right)\wh{t}\log n},$$
where
$$\wh{t}=\frac{t^*(\wh{r}^{(0)},\wh{s}^{(0)})+t^*(\wh{r}^{(1)},\wh{s}^{(1)})}{2}.$$
The additional factor $\left(1+\frac{1}{\sqrt{\log n}}\right)$ is to accommodate the error caused by estimators of $\theta$ and $\eta$. Before writing down the theorem that gives the desired theoretical guarantee for $\psi_{\rm ada-Bonferroni}$, let us define the loss functions
$$L(\wh{\theta},\theta)=\|\wh{\theta}-\theta\|\wedge\|\wh{\theta}+\theta\|,\quad L(\wh{\eta},\eta)=\|\wh{\eta}-\eta\|\wedge\|\wh{\eta}+\eta\|.$$
Though $\theta$ and $\eta$ can be of different dimensions, we use the same notation $L(\cdot,\cdot)$ for the two loss functions simplicity.

\begin{thm}\label{thm:ada-Bonf}
We consider the testing problem (\ref{eq:problem-exact}) with the calibration (\ref{eq:general-cali}). 
Assume that there is some constant $\gamma>0$, such that
\begin{equation}
\lim_{n\rightarrow\infty}\sup_{\substack{z\in\{-1,1\}^n\\\sigma\in\{-1,1\}^n}}P^{(n)}_{(\theta,\eta,z,\sigma)}\left(L(\wh{\theta}^{(0)},\theta)\vee L(\wh{\theta}^{(1)},\theta)\vee L(\wh{\eta}^{(0)},\eta)\vee L(\wh{\eta}^{(1)},\eta)>n^{-\gamma}\right) =0. \label{eq:estimation-error-weak}
\end{equation}
When $r+s-\sqrt{s}\sqrt{r+s}>\frac{1}{2}$, we have $\lim_{n\rightarrow\infty}R_n^{\rm exact}(\psi_{\rm ada-Bonferroni},\theta,\eta)=0$.
\end{thm} 

 {The condition \eqref{eq:estimation-error-weak} may seem abstract at first sight. Later in Section \ref{sec:para-est}, we shall give concrete estimators so that it is met under a mild growth condition on $p$ and $q$.
For full details, see Corollary \ref{cor:adaptive-dim}.}

\subsection{Adaptive HC-Type Test}

To modify (\ref{eq:general-test-comb}) into an adaptive procedure is more involved. 
This is  {due to the fact that} we  {not only} need to estimate the statistics $\{C^-(X_i,Y_i,\theta,\eta)\}_{1\leq i\leq n}$ and $\{C^+(X_i,Y_i,\theta,\eta)\}_{1\leq i\leq n}$,  {but} also need to estimate the survival function $S_{(r,s)}(t)$ defined in (\ref{eq:survival-r-s}). 
Our proposed strategy starts with a random data splitting step. 
This time we split the data into three parts  {instead of two}. 
Draw $d_1,\cdots, d_n\stackrel{iid}{\sim}\text{Uniform}(\{0,1,2\})$, and then define $\mathcal{D}_m=\{i\in[n]:d_i=m\}$ for $m\in\{0,1,2\}$.

Given some algorithms $\wh{\theta}(\cdot)$ and $\wh{\eta}(\cdot)$ that compute estimators of $\theta$ and $\eta$, we first define $\wh{\theta}=\wh{\theta}(\{(X_i,Y_i)\}_{i\in\mathcal{D}_0})$ and $\wh{\eta}=\wh{\eta}(\{(X_i,Y_i)\}_{i\in\mathcal{D}_0})$. 
We then use $\wh{\theta}$ and $\wh{\eta}$ for 
 {projection}
and compute $\wh{X}_i=\wh{\theta}^TX_i/\|\wh{\theta}\|$ and $\wh{Y}_i=\wh{\eta}^TY_i/\|\wh{\eta}\|$ for all $i\in\mathcal{D}_1\cup\mathcal{D}_2$. Note that conditioning on $\{d_i\}_{1\leq i\leq n}$ and $\{(X_i,Y_i)\}_{i\in\mathcal{D}_0}$, $\wh{X}_i$ and $\wh{Y}_i$ are distributed according to $N(z_ia,1)$ and $N(\sigma_ib,1)$, respectively, where $a=\wh{\theta}^T\theta/\|\wh{\theta}\|$ and $b=\wh{\eta}^T\eta/\|\wh{\eta}\|$. 
 {Given the projected data,}
we will use  {those in} $\mathcal{D}_1$ to estimate the one-dimensional parameters $|a|$ and $|b|$, and  {those in} $\mathcal{D}_2$ to construct the  {test} statistic. Define
$$\wh{a}=\sqrt{\left(\frac{1}{|\mathcal{D}_1|}\sum_{i\in\mathcal{D}_1}X_i^2-1\right)_+}\quad\text{and}\quad\wh{b}=\sqrt{\left(\frac{1}{|\mathcal{D}_1|}\sum_{i\in\mathcal{D}_1}Y_i^2-1\right)_+}.$$
With $\wh{a}$ and $\wh{b}$, we define
$$\wh{r}=\frac{(2|\wh{a}||\wh{b}|)^2}{(2\log n)(\wh{a}^2+\wh{b}^2)}\quad\text{and}\quad \wh{s}=\frac{|\wh{a}^2-\wh{b}^2|^2}{(2\log n)(\wh{a}^2+\wh{b}^2)}.$$
Then, the adaptive HC-type statistics are
\begin{eqnarray*}
\wh{T}_n^- &=& \sup_{|t|\leq \log n}\frac{\left|\sum_{i\in\mathcal{D}_2}\indc{C^-(\wh{X}_i,\wh{Y}_i,\wh{a},\wh{b})>t\sqrt{2\log n}}-|\mathcal{D}_2|S_{(\wh{r},\wh{s})}(t)\right|}{\sqrt{|\mathcal{D}_2|S_{(\wh{r},\wh{s})}(t)}}, \\
\wh{T}_n^+ &=& \sup_{|t|\leq \log n}\frac{\left|\sum_{i\in\mathcal{D}_2}\indc{C^+(\wh{X}_i,\wh{Y}_i,\wh{a},\wh{b})>t\sqrt{2\log n}}-|\mathcal{D}_2|S_{(\wh{r},\wh{s})}(t)\right|}{\sqrt{|\mathcal{D}_2|S_{(\wh{r},\wh{s})}(t)}}.
\end{eqnarray*}
This leads to the adaptive test
$$\wh{\psi}_{\rm ada-HC}=\indc{\wh{T}_n^-\wedge\wh{T}_n^+>(\log n)^3}.$$
Compared with (\ref{eq:stat-t-dot--}) and (\ref{eq:stat-t-dot-+}), the adaptive versions $\wh{T}_n^-$ and $\wh{T}_n^+$ restrict the supremum to the range $|t|\leq \log n$ and does not have an estimator of $1-S_{(r,s)}(t)$ in the denominator. Moreover, the test uses the threshold $(\log n)^3$ instead of the smaller $\sqrt{2(1+\delta)\log\log n}$. 
These changes are  {adopted} to accommodate the  {additional errors} caused by estimating the unknown parameters.

\begin{thm}\label{thm:HC-adaptive}
Consider testing (\ref{eq:problem}) with calibration (\ref{eq:general-cali}). 
Assume that there is some constant $\gamma>0$, such that
\begin{equation}
\lim_{n\rightarrow\infty}\sup_{\substack{z\in\{-1,1\}^n\\\sigma\in\{-1,1\}^n}}P^{(n)}_{(\theta,\eta,z,\sigma)}\left(L(\wh{\theta},\theta)\vee L(\wh{\eta},\eta)>n^{-\gamma}\right) =0. \label{eq:estimation-error-strong}
\end{equation}
When $\beta<\beta^*(r,s)$, we have $\lim_{n\rightarrow\infty}R_n(\psi_{\rm ada-HC},\theta,\eta)=0$.
\end{thm} 

 {As before, for estimators such that condition \eqref{eq:estimation-error-strong} can be fulfilled, see Section \ref{sec:para-est}.}

{Randomly splitting} the data into three parts is {needed for technical details in the proofs}.
{In order for our proof to go through}, we need to estimate the statistics $\{C^-(X_i,Y_i,\theta,\eta)\}_{1\leq i\leq n}$ and $\{C^+(X_i,Y_i,\theta,\eta)\}_{1\leq i\leq n}$ and the survival function $S_{(r,s)}(t)$  {at} different  {levels of} accuracy. 
In particular, we require that the estimation error of $S_{(r,s)}(t)$ to be at most $\frac{(\log n)^{O(1)}}{\sqrt{n}}$, independent of the dimensions $p$ and $q$. Therefore, in addition to the two-part data splitting strategy used in building the adaptive Bonferroni test, we need an additional part to  {estimate projection directions of two one-dimensional subspaces}.

\begin{remark}
One may wonder whether the adaptive HC-type test  {can} also achieve the optimal detection boundary for the problem (\ref{eq:problem-exact}). 
The answer would be no for the current definition of $\psi_{\rm ada-HC}$, because  {under $H_1$} with a {non-trivial} probability, $\mathcal{D}_2$ does not contain the coordinate that has the signal. 
However, a modification of $\psi_{\rm ada-HC}$  {can} resolve this issue. 
The modification requires rotating the roles of the three datasets $\mathcal{D}_0$, $\mathcal{D}_1$, and $\mathcal{D}_2$, and then we can define analogous versions of the HC-type statistics $\wh{T}_n^-$ and $\wh{T}_n^+$ on $\mathcal{D}_0$ and $\mathcal{D}_1$. These statistics can be combined in a similar way to (\ref{eq:levivsbeast}) and (\ref{eq:elvinisgreat}). 
We omit the details.	
\end{remark}

\paragraph{Computation}
We now discuss computation of $\wh{\psi}_{\rm ada-HC}$.
Note that both $\wh{T}_n^-$ and $\wh{T}_n^+$ can be computed efficiently using the $p$-value interpretation of the HC statistic in \cite{donoho2004higher}. In the ideal situation where $\theta$ and $\eta$ are known, the two sets of $p$-values are $\left\{S_{(r,s)}\left(\frac{C^-(X_i,Y_i,\theta,\eta)}{\sqrt{2\log n}}\right)\right\}_{1\leq i\leq n}$ and $\left\{S_{(r,s)}\left(\frac{C^+(X_i,Y_i,\theta,\eta)}{\sqrt{2\log n}}\right)\right\}_{1\leq i\leq n}$, which are involved in the computation of the test (\ref{eq:general-test-comb}). When $\theta$ and $\eta$ are unknown, the following proposition suggests a similar  {computation} strategy.
\begin{proposition}\label{prop:comp-p-value}
Define $\wh{p}_i^-=S_{(\wh{r},\wh{s})}\left(\frac{C^-(\wh{X}_i,\wh{Y}_i,\wh{a},\wh{b})}{\sqrt{2\log n}}\right)$ and $\wh{p}_i^+=S_{(\wh{r},\wh{s})}\left(\frac{C^+(\wh{X}_i,\wh{Y}_i,\wh{a},\wh{b})}{\sqrt{2\log n}}\right)$ for $i\in\mathcal{D}_2$. Then, with probability tending to $1$, we have
\begin{eqnarray}
\label{eq:haha} \wh{T}_n^- &=& \max_{1\leq i\leq |\mathcal{D}_2|}\frac{\sqrt{|\mathcal{D}_2|}\left|\frac{i}{|\mathcal{D}_2|}-\wh{p}_{(i,\mathcal{D}_2)}^-\right|}{\sqrt{\wh{p}_{(i,\mathcal{D}_2)}^-}}, \\
\label{eq:hehe} \wh{T}_n^+ &=& \max_{1\leq i\leq |\mathcal{D}_2|}\frac{\sqrt{|\mathcal{D}_2|}\left|\frac{i}{|\mathcal{D}_2|}-\wh{p}_{(i,\mathcal{D}_2)}^+\right|}{\sqrt{\wh{p}_{(i,\mathcal{D}_2)}^+}},
\end{eqnarray}
where the subscript $(i,\mathcal{D}_2)$ indicates the $i$th order statistic within the set $\mathcal{D}_2$.
\end{proposition}

The statistics $\wh{p}_i^-$ and $\wh{p}_i^+$ can be regarded as estimators of $p$-values, which is a useful interpretation of $\wh{\psi}_{\rm ada-HC}$. Since the formulas (\ref{eq:haha}) and (\ref{eq:hehe}) hold with high probability, $\lim_{n\rightarrow\infty}R_n(\psi_{\rm ada-HC},\theta,\eta)=0$ will continue to hold when $\beta<\beta^*(r,s)$ if (\ref{eq:haha}) and (\ref{eq:hehe}) are used in the computation of $\wh{\psi}_{\rm ada-HC}$.

\subsection{Parameter Estimation}
\label{sec:para-est}

We close this section by  {presenting} a simple estimator for $\theta$ and $\eta$. Since we have that $X_i\sim N(z_i\theta,I_p)$, 
the  {empirical second moment}
$\frac{1}{n}\sum_{i=1}^nX_iX_i^T$ is a consistent estimator of the 
 {population counterpart}
$\theta\theta^T+I_p$. 
Apply eigenvalue decomposition and we get $\frac{1}{n}\sum_{i=1}^nX_iX_i^T=\sum_{j=1}^p\wh{\lambda}_j\wh{u}_j\wh{u}_j^T$, and then a natural estimator for $\theta$ is $\wh{\theta}=\sqrt{\wh{\lambda}_1  {-1}}\,\wh{u}_1$. This simple estimator enjoys the following property.
\begin{proposition}\label{prop:parameter-estimation}
Consider independent observations $X_1,\cdots, X_n\sim N(z_i\theta,I_p)$ with some $z_i\in\{-1,1\}$ for all $i\in[n]$. Assume $p\leq n$, and then there exist universal constants $C,C'>0$, such that
$$L(\wh{\theta},\theta)\leq C\sqrt{\frac{p}{n}},$$
with probability at least $1-e^{-C'p}$ uniformly over all $z\in\{-1,1\}^n$ and all $\theta\in\mathbb{R}^p$ that satisfies $\|\theta\|\geq 1$.
\end{proposition}

\begin{corollary}\label{cor:adaptive-dim}
Consider the calibration (\ref{eq:general-cali}). Suppose $p\vee q<n^{1-\delta}$ for some constant $\delta\in(0,1)$, then there exists some constant $\gamma>0$ depending on $\delta$, such that the conditions (\ref{eq:estimation-error-weak}) and (\ref{eq:estimation-error-strong}) hold.
\end{corollary}

Combine Theorem \ref{thm:ada-Bonf}, Theorem \ref{thm:HC-adaptive}, and Corollary \ref{cor:adaptive-dim}, and we can conclude that the optimal detection boundaries of the testing problems (\ref{eq:problem}) and (\ref{eq:problem-exact}) can be achieved adaptively without the knowledge of $(\theta,\eta)$, as long as the dimensions do not grow too fast in the sense that $p\vee q<n^{1-\delta}$.

In a more general setting, one may have $X_i\sim N(z_i\theta,\sigma^2)$ with both $\theta$ and $\sigma^2$ unknown. The proposed algorithm still works for estimating $\theta$. To estimate $\sigma^2$, one can use the estimator $\wh{\sigma}^2=\frac{1}{p}\Tr\left(\frac{1}{n}\sum_{i=1}^nX_iX_i^T-\wh{\theta}\wh{\theta}^T\right)$. The theoretical analysis can be easily generalized to this case, and we omit the details here.

Last but not least, we remark that the condition $p\vee q<n^{1-\delta}$ can be extended if additional sparsity assumptions on $\theta$ and $\eta$ are imposed. This is related to the sparse clustering setting studied in the literature \citep{azizyan2013minimax,jin2016influential,jin2017phase}, and a sparse PCA algorithm \citep{johnstone2009consistency,ma2013sparse,birnbaum2013minimax,cai2013sparse,vu2013minimax} can be applied to estimate $\theta$ and $\eta$.

\section{Proofs}\label{sec:all-pf}

We give proofs of all the results of the paper in this section. After stating some technical lemmas in Section \ref{sec:pf-lemma-tech}, the proofs are organized in Sections \ref{sec:pf-org1}-\ref{sec:pf-org2}. The technical lemmas stated in Section \ref{sec:pf-lemma-tech} will be proved in Section \ref{sec:pf-last}.

\subsection{Some Technical Lemmas}\label{sec:pf-lemma-tech}

\begin{lemma}\label{prop:standard-normal}
Let $\phi(\cdot)$ and $\Phi(\cdot)$ be the density function and the cumulative distribution function of $N(0,1)$. The following facts hold.
\begin{enumerate}
\item For any $t>0$,
\begin{equation}
(1-t^{-2})\frac{e^{-t^2/2}}{t\sqrt{2\pi}}<1-\Phi(t)<(1-t^{-2}+3t^{-4})\frac{e^{-t^2/2}}{t\sqrt{2\pi}}.\label{eq:Gaussian-tail}
\end{equation}
\item For any $t_1,t_2\in\mathbb{R}$ such that $|t_1-t_2|\leq 1$, we have $\left|\int_{t_1}^{t_2}\phi(x)dx\right|\leq 2|t_1-t_2|\left(\phi(t_1)\vee\phi(t_2)\right)$.
\item We have $\sup_{t\in\mathbb{R}}\frac{\phi(t)/(1\vee t)}{1-\Phi(t)}\leq 20$.
\item For any constant $c>0$, we have $\sup_{\substack{|t_1|,|t_2|\leq (\log n)^2 \\ |t_1-t_2|\leq n^{-c}}}\frac{\phi(t_1)}{\phi(t_2)}\leq 2$ for a sufficiently large $n$.
\item For any constant $c>0$, we have $\sup_{\substack{|t_1|,|t_2|\leq (\log n)^2 \\ |t_1-t_2|\leq n^{-c}}}\frac{1-\Phi(t_1)}{1-\Phi(t_2)}\leq 2$ for a sufficiently large $n$.
\end{enumerate}
\end{lemma}

\begin{lemma}\label{lem:easy-tail}
Consider independent $U^2\sim \chi_{1,2r\log n}^2$ and $V^2\sim \chi_{1,2s\log n}^2$. We have
$$\mathbb{P}\left(U^2\leq 2t\log n\right)\asymp 
\begin{cases}
\frac{1}{\sqrt{\log n}}n^{-(\sqrt{r}-\sqrt{t})^2}, & 0<t<r, \\
1, & t\geq r,
\end{cases}$$
and
$$\mathbb{P}\left(|U|-|V|>t\sqrt{2\log n}\right)\asymp 
\begin{cases}
\frac{1}{\log n}n^{-\left[(t-\sqrt{r})^2+s\right]}, & t>\sqrt{r}+\sqrt{s}, \\
\frac{1}{\sqrt{\log n}}n^{-\frac{1}{2}(t-\sqrt{r}+\sqrt{s})^2}, & \sqrt{r}-\sqrt{s}<t\leq\sqrt{r}+\sqrt{s}, \\
1, & t\leq \sqrt{r}-\sqrt{s}. 
\end{cases}$$
\end{lemma}

\begin{lemma}\label{lem:comp-tail-0}
Consider independent  {$U\sim N(0,1)$ and $V\sim N(\sqrt{2(r+s)\log n},1)$}. 
We have
\begin{eqnarray*}
&& \mathbb{P}\left(|\sqrt{r}U+\sqrt{s}V|-\sqrt{r+s}|V|>t\sqrt{2\log n}\right) \\
&\asymp& \mathbb{P}\left(\sqrt{r}|U|-(\sqrt{r+s}-\sqrt{s})|V|>t\sqrt{2\log n}\right) \\
&\asymp& \begin{cases}
\frac{1}{\log n}n^{-\frac{t^2+r(r+s)}{r}}, & t>r+s+\sqrt{s}\sqrt{r+s}, \\
\frac{1}{\sqrt{\log n}}n^{-\frac{(t+r+s-\sqrt{s}\sqrt{r+s})^2}{2(r+s-\sqrt{s}\sqrt{r+s})}},& -(r+s)+\sqrt{s}\sqrt{r+s}<t\leq r+s+\sqrt{s}\sqrt{r+s}, \\
1, & t\leq -(r+s)+\sqrt{s}\sqrt{r+s}.
\end{cases}
\end{eqnarray*}
\end{lemma}

\begin{lemma}\label{lem:comp-tail-1}
Consider independent $U\sim N(\sqrt{2r\log n},1)$ and $V\sim N(\sqrt{2s\log n},1)$. We have
\begin{eqnarray*}
&& \mathbb{P}\left(|\sqrt{r}U+\sqrt{s}V|-\sqrt{r+s}|V|>t\sqrt{2\log n}\right) \\
&\asymp& \begin{cases}
\frac{1}{\log n}n^{-\frac{(t-r)^2+rs}{r}}, & t>r+s+\sqrt{s}\sqrt{r+s}, \\
\frac{1}{\sqrt{\log n}}n^{-\frac{(t-(r+s)+\sqrt{s}\sqrt{r+s})^2}{2(r+s-\sqrt{s}\sqrt{r+s})}}, & r+s-\sqrt{s}\sqrt{r+s}<t \leq r+s+\sqrt{s}\sqrt{r+s}, \\
1, & t\leq r+s-\sqrt{s}\sqrt{r+s}.
\end{cases}
\end{eqnarray*}
Moreover, for $t<r+s-\sqrt{s}\sqrt{r+s}$, we have
$$\mathbb{P}\left(|\sqrt{r}U+\sqrt{s}V|-\sqrt{r+s}|V|>t\sqrt{2\log n}\right)\rightarrow 1.$$
\end{lemma}

\begin{lemma}[Lemma 3 of \cite{cai2014optimal}]\label{lem:Laplace}
Let $f:\mathcal{X}\rightarrow\mathbb{R}$ be a measurable function on the measurable space $(\mathcal{X},\mathcal{F},\nu)$ that satisfies $\int_{\mathcal{X}}e^{n_0f}d\nu<\infty$ for some $n_0>0$. Then, we have
$$\lim_{n\rightarrow\infty}\frac{1}{n}\log\int_{\mathcal{X}}e^{n f}d\nu=\ess\sup_{x\in\mathcal{X}}f(x),$$
where $\ess\sup_{x\in\mathcal{X}}f(x)=\inf\{a\in\mathbb{R}:\nu(\{x:f(x)>a\})=0\}$ is the essential supremum.
\end{lemma}

\begin{lemma}[Lemma 4 of \cite{cai2014optimal}]\label{lem:Yihong-Wu}
For any $t\geq 0$, $(\sqrt{2}-1)^2 t\wedge t^2\leq (\sqrt{1+t}-1)^2\leq t\wedge t^2$. For any $t\geq -1$, $\sqrt{1+t}\geq 1+\frac{t}{2}-t^2$.
\end{lemma}

\subsection{Proofs of Theorem \ref{thm:V-equal} and Theorem \ref{thm:U-V-equal}} \label{sec:pf-org1}

\begin{proof}[Proof of Theorem \ref{thm:V-equal} (upper bound)]
According to the Neyman-Pearson lemma, the optimal test for the testing problem (\ref{eq:equal-sum-null})-(\ref{eq:equal-sum-alt}) is the likelihood ratio test. So it is sufficient to prove the consistency of any test, and then the consistency of the likelihood ratio test is implied. We define the statistic
$$T_n(t)=\frac{\sum_{i=1}^n\indc{V_i^2\leq 2t\log n}-n\mathbb{P}(\chi_{1,2r\log n}^2\leq 2t\log n)}{\sqrt{n\mathbb{P}(\chi_{1,2r\log n}^2\leq 2t\log n)\left(1-\mathbb{P}(\chi_{1,2r\log n}^2\leq 2t\log n)\right)}},$$
with $t\in(0,r)$ to be chosen later. Since $T_n(t)$ has mean $0$ and variance $1$ under $H_0$, the test $\indc{|T_n(t)|>\sqrt{\log n}}$ has a vanishing Type-I error by applying Chebyshev's inequality. To analyze the Type-II error, we need to study the expectation and the variance of $T_n(t)$ under $H_1$. By Lemma \ref{lem:easy-tail}, we have
$$\mathbb{P}(\chi_{1,2r\log n}^2\leq 2t\log n)\asymp \frac{1}{\sqrt{\log n}}n^{-(\sqrt{r}-\sqrt{t})^2},\quad\text{and}\quad \mathbb{P}(\chi_1^2\leq 2t\log n)\asymp 1,$$
which implies
\begin{eqnarray*}
&& \mathbb{E}_{H_1}\left(\sum_{i=1}^n\indc{V_i^2\leq 2t\log n}-n\mathbb{P}(\chi_{1,2r\log n}^2\leq 2t\log n)\right) \\
&=& n\epsilon\left(\mathbb{P}(\chi_1^2\leq 2t\log n)-\mathbb{P}(\chi_{1,2r\log n}^2\leq 2t\log n)\right) \\
&\asymp& n^{1-\beta},
\end{eqnarray*}
and
\begin{eqnarray*}
&& \Var_{H_1}\left(\sum_{i=1}^n\indc{V_i^2\leq 2t\log n}-n\mathbb{P}(\chi_{1,2r\log n}^2\leq 2t\log n)\right) \\
&\asymp& n(1-\epsilon)\mathbb{P}(\chi_{1,2r\log n}^2\leq 2t\log n) + n\epsilon \mathbb{P}(\chi_1^2\leq 2t\log n) \\
&\asymp& \frac{1}{\sqrt{\log n}}n^{1-(\sqrt{r}-\sqrt{t})^2} + n^{1-\beta}.
\end{eqnarray*}
Hence,
\begin{equation}
\frac{(\mathbb{E}_{H_1}T_n(t))^2}{\Var_{H_1}(T_n(t))}\asymp n^{1-\beta} + (\sqrt{\log n})n^{1-2\beta+(\sqrt{r}-\sqrt{t})^2}.\label{eq:E-Var-ratio-V-equal}
\end{equation}
Therefore, as long as $1-\beta>0$ and $1-2\beta+r>0$, we can choose a sufficiently small constant $t\in(0,r)$, such that $\frac{(\mathbb{E}_{H_1}T_n(t))^2}{\Var_{H_1}(T_n(t))}$ diverges to infinity at some polynomial rate. This implies a vanishing Type-II error by Chebyshev's inequality. Finally, note that the conditions $1-\beta>0$ and $1-2\beta+r>0$ can be equivalently written as $\beta<1\wedge\frac{r+1}{2}$, and the proof is complete.
\end{proof}

\begin{proof}[Proof of Theorem \ref{thm:V-equal} (lower bound)]
Let us write the null distribution (\ref{eq:equal-sum-null}) and the alternative distribution (\ref{eq:equal-sum-alt}) as $P$ and $(1-\epsilon)P+\epsilon Q$, where the densities of $P$ and $Q$ are given by
$$p(v)=\frac{1}{2}\phi(v-\sqrt{2r\log n})+\frac{1}{2}\phi(v+\sqrt{2r\log n}),\quad\text{and}\quad q(v)=\phi(v).$$
We use $\phi(\cdot)$ for the density function of $N(0,1)$. It suffices to prove that $H(P,(1-\epsilon)P+\epsilon Q)^2=o(n^{-1})$ so that no consistent test exists (see \cite{cai2014optimal}). We adapt an argument in \cite{tony2011optimal} to upper bound the Hellinger distance. 
We have  {for any measurable set $D$},
\begin{eqnarray}
\nonumber H(P,(1-\epsilon)P+\epsilon Q)^2 &=& 1-\int\sqrt{p((1-\epsilon) p+\epsilon q)} \\
\nonumber &=& 1-\mathbb{E}_P\sqrt{1+\epsilon\left(\frac{q}{p}-1\right)} \\
\nonumber &\leq& 1-\mathbb{E}_P\sqrt{1+\epsilon\left(\frac{q}{p}\mathbf{1}_D-1\right)} \\
\label{eq:used-Jeng} &\leq& -\frac{1}{2}\epsilon\mathbb{E}_P\left(\frac{q}{p}\mathbf{1}_D-1\right) + \epsilon^2\mathbb{E}_P\left(\frac{q}{p}\mathbf{1}_D-1\right)^2 \\
\nonumber &\leq& \frac{1}{2}\epsilon\int_{D^c}q + 2\epsilon^2\int_D\frac{q^2}{p} + 2\epsilon^2,
\end{eqnarray}
where the inequality (\ref{eq:used-Jeng}) is by Lemma \ref{lem:Yihong-Wu} and the fact that $\epsilon\left(\frac{q}{p}\mathbf{1}_D-1\right)\geq -\epsilon\geq -1$. Since $\beta>1\wedge\frac{r+1}{2}$ implies $n\epsilon^2=n^{1-2\beta}=o(n^{-1})$, it suffices to prove
\begin{equation}
\epsilon\int_{D^c}q=o(n^{-1}),\quad\text{and}\quad \epsilon^2\int_D\frac{q^2}{p}=o(n^{-1}). \label{eq:two-conditions-D}
\end{equation}
To this end, we choose $D=\mathbb{R}$, which implies $\epsilon\int_{D^c}q=0$. For the second term, we have
\begin{eqnarray*}
\epsilon^2\int\frac{q^2}{p} &=& 4\epsilon^2\int_0^{\infty}\frac{\phi(v)^2}{\phi(v-\sqrt{2r\log n})+\phi(v+\sqrt{2r\log n})}dv \\
&\leq& 4\epsilon^2\int_0^{\infty}\frac{\phi(v)^2}{\phi(v-\sqrt{2r\log n})}dv \\
&=& 4\epsilon^2n^{2r}\int_0^{\infty}\phi(v+\sqrt{2r\log n})dv \\
&\leq& 4\epsilon^2n^{r},
\end{eqnarray*}
where the last inequality is a standard Gaussian tail bound (\ref{eq:Gaussian-tail}). Therefore, when $\beta>\frac{r+1}{2}$, we have $\epsilon^2\int\frac{q^2}{p}\leq 4\epsilon^2n^{r}=4n^{-2\beta+r}=o(n^{-1})$.

When $\beta>1$, we can choose $D=\varnothing$ in (\ref{eq:two-conditions-D}), which implies $\epsilon^2\int_D\frac{q^2}{p}=0$. Then we have $\epsilon^2\int_D\frac{q^2}{p}=\epsilon=n^{-\beta}=o(n^{-1})$.
The proof is completed by combining the two cases.
\end{proof}

\begin{proof}[Proof of Theorem \ref{thm:U-V-equal} (upper bound)]
Let us write the null distribution (\ref{eq:equal-comb-null}) and the alternative distribution (\ref{eq:equal-comb-alt}) by $P$ and $(1-\epsilon)P+\epsilon Q$, respectively, where the density functions of $P$ and $Q$ are given by
\begin{eqnarray}
\label{eq:P-up} p(u,v) &=& \frac{1}{2}\phi(u)\phi(v-\sqrt{2r\log n}) + \frac{1}{2}\phi(u)\phi(v+\sqrt{2r\log n}), \\
\label{eq:Q-up} q(u,v) &=& \frac{1}{2}\phi(u-\sqrt{2r\log n})\phi(v) + \frac{1}{2}\phi(u+\sqrt{2r\log n})\phi(v).
\end{eqnarray}
We only need to prove consistency of (\ref{eq:HC-comb}), and then consistency of the likelihood ratio test is implied by the Neyman--Pearson lemma. 
We can equivalently write the test (\ref{eq:HC-comb}) as $\indc{\sup_{t\in\mathbb{R}}|T_n(t)|>\sqrt{2(1+\delta)\log\log n}}$, where
$$T_n(t)=\frac{\sum_{i=1}^n\indc{|U_i|-|V_i|>t\sqrt{2\log n}}-nS_{\|\theta\|}(t\sqrt{2\log n})}{\sqrt{nS_{\|\theta\|}(t\sqrt{2\log n})(1-S_{\|\theta\|}(t\sqrt{2\log n}))}}.$$
By \cite{shorack2009empirical} and a standard argument in \cite{donoho2004higher}, $\frac{\sup_{t\in\mathbb{R}}|T_n(t)|}{\sqrt{2\log\log n}}$ converges to $1$ in probability under $H_0$, which then implies a vanishing Type-I error. The Type-II error can be bounded by
$$\mathbb{P}_{H_1}\left(\sup_{t\in\mathbb{R}}|T_n(t)|\leq\sqrt{2(1+\delta)\log\log n}\right)\leq \mathbb{P}_{H_1}\left(|T_n(\bar{t})|\leq\sqrt{2(1+\delta)\log\log n}\right),$$
for some $\bar{t}\in\mathbb{R}$ to be chosen appropriately. So it suffices to choose a $\bar{t}$ so that $\frac{(\mathbb{E}_{H_1}T_n(\bar{t}))^2}{\Var_{H_1}(T_n(\bar{t}))}$ diverges to infinity at  {some} polynomial rate. By Lemma \ref{lem:easy-tail}, we have
\begin{eqnarray*}
P\left(|U|-|V|>\bar{t}\sqrt{2\log n}\right) &\asymp& \frac{1}{\log n}n^{-(\bar{t}^2+r)}, \\
Q\left(|U|-|V|>\bar{t}\sqrt{2\log n}\right) &\asymp& \frac{1}{\log n}n^{-(\bar{t}-\sqrt{r})^2}, \\
\end{eqnarray*}
for any constant $\bar{t}>\sqrt{r}$. Therefore,
\begin{eqnarray*}
&& \mathbb{E}_{H_1}\left(\sum_{i=1}^n\indc{|U_i|-|V_i|>t\sqrt{2\log n}}-nS_{\|\theta\|}(t\sqrt{2\log n})\right) \\
&=& n\epsilon\left(Q\left(|U|-|V|>\bar{t}\sqrt{2\log n}\right)-P\left(|U|-|V|>\bar{t}\sqrt{2\log n}\right)\right) \\
&\asymp& \frac{1}{\log n}n^{1-\beta-(\bar{t}-\sqrt{r})^2},
\end{eqnarray*}
and
\begin{eqnarray*}
&& \Var_{H_1}\left(\sum_{i=1}^n\indc{|U_i|-|V_i|>t\sqrt{2\log n}}-nS_{\|\theta\|}(t\sqrt{2\log n})\right) \\
&\asymp& n(1-\epsilon)P\left(|U|-|V|>\bar{t}\sqrt{2\log n}\right) + n\epsilon Q\left(|U|-|V|>\bar{t}\sqrt{2\log n}\right) \\
&\asymp& \frac{1}{\log n}n^{1-(\bar{t}^2+r)} + \frac{1}{\log n}n^{1-\beta-(\bar{t}-\sqrt{r})^2},
\end{eqnarray*}
which implies
$$\frac{(\mathbb{E}_{H_1}T_n(\bar{t}))^2}{\Var_{H_1}(T_n(\bar{t}))}\asymp \frac{1}{\log n}\left(n^{1-2\beta-2(\bar{t}-\sqrt{r})^2+\bar{t}^2+r}+n^{1-\beta-(\bar{t}-\sqrt{r})^2}\right).$$
We choose $\bar{t}=2\sqrt{r}\wedge(\sqrt{r}+\sqrt{1-\beta^*(r)})$, and then $\frac{(\mathbb{E}_{H_1}T_n(\bar{t}))^2}{\Var_{H_1}(T_n(\bar{t}))}\rightarrow\infty$ at some polynomial rate as long as $\beta<\beta^*(r)$. This implies a vanishing Type-II error by Chebyshev's inequality. The proof is complete.
\end{proof}

\begin{proof}[Proof of Theorem \ref{thm:U-V-equal} (lower bound)]
Recall the notation $p$ and $q$ in (\ref{eq:P-up}) and (\ref{eq:Q-up}).
By the same argument  {as} in the proof of Theorem \ref{thm:V-equal} (lower bound), it suffices to show (\ref{eq:two-conditions-D}) for some set $D$. To this end, we choose $D=\{(u,v):|u|\leq (\sqrt{r}+\sqrt{1-\beta^*(r)})\sqrt{2\log n}\}$. Then,
\begin{eqnarray*}
\epsilon\int_{D^c}q &=& \epsilon \mathbb{P}\left(|N(\sqrt{2r\log n},1)|>(\sqrt{r}+\sqrt{1-\beta})\sqrt{2\log n}\right) \\
&\leq& 2\epsilon \mathbb{P}\left(N(0,1)>\sqrt{1-\beta^*(r)}\sqrt{2\log n}\right) \\
&\leq& 2\epsilon n^{-(1-\beta^*(r))},
\end{eqnarray*}
which implies $\epsilon\int_{D^c}q=o(n^{-1})$ when $\beta>\beta^*(r)$.
We also have
\begin{eqnarray}
\nonumber \epsilon^2\int_D\frac{q^2}{p} &=& 4\epsilon^2 \int_{D\cap\{(u,v):u>0,v>0\}}\frac{q(u,v)^2}{p(u,v)}dudv \\
\label{eq:bi-da-xiao} &\leq& 8\epsilon^2\int_{D\cap\{(u,v):u>0,v>0\}}\frac{\phi(u-\sqrt{2r\log n})^2\phi(v)^2}{\phi(u)\phi(v-\sqrt{2r\log n})}dudv \\
\nonumber &=& 8\epsilon^2\int_0^{(\sqrt{r}+\sqrt{1-\beta^*(r)})\sqrt{2\log n}}\frac{\phi(u-\sqrt{2r\log n})^2}{\phi(u)}du\int_0^{\infty}\frac{\phi(v)^2}{\phi(v-\sqrt{2r\log n})}dv \\
\nonumber &\leq& 8\epsilon^2n^{3r}\int_0^{(\sqrt{r}+\sqrt{1-\beta^*(r)})\sqrt{2\log n}}\phi(u-2\sqrt{2r\log n})du \\
\nonumber &=& 8\epsilon^2n^{3r}\mathbb{P}\left(N(0,1)\leq -(\sqrt{r}-\sqrt{1-\beta^*(r)})\sqrt{2\log n}\right) \\
\nonumber &\leq& \begin{cases}
8n^{-2\beta+3r-(\sqrt{r}-\sqrt{1-\beta^*(r)})^2}, & r>1-\beta^*(r), \\
8n^{-2\beta+3r}, & r\leq 1-\beta^*(r),
\end{cases} \\
\nonumber &=& \begin{cases}
8n^{-2\beta+3r-(\sqrt{r}-\sqrt{1-\beta^*(r)})^2}, & r>\frac{1}{5}, \\
8n^{-2\beta+3r}, & r\leq \frac{1}{5},
\end{cases}
\end{eqnarray}
where we have used the fact that $\phi(u-\sqrt{2r\log n})>\phi(u+\sqrt{2r\log n})$ when $u>0$ in (\ref{eq:bi-da-xiao}). When $r\leq \frac{1}{5}$, we have $-2\beta+3r< -2\beta^*(r)+3r=-1$. When $\frac{1}{5}<r< \frac{1}{2}$, we have
$-2\beta+3r-(\sqrt{r}-\sqrt{1-\beta^*(r)})^2<-2\beta^*(r)+3r-(\sqrt{r}-\sqrt{1-\beta^*(r)})^2\leq -1$ by the definition of $\beta^*(r)$. Therefore, we have $\epsilon^2\int_D\frac{q^2}{p}=o(n^{-1})$ and thus (\ref{eq:two-conditions-D}) holds whenever $r<\frac{1}{2}$. When $r\geq\frac{1}{2}$, we have $\beta^*(r)=1$, and we need to establish (\ref{eq:two-conditions-D}) for $\beta>1$. This can be done by choosing $D=\varnothing$ in (\ref{eq:two-conditions-D}), which implies $\epsilon^2\int_D\frac{q^2}{p}=0$. Then, when $\beta>1$, we have $\epsilon^2\int_D\frac{q^2}{p}=\epsilon=n^{-\beta}=o(n^{-1})$.
The proof is complete.
\end{proof}

\subsection{Proofs of Proposition \ref{prop:equal-diff}, Theorem \ref{thm:equal-sum}, and Theorem \ref{thm:main-equal}}

\begin{proof}[Proof of Proposition \ref{prop:equal-diff}]
We first bound the Type-I error. For any $z,\sigma\in\{-1,1\}^n$ such that $\ell(z,\sigma)=0$, we either have $z=\sigma$ or $z=-\sigma$. By a union bound, the Type-I error can be bounded from above by
\begin{equation}
\sup_{z\in\{-1,1\}^n}P_{(\theta,\eta,z,z)}^{(n)}\psi + \sup_{z\in\{-1,1\}^n}P_{(\theta,\eta,z,-z)}^{(n)}\psi.\label{eq:type-1-3.1}
\end{equation}
By the definition of $\psi$ in (\ref{eq:equal-test-diff}), the first term of (\ref{eq:type-1-3.1}) satisfies
\begin{equation}
\sup_{z\in\{-1,1\}^n}P_{(\theta,\eta,z,z)}^{(n)}\psi\leq \sup_{z\in\{-1,1\}^n}P_{(\theta,\eta,z,z)}^{(n)}\left(T_n^->\sqrt{2(1+\delta)\log \log n}\right)\rightarrow 0, \label{eq:type-1-3.1-part1}
\end{equation}
because $T_n^-/\sqrt{2\log\log n}\rightarrow 0$ in $P_{(\theta,\eta,z,z)}^{(n)}$-probability for any $\theta,\eta,z$ \citep{donoho2004higher,shorack2009empirical}. 
Similarly, for the second term in (\ref{eq:type-1-3.1}), we have
\begin{equation}
\sup_{z\in\{-1,1\}^n}P_{(\theta,\eta,z,-z)}^{(n)}\psi\leq \sup_{z\in\{-1,1\}^n}P_{(\theta,\eta,z,-z)}^{(n)}\left(T_n^+>\sqrt{2(1+\delta)\log \log n}\right)\rightarrow 0, \label{eq:type-1-3.1-part2}
\end{equation}
and thus the Type-I error is vanishing.

To analyze the Type-II error, we notice that by the definition of $\ell(z,z^*)$, we have
\begin{eqnarray}
\nonumber && \sup_{\substack{z\in\{-1,1\}^n\\\sigma\in\{-1,1\}^n\\\ell(z,\sigma)> \epsilon}}P^{(n)}_{(\theta,\eta,z,\sigma)}(1-\psi) \\
\nonumber &\leq& \sup_{\substack{z\in\{-1,1\}^n\\\sigma\in\{-1,1\}^n\\ \frac{1}{n}\sum_{i=1}^n\indc{z_i\neq \sigma_i}> \epsilon}}P^{(n)}_{(\theta,\eta,z,\sigma)}\left(T_n^-\leq\sqrt{2(1+\delta)\log \log n}\right) \\
\label{eq:type-2-3.1} && + \sup_{\substack{z\in\{-1,1\}^n\\\sigma\in\{-1,1\}^n\\ \frac{1}{n}\sum_{i=1}^n\indc{z_i\neq -\sigma_i}> \epsilon}}P^{(n)}_{(\theta,\eta,z,\sigma)}\left(T_n^+\leq\sqrt{2(1+\delta)\log \log n}\right).
\end{eqnarray}
By symmetry, the analyses of the two terms in the above are the same, and thus we only analyze the first term. For any $z,\sigma\in\{-1,1\}^n$ that satisfy $\frac{1}{n}\sum_{i=1}^n\indc{z_i\neq \sigma_i}>\epsilon$, we have
\begin{equation}
P^{(n)}_{(\theta,\eta,z,\sigma)}\left(T_n^-\leq\sqrt{2(1+\delta)\log \log n}\right)\leq P^{(n)}_{(\theta,\eta,z,\sigma)}\left(T_n^-((4r\wedge 1)2\log n)\leq\sqrt{2(1+\delta)\log \log n}\right), \label{eq:type-2-3.1-part}
\end{equation}
where for any $t {>0}$, we use the notation
$$T_n^-(t)=\frac{\left|\sum_{i=1}^n\indc{|\wt{X}_i-\wt{Y}_i|^2/2> 2t\log n}-n\mathbb{P}(\chi_1^2>2t\log n)\right|}{\sqrt{n\mathbb{P}(\chi_1^2>2t\log n)(1-\mathbb{P}(\chi_1^2>2t\log n))}}.$$
We can follow the same analysis in \cite{donoho2004higher,cai2014optimal} and show that
$$\frac{\Var(T_n^-((4r\wedge 1)))}{(\mathbb{E}T_n^-((4r\wedge 1)))^2}\rightarrow 0$$
at some polynomial rate as $n\rightarrow\infty$ whenever $\beta<\beta_{\rm IDJ}^*(r)$, where the variance and expectation above are under $P^{(n)}_{(\theta,\eta,z,\sigma)}$. This implies a vanishing Type-II error by Chebyshev's inequality, and thus the proof is complete.
\end{proof}

\begin{proof}[Proof of Theorem \ref{thm:equal-sum}]
The Type-I error is vanishing by the same arguments  {as} used in (\ref{eq:type-1-3.1})-(\ref{eq:type-1-3.1-part2}). For the Type-II error, we follow (\ref{eq:type-2-3.1}) and (\ref{eq:type-2-3.1-part}), and thus it suffices to prove
\begin{equation}
P^{(n)}_{(\theta,\eta,z,\sigma)}\left(\bar{T}_n^+(t)\leq\sqrt{2(1+\delta)\log \log n}\right)\rightarrow 0, \label{eq:want-to-prove-eq-sum}
\end{equation}
uniformly over any $z,\sigma\in\{-1,1\}^n$ that satisfy $\frac{1}{n}\sum_{i=1}^n\indc{z_i\neq \sigma_i}>\epsilon$. The $\bar{T}_n^+(t)$ in (\ref{eq:want-to-prove-eq-sum}) is defined by
$$\bar{T}_n^+(t)=\frac{\sum_{i=1}^n\indc{(\wt{X}_i+\wt{Y}_i)^2/2\leq 2t\log n}-\mathbb{P}(\chi_{1,2\|\theta\|^2}^2\leq 2t\log n)}{\sqrt{n\mathbb{P}(\chi_{1,2\|\theta\|^2}^2\leq 2t\log n)\left(1-\mathbb{P}(\chi_{1,2\|\theta\|^2}^2\leq 2t\log n)\right)}}.$$
The mean and variance of $\bar{T}_n^+(t)$ can be analyzed by following the same argument in the proof of Theorem \ref{thm:V-equal}, and thus we can obtain (\ref{eq:E-Var-ratio-V-equal}) with $T_n(t)$ replaced by $\bar{T}_n^+(t)$. With an appropriate choice of $t$ and an application of Chebyshev's inequality, we can show the Type-II error is vanishing, and thus the proof is complete.
\end{proof}

\begin{proof}[Proof of Theorem \ref{thm:main-equal}]
Similar to the proof of Theorem \ref{thm:equal-sum}, the upper bound conclusion directly follows the arguments used in the proofs of Proposition \ref{prop:equal-diff} and Theorem \ref{thm:U-V-equal}. Thus, we only prove the lower bound. For the first term of (\ref{eq:def-worst-risk}), we have
\begin{equation}
\sup_{\substack{z\in\{-1,1\}^n\\\sigma\in\{-1,1\}^n\\\ell(z,\sigma)=0}}P^{(n)}_{(\theta,\eta,z,\sigma)}\psi\geq \sup_{z\in  {\{-1,1\}^n} }P^{(n)}_{(\theta,\eta,z,z)}\psi. \label{eq:remove-label-switch}
\end{equation}
To analyze the second term of (\ref{eq:def-worst-risk}), we note that the condition $\beta>\beta^*(r)$ implies that there exists some small constant $\delta>0$ such that $\beta>\beta^*(r)+\delta$. We use the notation $\bar{\epsilon}=n^{-(\beta-\delta)}$ so that $\bar{\epsilon}>\epsilon$. Now we  {use $\wt{P}^{(n)}_{(\theta,\eta,\bar{\epsilon})}$ to denote a joint distribution over $z,\sigma,X,Y$, the sampling process of 
which
is described below}:
\begin{enumerate}
\item Draw $z_i$ uniformly from $\{-1,1\}$ independently over all $i\in[n]$.
\item Conditioning on $z$, draw $\sigma_i$ independently over all $i\in[n]$ so that $\sigma_i=z_i$ with probability $1-\bar{\epsilon}$ and $\sigma_i=-z_i$ with probability $\bar{\epsilon}$.
\item Conditioning on $z$ and $\sigma$, independently sample $X_i|z_i\sim N(z_i\theta,I_p)$ and $Y_i|\sigma_i\sim N(\sigma_i\eta,I_q)$ for all $i\in[n]$.
\end{enumerate}
Consider the event
$$G=\left\{\left|\sum_{i=1}^n\indc{z_i\neq \sigma_i}-n\bar{\epsilon}\right|\leq\sqrt{n\bar{\epsilon}\log n}\right\}.$$
We can check by Chebyshev's inequality that $\wt{P}^{(n)}_{(\theta,\eta,\bar{\epsilon})}(G^c)\rightarrow 0$ as $n\rightarrow\infty$. We also have
$$G\subset\left\{\frac{1}{n}\sum_{i=1}^n\indc{z_i\neq \sigma_i}\in(\epsilon,1-\epsilon)\right\},$$
as long as $n$ is sufficiently large by the definition of $\bar{\epsilon}$.
Now we can lower bound the second term of (\ref{eq:def-worst-risk}) by
\begin{eqnarray*}
\sup_{\substack{z\in\{-1,1\}^n\\\sigma\in\{-1,1\}^n\\\ell(z,\sigma)> \epsilon}}P^{(n)}_{(\theta,\eta,z,\sigma)}(1-\psi) &=& \sup_{\substack{z\in\{-1,1\}^n\\\sigma\in\{-1,1\}^n\\\frac{1}{n}\sum_{i=1}^n\indc{z_i\neq \sigma_i}\in(\epsilon,1-\epsilon)}}P^{(n)}_{(\theta,\eta,z,\sigma)}(1-\psi) \\
&\geq& \wt{P}^{(n)}_{(\theta,\eta,\bar{\epsilon})}(1-\psi)\indc{G} \\
&\geq& \wt{P}^{(n)}_{(\theta,\eta,\bar{\epsilon})}(1-\psi)-\wt{P}^{(n)}_{(\theta,\eta,\bar{\epsilon})}(G^c).
\end{eqnarray*}
Together with (\ref{eq:remove-label-switch}), this implies
\begin{eqnarray*}
&& \inf_{\psi}\left(\sup_{\substack{z\in\{-1,1\}^n\\\sigma\in\{-1,1\}^n\\\ell(z,\sigma)=0}}P^{(n)}_{(\theta,\eta,z,\sigma)}\psi+\sup_{\substack{z\in\{-1,1\}^n\\\sigma\in\{-1,1\}^n\\\ell(z,\sigma)> \epsilon}}P^{(n)}_{(\theta,\eta,z,\sigma)}(1-\psi)\right) \\
&\geq& \inf_{\psi}\left(\wt{P}^{(n)}_{(\theta,\eta,0)}\psi+\wt{P}^{(n)}_{(\theta,\eta,\bar{\epsilon})}(1-\psi)\right)-\wt{P}^{(n)}_{(\theta,\eta,\bar{\epsilon})}(G^c).
\end{eqnarray*}
Since the second term in the above bound is vanishing, it is sufficient to lower bound $\inf_{\psi}\left(\wt{P}^{(n)}_{(\theta,\eta,0)}\psi+\wt{P}^{(n)}_{(\theta,\eta,\bar{\epsilon})}(1-\psi)\right)$ by a constant. Define $U_i=\frac{1}{\sqrt{2}}(\theta^TX_i/\|\theta\|-\eta^TY_i/\|\eta\|)$, $V_i=\frac{1}{\sqrt{2}}(\theta^TX_i/\|\theta\|+\eta^TY_i/\|\eta\|)$, and $W_i=R^T\begin{pmatrix}
X_i \\
Y_i
\end{pmatrix}$ 
for all $i\in[n]$, where $R\in\mathbb{R}^{(p+q)\times (p+q-2)}$ is a matrix the columns of which form an orthonormal basis of $\mathbb{R}^{p+q}$ together with $\frac{1}{\sqrt{2}}\begin{pmatrix}
\theta/\|\theta\| \\
-\eta/\|\eta\|
\end{pmatrix}$ and $\frac{1}{\sqrt{2}}\begin{pmatrix}
\theta/\|\theta\| \\
\eta/\|\eta\|
\end{pmatrix}$.
We note that the distributions of $\{W_i\}_{i\in[n]}$ under $\wt{P}^{(n)}_{(\theta,\eta,0)}$ and $\wt{P}^{(n)}_{(\theta,\eta,\bar{\epsilon})}$ are the same.
Moreover, $\{W_i\}_{i\in[n]}$ is independent from both $\{U_i\}_{i\in [n]}$ and $\{V_i\}_{i\in[n]}$ under both $\wt{P}^{(n)}_{(\theta,\eta,0)}$ and $\wt{P}^{(n)}_{(\theta,\eta,\bar{\epsilon})}$.
Therefore,  {by the connection} between testing error and total variation distance, we have
\begin{eqnarray*}
&& \inf_{\psi}\left(\wt{P}^{(n)}_{(\theta,\eta,0)}\psi+\wt{P}^{(n)}_{(\theta,\eta,\bar{\epsilon})}(1-\psi)\right) \\
&=& 1-\frac{1}{2}\int |p_0(x,y)-p_1(x,y)| \\
&=& 1-\frac{1}{2}\int|p_0(u,v)p_0(w)-p_1(u,v)p_1(w)| \\
&=& 1-\frac{1}{2}|p_0(u,v)-p_1(u,v)|,
\end{eqnarray*}
where we abuse the notation $p_0$ and $p_1$ for the density functions of $X,Y,U,V,W$ under $\wt{P}^{(n)}_{(\theta,\eta,0)}$ and $\wt{P}^{(n)}_{(\theta,\eta,\bar{\epsilon})}$, respectively. The last equality above uses the fact that $p_0(w)=p_1(w)$. Note that $1-\frac{1}{2}|p_0(u,v)-p_1(u,v)|$ is exactly the testing error of (\ref{eq:equal-comb-null})-(\ref{eq:equal-comb-alt}) with $\epsilon$ replaced by $\bar{\epsilon}$. Since $\beta-\delta>\beta^*(r)$, we can apply Theorem \ref{thm:U-V-equal} and get
$$\liminf_{n\rightarrow\infty}\left(1-\frac{1}{2}|p_0(u,v)-p_1(u,v)|\right)>c,$$
for some constant $c>0$, and this completes the proof.
\end{proof}

\subsection{Proofs of Theorem \ref{thm:general-separate}, Theorem \ref{thm:general-HC}, and Theorem \ref{thm:main-general}}

To facilitate the proofs of these theorems, we first state and prove several propositions. Each solves a small optimization problem that will be used in the arguments.

\begin{proposition}\label{prop:opt1}
We have
$$\max_u\left(2\sqrt{r}|u|-\frac{u^2}{2}\right)=2r,$$
and the maximum is achieved at $|u|=2\sqrt{r}$.
\end{proposition}
\begin{proof}
Obvoius.
\end{proof}

\begin{proposition}\label{prop:opt2}
We have
$$\max_v\left(-2(\sqrt{r+s}-\sqrt{s})|v+\sqrt{r+s}|-\frac{v^2}{2}\right)=
\begin{cases}
-\frac{r+s}{2}, & 3s\leq r, \\
-2(\sqrt{r+s}-\sqrt{s})\sqrt{s}, & 3s> r.
\end{cases}$$
The maximum is achieved at $v=-\sqrt{r+s}$ and at $v=-2(\sqrt{r+s}-\sqrt{s})$ in the two cases respectively.
\end{proposition}
\begin{proof}
We write the objective function by $f(v)$, and then
$$f'(v)=\begin{cases}
2(\sqrt{r+s}-\sqrt{s})-v, & v<-\sqrt{r+s}, \\
-2(\sqrt{r+s}-\sqrt{s})-v, & v\geq-\sqrt{r+s}.
\end{cases}$$
It is easy to see that $f'(v)$ is a decreasing function, and it 
 {goes from positive to negative as its argument goes from $-\infty$ to $\infty$}. 
This implies that $f(v)$ is first increasing and then decreasing. When $3s<r$, we have
$$-\sqrt{r+s}>-2(\sqrt{r+s}-\sqrt{s}),$$
so that the point where $f(v)$ changes from increasing to decreasing is $-\sqrt{r+s}$, and thus the maximum is
$$f(-\sqrt{r+s})=-\frac{r+s}{2}.$$
When $3s\geq r$, we have
$$-\sqrt{r+s}\leq-2(\sqrt{r+s}-\sqrt{s}),$$
which implies that the point where $f(v)$ changes from increasing to decreasing is $-2(\sqrt{r+s}-\sqrt{s})$, and thus the maximum is
$$f(-2(\sqrt{r+s}-\sqrt{s}))=-2(\sqrt{r+s}-\sqrt{s})\sqrt{s}.$$
\end{proof}

\begin{proposition}\label{prop:opt3}
We have
$$\max_u\left(2\sqrt{r}|u|-u^2\right)=r,$$
and the maximum is achieved at $|u|=\sqrt{r}$.
\end{proposition}
\begin{proof}
Obvious.
\end{proof}

\begin{proposition}\label{prop:opt4}
We have
$$\max_v\left(-2(\sqrt{r+s}-\sqrt{s})|v+\sqrt{r+s}|-v^2\right)=-r,$$
and the maximum is achieved at $v=-(\sqrt{r+s}-\sqrt{s})$.
\end{proposition}
\begin{proof}
We write the objective function as $f(v)$, and then
$$f'(v)=\begin{cases}
2(\sqrt{r+s}-\sqrt{s})-2v, & v<-\sqrt{r+s}, \\
-2(\sqrt{r+s}-\sqrt{s})-2v, & v\geq-\sqrt{r+s}.
\end{cases}$$
Since $f'(v)$  {goes from positive to negative as its argument goes from $-\infty$ to $\infty$}, $f(v)$ is first increasing and then decreasing. The point where it changes from increasing to decreasing is at $v=-(\sqrt{r+s}-\sqrt{s})$, and thus the maximum is
$f(-(\sqrt{r+s}-\sqrt{s}))=-r$.
\end{proof}

\begin{proposition}\label{prop:opt5}
We have
\begin{eqnarray*}
&& \max_{u^2+v^2=1}\left(2\sqrt{r}|u|-2(\sqrt{r+s}-\sqrt{s})|v+\sqrt{r+s}|\right) \\
&=& \begin{cases}
2\sqrt{r}\sqrt{1-r-s}, & 2(1-r-s)(r+s-\sqrt{s}\sqrt{r+s})>r, \\
\Big[2\sqrt{2(r+s-\sqrt{s}\sqrt{r+s})} & \\
~~~ -2(r+s-\sqrt{s}\sqrt{r+s})\Big], & 2(1-r-s)(r+s-\sqrt{s}\sqrt{r+s})\leq r. 
\end{cases}
\end{eqnarray*}
\end{proposition}
\begin{proof}
The constraint $u^2+v^2=1$ implies $|u|=\sqrt{1-v^2}$. Then, we can equivalently write the optimization problem as
$$\max_{|v|\leq 1}\left(2\sqrt{r}\sqrt{1-v^2}-2(\sqrt{r+s}-\sqrt{s})|v+\sqrt{r+s}|\right).$$
Denote the above object function as $f(v)$, and we have
$$f'(v)=\begin{cases}
-2\sqrt{r}\frac{v}{\sqrt{1-v^2}} + 2(\sqrt{r+s}-\sqrt{s}), & v<-\sqrt{r+s}, \\
-2\sqrt{r}\frac{v}{\sqrt{1-v^2}} - 2(\sqrt{r+s}-\sqrt{s}), & v\geq-\sqrt{r+s}. \\
\end{cases}$$
We observe that $f'(v)$ is a decreasing function on $(-1,1)$. Moreover, $f'(v)$  {goes from positive to negative as its argument goes from $-\infty$ to $\infty$}. This implies $f(v)$ is first increasing and then decreasing on $(-1,1)$, and we just need to find the point that the derivative changes its sign. First, the solution to the equation
$$-2\sqrt{r}\frac{v}{\sqrt{1-v^2}} - 2(\sqrt{r+s}-\sqrt{s})=0$$
is
$$v=-\frac{\sqrt{r+s}-\sqrt{s}}{\sqrt{r+(\sqrt{r+s}-\sqrt{s})^2}}.$$
There are two cases. In the first case,
$$-\frac{\sqrt{r+s}-\sqrt{s}}{\sqrt{r+(\sqrt{r+s}-\sqrt{s})^2}}<-\sqrt{r+s},$$
which is equivalently to
$$2(1-r-s)(r+s-\sqrt{s}\sqrt{r+s})>r,$$
and then the $-\sqrt{r+s}$ is the point where $f'(v)$ changes its sign. Thus, the maximum is
$$f(-\sqrt{r+s})=2\sqrt{r}\sqrt{1-r-s}.$$
In the second case,
$$-\frac{\sqrt{r+s}-\sqrt{s}}{\sqrt{r+(\sqrt{r+s}-\sqrt{s})^2}}\geq-\sqrt{r+s},$$
which is equivalently to
$$2(1-r-s)(r+s-\sqrt{s}\sqrt{r+s})\leq r,$$
and then $-\frac{\sqrt{r+s}-\sqrt{s}}{\sqrt{r+(\sqrt{r+s}-\sqrt{s})^2}}$ is the point where $f'(v)$ changes its sign. Thus, the maximum is
$$f\left(-\frac{\sqrt{r+s}-\sqrt{s}}{\sqrt{r+(\sqrt{r+s}-\sqrt{s})^2}}\right)=2\sqrt{2(r+s-\sqrt{s}\sqrt{r+s})}-2(r+s-\sqrt{s}\sqrt{r+s}).$$
\end{proof}

\begin{proof}[Proof of Theorem \ref{thm:general-separate} (upper bound)]
Conclusion 1 is the result in \cite{donoho2004higher}. We only need to prove Conclusion 2. We define the statistic
$$T_n(t)=\frac{\sum_{i=1}^n\indc{|V_i|^2\leq 2t\log n}-n\mathbb{P}\left(\chi_{1,2(r+s)\log n}^2\leq 2t\log n\right)}{\sqrt{n\mathbb{P}\left(\chi_{1,2(r+s)\log n}^2\leq 2t\log n\right)\left(1-\mathbb{P}\left(\chi_{1,2(r+s)\log n}^2\leq 2t\log n\right)\right)}},$$
and then we can write the test as $\indc{\sup_{t>0}|T_n(t)|>\sqrt{2(1+\delta)\log\log n}}$. By \cite{shorack2009empirical}, $\frac{\sup_{t>0}|T_n(t)|}{\sqrt{2\log\log n}}$ converges to $1$ in probability under $H_0$, which then implies a vanishing Type-I error. The Type-II error can be bounded by
$$\mathbb{P}_{H_1}\left(\sup_{t>0}|T_n(t)|\leq\sqrt{2(1+\delta)\log\log n}\right)\leq \inf_{t>0}\mathbb{P}_{H_1}\left(|T_n(t)|\leq\sqrt{2(1+\delta)\log\log n}\right).$$
 {To control the rightside of the last display}, it suffices to show there exists some $t>0$ so that $\frac{(\mathbb{E}_{H_1}T_n(t))^2}{\Var_{H_1}(T_n(t))}$ diverges to infinity at a polynomial rate. By Lemma \ref{lem:easy-tail}, we have
\begin{eqnarray}
\label{eq:g-s-l81} \mathbb{P}\left(\chi_{1,2(r+s)\log n}^2\leq 2t\log n\right) &\asymp& \begin{cases}
\frac{1}{\sqrt{\log n}}n^{-(\sqrt{r+s}-\sqrt{t})^2}, & 0<t<r+s, \\
1, & t \geq r+s,
\end{cases} \\
\label{eq:g-s-l82} 
\mathrm{~~and~~}\mathbb{P}\left(\chi_{1,2s\log n}^2\leq 2t\log n\right) &\asymp& \begin{cases}
\frac{1}{\sqrt{\log n}}n^{-(\sqrt{s}-\sqrt{t})^2}, & 0<t<s, \\
1, & t \geq s.
\end{cases}
\end{eqnarray}
Since
$$\frac{(\mathbb{E}_{H_1}T_n(t))^2}{\Var_{H_1}(T_n(t))}\asymp \frac{n^2\epsilon^2\mathbb{P}\left(\chi_{1,2s\log n}^2\leq 2t\log n\right)^2}{n\mathbb{P}\left(\chi_{1,2(r+s)\log n}^2\leq 2t\log n\right) + n\epsilon \mathbb{P}\left(\chi_{1,2s\log n}^2\leq 2t\log n\right)},$$
we have $\frac{(\mathbb{E}_{H_1}T_n(t))^2}{\Var_{H_1}(T_n(t))}\rightarrow\infty$ is equivalent to $\frac{n\epsilon^2\mathbb{P}\left(\chi_{1,2s\log n}^2\leq 2t\log n\right)^2}{\mathbb{P}\left(\chi_{1,2(r+s)\log n}^2\leq 2t\log n\right)}\rightarrow\infty$ and $n\epsilon \mathbb{P}\left(\chi_{1,2s\log n}^2\leq 2t\log n\right)\rightarrow\infty$. By (\ref{eq:g-s-l81}) and (\ref{eq:g-s-l82}), we have
$$\frac{n\epsilon^2\mathbb{P}\left(\chi_{1,2s\log n}^2\leq 2t\log n\right)^2}{\mathbb{P}\left(\chi_{1,2(r+s)\log n}^2\leq 2t\log n\right)}\asymp \begin{cases}
\frac{1}{\sqrt{\log n}}n^{1-2\beta-2(\sqrt{s}-\sqrt{t})^2+(\sqrt{r+s}-\sqrt{t})^2}, & 0<t<s, \\
(\sqrt{\log n})n^{1-2\beta+(\sqrt{r+s}-\sqrt{t})^2}, & s\leq t<r+s, \\
n^{1-2\beta}, & t>r+s,
\end{cases}$$
and
$$n\epsilon \mathbb{P}\left(\chi_{1,2s\log n}^2\leq 2t\log n\right)\asymp \begin{cases}
\frac{1}{\sqrt{\log n}}n^{1-\beta-(\sqrt{s}-\sqrt{t})^2}, & 0<t<s, \\
n^{1-\beta}, & t \geq s.
\end{cases}
$$
Therefore, a sufficient condition for the existence of $t>0$ such that $\frac{(\mathbb{E}_{H_1}T_n(t))^2}{\Var_{H_1}(T_n(t))}$ diverges to infinity at a polynomial rate is that
\begin{equation}
\beta<\sup_{t\in T_1}(f_1(t)\wedge g_1(t)) \vee \sup_{t\in T_2}(1\wedge g_2(t)) \vee \frac{1}{2}, \label{eq:min-max-1/2-V}
\end{equation}
where
\begin{eqnarray*}
f_1(t) &=& 1-(\sqrt{s}-\sqrt{t})^2, \\
g_1(t) &=& \frac{1}{2}-(\sqrt{s}-\sqrt{t})^2 + \frac{1}{2}(\sqrt{r+s}-\sqrt{t})^2, \\
g_2(t) &=& \frac{1}{2}+(\sqrt{r+s}-\sqrt{t})^2,
\end{eqnarray*}
and $T_1=(0,s)$ and $T_2=[s,r+s)$. We need to show that $\beta<\bar{\beta}^*(r,s)$ is a sufficient condition for (\ref{eq:min-max-1/2-V}) by calculating $\sup_{t\in T_1}(f_1(t)\wedge g_1(t))$. Note that the maximizers of $f_1(t)$ and $g_1(t)$ are $\sqrt{t}=\sqrt{s}$ and $\sqrt{t}=2\sqrt{s}-\sqrt{r+s}$, respectively. Let us consider the following four cases.

\textit{Case 1.} $3s\leq r$ and $r+s\leq 1$. Since $r+s\leq 1$, we have $f_1(t)\wedge g_1(t)=g_1(t)$. Moreover, the condition $3s\leq r$ guarantees that $g_1(t)$ is decreasing on $T_1$. Therefore,
$$\sup_{t\in T_1}(f_1(t)\wedge g_1(t))=g_1(0)=\frac{1+r-s}{2}.$$

\textit{Case 2.} $3s>r$ and $(\sqrt{r+s}-\sqrt{s})^2\leq\frac{1}{4}$. Note that $f_1(t)\wedge g_1(t)=f_1(t)$ for $\sqrt{t}\leq \sqrt{r+s}-1$ and $f_1(t)\wedge g_1(t)=g_1(t)$ for $\sqrt{r+s}-1<\sqrt{t}<\sqrt{s}$. The condition $(\sqrt{r+s}-\sqrt{s})^2\leq\frac{1}{4}$ implies that $\sqrt{r+s}-1\leq 2\sqrt{s}-\sqrt{r+s}<\sqrt{s}$, and the condition $3s>r$ guarantees that $(2\sqrt{s}-\sqrt{r+s})^2\in T_1$. Therefore, $f_1(t)\wedge g_1(t)$ is increasing when $\sqrt{t}\leq 2\sqrt{s}-\sqrt{r+s}$ and decreasing when $2\sqrt{s}-\sqrt{r+s}<\sqrt{t}<\sqrt{s}$. We thus have
$$\sup_{t\in T_1}(f_1(t)\wedge g_1(t))=g_1((2\sqrt{s}-\sqrt{r+s})^2)=\frac{1}{2}+r-2\sqrt{s}(\sqrt{r+s}-\sqrt{s}).$$

\textit{Case 3.} $r+s>1$ and $\frac{1}{4}<(\sqrt{r+s}-\sqrt{s})^2\leq 1$. The condition $\frac{1}{4}<(\sqrt{r+s}-\sqrt{s})^2\leq 1$ implies that $2\sqrt{s}-\sqrt{r+s}<\sqrt{r+s}-1\leq\sqrt{s}$, and the condition $r+s>1$ guarantees that $\sqrt{r+s}-1\in (0,\sqrt{s}]$. Therefore, $f_1(t)\wedge g_1(t)=f_1(t)$ for $\sqrt{t}\leq \sqrt{r+s}-1$ and is increasing. We also have $f_1(t)\wedge g_1(t)=g_1(t)$ for $\sqrt{r+s}-1<\sqrt{t}<\sqrt{s}$ and is decreasing. Hence,
$$\sup_{t\in T_1}(f_1(t)\wedge g_1(t))=g_1((\sqrt{r+s}-1)^2)=r-2(\sqrt{r+s}-\sqrt{s})(\sqrt{r+s}-1).$$

\textit{Case 4.} $(\sqrt{r+s}-\sqrt{s})^2> 1$. This condition implies $\sqrt{s}<\sqrt{r+s}-1$, and thus $f_1(t)\wedge g_1(t)=f_1(t)$ for all $t\in T_1$, which leads to
$$\sup_{t\in T_1}(f_1(t)\wedge g_1(t))=f_1(s)=1.$$

Combine the four cases above, and we conclude that $\beta<\bar{\beta}^*(r,s)$ is a sufficient condition for (\ref{eq:min-max-1/2-V}), which completes the proof.
\end{proof}

\begin{proof}[Proof of Theorem \ref{thm:general-separate} (lower bound)]
Conclusion 1 is a result of \cite{ingster1997some}. We only need to prove Conclusion 2.
If we only use $\{V_i\}_{1\leq i\leq n}$, the testing problem (\ref{eq:general-comb-null})-(\ref{eq:general-comb-alt}) becomes
$$
H_0: V_i \stackrel{iid}{\sim} P, \quad i \in[n],\qquad H_1: V_i \stackrel{iid}{\sim} (1-\epsilon)P+\epsilon Q, \quad i\in[n],
$$
where the densities of $P$ and $Q$ are given by
$$p(v)=\frac{1}{2}\phi(v-\sqrt{2(r+s)\log n})+\frac{1}{2}\phi(v+\sqrt{2(r+s)\log n}),$$
and
$$q(v)=\frac{1}{2}\phi(v-\sqrt{2s\log n})+\frac{1}{2}\phi(v+\sqrt{2s\log n}).$$
Suppose $v\geq 0$, and then we have
$\phi(v-\sqrt{2(r+s)\log n})\geq \phi(v+\sqrt{2(r+s)\log n})$ and $\phi(v-\sqrt{2s\log n})\geq \phi(v+\sqrt{2s\log n})$. These two inequalities imply
$$p(v)\leq \phi(v-\sqrt{2(r+s)\log n})\leq 2p(v),$$
and
$$q(v)\leq \phi(v-\sqrt{2s\log n})\leq 2q(v).$$
Thus, we have
$$\frac{q(v)}{2p(v)}\leq e^{-(\sqrt{r+s}-\sqrt{s})v\sqrt{2\log n}}n^r\leq \frac{2q(v)}{p(v)},$$
 {which is} due to the fact that $\frac{\phi(v-\sqrt{2s\log n})}{\phi(v-\sqrt{2(r+s)\log n})}=e^{-(\sqrt{r+s}-\sqrt{s})v\sqrt{2\log n}}n^r$. 
 {By symmetry, we obtain that}
\begin{equation}
\frac{q(v)}{2p(v)}\leq e^{-(\sqrt{r+s}-\sqrt{s})|v|\sqrt{2\log n}}n^r\leq \frac{2q(v)}{p(v)}, \label{eq:lh-ratio-V}
\end{equation}
for all $v\in\mathbb{R}$.

Now we proceed to bound the Hellinger distance, and it is sufficient to show that $H(P,(1-\epsilon)P+\epsilon Q)^2=o(n^{-1})$. By direct calculation, we have
\begin{eqnarray}
\nonumber H(P,(1-\epsilon)P+\epsilon Q)^2 &=& \mathbb{E}_P\left(\sqrt{1+\epsilon\left(\frac{q}{p}-1\right)}-1\right)^2 \\
\label{eq:Hell-V1} &=& \mathbb{E}_P\left[\left(\sqrt{1+\epsilon\left(\frac{q}{p}-1\right)}-1\right)^2\indc{q\leq p}\right] \\
\label{eq:Hell-V2} && + \mathbb{E}_P\left[\left(\sqrt{1+\epsilon\left(\frac{q}{p}-1\right)}-1\right)^2\indc{ {q>p}}\right].
\end{eqnarray}
By Equation (88) of \cite{cai2014optimal}, the first term (\ref{eq:Hell-V1}) can be bounded by $n^{-2\beta}$, which is $o(n^{-1})$ as long as $\beta>\frac{1}{2}$. For (\ref{eq:Hell-V2}), we have
\begin{eqnarray}
\nonumber && \mathbb{E}_P\left[\left(\sqrt{1+\epsilon\left(\frac{q}{p}-1\right)}-1\right)^2\indc{ {q > p}}\right] \\
\nonumber &\leq& \mathbb{E}_P\left(\sqrt{1+\epsilon\frac{q}{p}}-1\right)^2 \\
\label{eq:He-bound-V-CW} &\leq& \mathbb{E}_{V\sim P}\left(\sqrt{1+2\epsilon e^{-(\sqrt{r+s}-\sqrt{s})| {V}|\sqrt{2\log n}}n^r}-1\right)^2,
\end{eqnarray}
where the last inequality uses (\ref{eq:lh-ratio-V}). Let us define the function
$$\alpha(v)=-2(\sqrt{r+s}-\sqrt{s})|v+\sqrt{r+s}|+r,$$
and then we can  {rewrite} (\ref{eq:He-bound-V-CW}) as $\mathbb{E}\left(\sqrt{1+2n^{-\beta+\alpha(V)}}-1\right)^2$, where $V\sim N(0,(2\log n)^{-1})$. By Lemma \ref{lem:Yihong-Wu}, we have
\begin{eqnarray}
\nonumber \mathbb{E}\left(\sqrt{1+2n^{-\beta+\alpha(V)}}-1\right)^2 &\leq& 4\mathbb{E}n^{(\alpha(V)-\beta)\wedge(2\alpha(V)-2\beta)} \\
\label{eq:intetral-V-lower} &=& 4\sqrt{\frac{\log n}{\pi}}\int n^{(\alpha(v)-\beta)\wedge(2\alpha(v)-2\beta)-v^2}dv.
\end{eqnarray}
Then, by Lemma \ref{lem:Laplace}, a sufficient condition for (\ref{eq:intetral-V-lower}) to be $o(n^{-1})$ is
\begin{equation}
\max_v\left[(\alpha(v)-\beta)\wedge(2\alpha(v)-2\beta)-v^2\right] < -1. \label{eq:half-baked}
\end{equation}

 {In the rest of this proof}, we show that condition (\ref{eq:half-baked})  {is equivalent to} $\beta>\bar{\beta}^*(r,s)$. First, we show (\ref{eq:half-baked}) is equivalent to
\begin{equation}
\beta > \frac{1}{2} + \max_v\left[\alpha(v)-v^2+\frac{v^2\wedge 1}{2}\right]. \label{eq:almost-done}
\end{equation}
Suppose (\ref{eq:almost-done}) is true. Then, for any $v\in\mathbb{R}$, either $\beta>\alpha(v)-v^2+\frac{v^2}{2}$, which is equivalent to
\begin{equation}
2\alpha(v)-2\beta-v^2<-1, \label{eq:eren}
\end{equation}
or $\beta>\alpha(v)-v^2+\frac{1}{2}$, which is equivalent to
\begin{equation}
\alpha(v) - \beta -v^2 < -1. \label{eq:armin}
\end{equation}
Since one of the two inequalities (\ref{eq:eren}) and (\ref{eq:armin}) must hold, we have $(\alpha(v)-\beta)\wedge(2\alpha(v)-2\beta)-v^2<-1$. Taking maximum over $v\in\mathbb{R}$, we obtain (\ref{eq:half-baked}). For the other direction, suppose (\ref{eq:half-baked}) is true. Then, for any $v\in\mathbb{R}$, we have either (\ref{eq:eren}) or (\ref{eq:armin}), which is equivalent to either $\beta>\alpha(v)-v^2+\frac{v^2}{2}$ or $\beta>\alpha(v)-v^2+\frac{1}{2}$. This implies $\beta>\frac{1}{2}+\alpha(v)-v^2+\frac{v^2\wedge 1}{2}$. Taking maximum over $v\in\mathbb{R}$, we obtain (\ref{eq:almost-done}). So we have established the equivalence between (\ref{eq:half-baked}) and (\ref{eq:almost-done}). To solve the righthand side of (\ref{eq:almost-done}), let us write
\begin{eqnarray*}
&& \frac{1}{2} + \max_v\left[\alpha(v)-v^2+\frac{v^2\wedge 1}{2}\right] \\
&=& \left(\frac{1}{2} + \max_{|v|\leq 1}\left[\alpha(v)-\frac{v^2}{2}\right]\right)\vee\left(1 + \max_{|v|\geq 1}\left[\alpha(v)-v^2\right]\right).
\end{eqnarray*}
By Proposition \ref{prop:opt2}, when $3s\leq r$ and $r+s\leq 1$, we have
$$\frac{1}{2} + \max_{|v|\leq 1}\left[\alpha(v)-\frac{v^2}{2}\right]=\frac{1+r-s}{2}.$$
When $3s>r$ and $4(\sqrt{r+s}-\sqrt{s})^2\leq 1$, we have
$$\frac{1}{2} + \max_{|v|\leq 1}\left[\alpha(v)-\frac{v^2}{2}\right]=\frac{1}{2}+r-2\sqrt{s}(\sqrt{r+s}-\sqrt{s}).$$
By Proposition \ref{prop:opt4}, when $(\sqrt{r+s}-\sqrt{s})^2>1$, we have
$$1 + \max_{|v|\geq 1}\left[\alpha(v)-v^2\right]=1.$$
Finally, we also have
\begin{eqnarray*}
 \frac{1}{2} + \max_{|v|= 1}\left[\alpha(v)-\frac{v^2}{2}\right] &=& 1 + \max_{|v|=1}\left[\alpha(v)-v^2\right] \\
&=& r-2(\sqrt{r+s}-\sqrt{s})(\sqrt{r+s}-1).
\end{eqnarray*}
After we properly organize the above cases, we obtain that
$$\frac{1}{2} + \max_v\left[\alpha(v)-v^2+\frac{v^2\wedge 1}{2}\right]=\bar{\beta}^*(r,s),$$
which implies (\ref{eq:almost-done}) is equivalent to $\beta>\bar{\beta}^*(r,s)$, and the proof is complete.
\end{proof}

\begin{proof}[Proof of Theorem \ref{thm:general-HC} (upper bound)]
Define the statistic
$$T_n(t)=\frac{\sum_{i=1}^n\indc{|\sqrt{r}U_i+\sqrt{s}V_i|-\sqrt{r+s}|V_i|>t\sqrt{2\log n}}-nS_{(r,s)}(t\sqrt{2\log n})}{\sqrt{n S_{(r,s)}(t\sqrt{2\log n})(1-S_{(r,s)}(t\sqrt{2\log n}))}},$$
and then we can write the test as $\indc{\sup_{t\in\mathbb{R}}|T_n(t)|>\sqrt{2(1+\delta)\log\log n}}$. By \cite{shorack2009empirical}, $\frac{\sup_{t\in\mathbb{R}}|T_n(t)|}{\sqrt{2\log\log n}}$ converges to $1$ in probability under $H_0$, which then implies a vanishing Type-I error. The Type-II error can be bounded by
$$\mathbb{P}_{H_1}\left(\sup_{t\in\mathbb{R}}|T_n(t)|\leq\sqrt{2(1+\delta)\log\log n}\right)\leq \min_{t\in\mathbb{R}}\mathbb{P}_{H_1}\left(|T_n(t)|\leq\sqrt{2(1+\delta)\log\log n}\right).$$
So it suffices to show there exists some $t$ so that $\frac{(\mathbb{E}_{H_1}T_n(t))^2}{\Var_{H_1}(T_n(t))}$ diverges to infinity at a polynomial rate. Let us write the null distribution (\ref{eq:general-comb-null}) and the alternative distribution (\ref{eq:general-comb-alt}) as $P$ and $(1-\epsilon)P+\epsilon Q$, respectively. Then, $\frac{(\mathbb{E}_{H_1}T_n(t))^2}{\Var_{H_1}(T_n(t))}$ is at the same order of
$$\frac{n^2\epsilon^2Q\left(|\sqrt{r}U+\sqrt{s}V|-\sqrt{r+s}|V|>t\sqrt{2\log n}\right)^2}{nP\left(|\sqrt{r}U+\sqrt{s}V|-\sqrt{r+s}|V|>t\sqrt{2\log n}\right) + n\epsilon Q\left(|\sqrt{r}U+\sqrt{s}V|-\sqrt{r+s}|V|>t\sqrt{2\log n}\right)}.$$
By Lemma \ref{lem:comp-tail-0} and Lemma \ref{lem:comp-tail-1}, we have
\begin{eqnarray*}
&& n\epsilon Q\left(|\sqrt{r}U+\sqrt{s}V|-\sqrt{r+s}|V|>t\sqrt{2\log n}\right) \\
&\asymp& \begin{cases}
\frac{1}{\log n}n^{1-\beta-\frac{(t-r)^2+rs}{r}}, &   t\in T_1, \\
\frac{1}{\sqrt{\log n}}n^{1-\beta-\frac{(t-(r+s)+\sqrt{s}\sqrt{r+s})^2}{2(r+s-\sqrt{s}\sqrt{r+s})}},  & t\in T_2, \\
n^{1-\beta},  & t\in T_3\cup T_4,
\end{cases}
\end{eqnarray*}
and
\begin{eqnarray*}
&& \frac{n\epsilon^2Q\left(|\sqrt{r}U+\sqrt{s}V|-\sqrt{r+s}|V|>t\sqrt{2\log n}\right)^2}{P\left(|\sqrt{r}U+\sqrt{s}V|-\sqrt{r+s}|V|>t\sqrt{2\log n}\right)} \\
&\asymp& \begin{cases}
\frac{1}{\log n}n^{1-2\beta-\frac{2(t-r)^2+2rs}{r}+\frac{t^2+r(r+s)}{r}}, & t\in T_1, \\
\frac{1}{\sqrt{\log n}}n^{1-2\beta-\frac{(t-r-s-\sqrt{s}\sqrt{r+s})^2}{r+s-\sqrt{s}\sqrt{r+s}}+\frac{(t+r+s-\sqrt{s}\sqrt{r+s})^2}{2(r+s-\sqrt{s}\sqrt{r+s})}},  & t\in T_2, \\
(\sqrt{\log n})n^{1-2\beta+\frac{(t+r+s-\sqrt{s}\sqrt{r+s})^2}{2(r+s-\sqrt{s}\sqrt{r+s})}},  & t\in T_3, \\
1, &   t\in T_4,
\end{cases}
\end{eqnarray*}
where
\begin{eqnarray*}
T_1 &=& (r+s+\sqrt{s}\sqrt{r+s},\infty), \\
T_2 &=& (r+s-\sqrt{s}\sqrt{r+s},r+s+\sqrt{s}\sqrt{r+s}], \\
T_3 &=& (-r-s+\sqrt{s}\sqrt{r+s},r+s-\sqrt{s}\sqrt{r+s}], \\
T_4 &=& (-\infty, -r-s+\sqrt{s}\sqrt{r+s}].
\end{eqnarray*}
Therefore, in order that there exists some $t$ such that $\frac{(\mathbb{E}_{H_1}T_n(t))^2}{\Var_{H_1}(T_n(t))}$ diverges to infinity at a polynomial rate, it is sufficient to require
\begin{equation}
\beta < \max_{t\in T_1}(f_1(t)\wedge g_1(t))\vee\max_{t\in T_2}(f_2(t)\wedge g_2(t))\vee\max_{t\in T_3}(1\wedge g_3(t))\vee\frac{1}{2}, \label{eq:general-upper-HC-abstract}
\end{equation}
where
\begin{eqnarray*}
f_1(t) &=& 1 -\frac{(t-r)^2+rs}{r}, \\
g_1(t) &=& \frac{1}{2} -\frac{(t-r)^2+rs}{r}+\frac{t^2+r(r+s)}{2r}, \\
f_2(t) &=& 1-\frac{(t-(r+s)+\sqrt{s}\sqrt{r+s})^2}{2(r+s-\sqrt{s}\sqrt{r+s})}, \\
g_2(t) &=& \frac{1}{2}-\frac{(t-(r+s)+\sqrt{s}\sqrt{r+s})^2}{2(r+s-\sqrt{s}\sqrt{r+s})} + \frac{(t+r+s-\sqrt{s}\sqrt{r+s})^2}{4(r+s-\sqrt{s}\sqrt{r+s})}, \\
g_3(t) &=& \frac{1}{2} + \frac{(t+r+s-\sqrt{s}\sqrt{r+s})^2}{4(r+s-\sqrt{s}\sqrt{r+s})}.
\end{eqnarray*}
Now we need to show that $\beta<\beta^*(r,s)$ is a sufficient condition of (\ref{eq:general-upper-HC-abstract}). According to the definition of $\beta^*(r,s)$, we will discuss the five cases respectively, and in each case, we will show $\beta^*(r,s)\leq \max_{t\in T_1}(f_1(t)\wedge g_1(t))\vee\max_{t\in T_2}(f_2(t)\wedge g_2(t))\vee\max_{t\in T_3}(1\wedge g_3(t))\vee\frac{1}{2}$.

\textit{Case 1.} $3s> r$ and $r+s-\sqrt{s}\sqrt{r+s}\leq \frac{1}{8}$. Note that $3s> r$ is equivalent to
$$3(r+s-\sqrt{s}\sqrt{r+s})< r+s+\sqrt{s}\sqrt{r+s},$$
and $r+s-\sqrt{s}\sqrt{r+s}\leq \frac{1}{8}$ is equivalent to
$$3(r+s-\sqrt{s}\sqrt{r+s})<\sqrt{2(r+s-\sqrt{s}\sqrt{r+s})}-(r+s-\sqrt{s}\sqrt{r+s}).$$
It is easy to see that $f_2(t)$ is a quadratic function maximized at $t=r+s-\sqrt{s}\sqrt{r+s}$, and $g_2(t)$ is a quadratic function maximized at $t=3(r+s-\sqrt{s}\sqrt{r+s})$.
Moreover, when $t\leq \sqrt{2(r+s-\sqrt{s}\sqrt{r+s})}-(r+s-\sqrt{s}\sqrt{r+s})$ and $t\in T_2$, $f_2(t)\wedge g_2(t)=g_2(t)$ achieves  {its} maximum at $t=3(r+s-\sqrt{s}\sqrt{r+s})$. 
When $t> \sqrt{2(r+s-\sqrt{s}\sqrt{r+s})}-(r+s-\sqrt{s}\sqrt{r+s})$ and $t\in T_2$, $f_2(t)\wedge g_2(t)=f_2(t)$ is decreasing. Therefore,
$$\max_{t\in T_2}(f_2(t)\wedge g_2(t))=g_2(3(r+s-\sqrt{s}\sqrt{r+s}))=\frac{1}{2}+2(r+s-\sqrt{s}\sqrt{r+s})=\beta^*(r,s),$$
and we can conclude that $\beta<\beta^*(r,s)$ implies (\ref{eq:general-upper-HC-abstract}).

\textit{Case 2.} $3s\leq r$ and $5r+s\leq 1$. Note that we can equivalently write the condition as $r+s+\sqrt{s}\sqrt{r+s}\leq 2r\leq \sqrt{r(1-r-s)}$. It can be checked that
$$f_1(t)\wedge g_1(t)=\begin{cases}
f_1(t), & t> \sqrt{r(1-r-s)},\\
g_1(t), & r+s+\sqrt{s}\sqrt{r+s}<t\leq \sqrt{r(1-r-s)}.
\end{cases}$$
Moreover, when $r+s+\sqrt{s}\sqrt{r+s}<t\leq \sqrt{r(1-r-s)}$, $g_1(t)$ is maximized at $t=2r$, and when $t> \sqrt{r(1-r-s)}$, $f_1(t)$ is decreasing. Therefore,
$$\max_{t\in T_1}(f_1(t)\wedge g_1(t))=g_1(2r)=\frac{1-s+3r}{2}=\beta^*(r,s),$$
and we can conclude that $\beta<\beta^*(r,s)$ implies (\ref{eq:general-upper-HC-abstract}).

\textit{Case 3.} $5r+s> 1$, $\frac{1}{8}<r+s-\sqrt{s}\sqrt{r+s}\leq \frac{1}{2}$ and $2(1-r-s)(r+s-\sqrt{s}\sqrt{r+s})>r$. Let us first show that the condition $2(1-r-s)(r+s-\sqrt{s}\sqrt{r+s})>r$ implies
$\sqrt{r(1-r-s)}>r+s+\sqrt{s}\sqrt{r+s}$. By $2(1-r-s)(r+s-\sqrt{s}\sqrt{r+s})>r$, we have
$$r+s<1-\frac{r}{2(r+s-\sqrt{s}\sqrt{r+s})}=\frac{r}{2(r+s+\sqrt{s}\sqrt{r+s})}.$$
The inequality $r+s<\frac{r}{2(r+s+\sqrt{s}\sqrt{r+s})}$ can be rearranged into
$$\frac{r}{2(r+s-\sqrt{s}\sqrt{r+s})}>\frac{(r+s+\sqrt{s}\sqrt{r+s})^2}{r}.$$
Then, from $2(1-r-s)(r+s-\sqrt{s}\sqrt{r+s})>r$, we get
$$1-r-s>\frac{r}{2(r+s-\sqrt{s}\sqrt{r+s})}>\frac{(r+s+\sqrt{s}\sqrt{r+s})^2}{r},$$
which leads to the desired inequality $\sqrt{r(1-r-s)}>r+s+\sqrt{s}\sqrt{r+s}$. Moreover, we can easily check that the condition $5r+s> 1$ implies $\sqrt{r(1-r-s)}<2r$, and thus we have
\begin{equation}
r+s+\sqrt{s}\sqrt{r+s}<\sqrt{r(1-r-s)}<2r. \label{eq:greatly-simplified}
\end{equation}
Now we analyze $\max_{t\in T_1}(f_1(t)\wedge g_1(t))$ under (\ref{eq:greatly-simplified}). Since $f_1(t)$ is a quadratic function that achieves maximum at $t=r$ and $g_1(t)$ is a quadratic function that achieves maximum at $t=2r$, we know that when $r+s+\sqrt{s}\sqrt{r+s}<t\leq \sqrt{r(1-r-s)}$, $f_1(t)\wedge g_1(t)=g_1(t)$ is increasing, and when $t>\sqrt{r(1-r-s)}$, $f_1(t)\wedge g_1(t)=f_1(t)$ is decreasing. Thus,
$$\max_{t\in T_1}(f_1(t)\wedge g_1(t))=f_1(\sqrt{r(1-r-s)})=2\sqrt{r}\sqrt{1-r-s}=\beta^*(r,s),$$
and we can conclude that $\beta<\beta^*(r,s)$ implies (\ref{eq:general-upper-HC-abstract}).

\textit{Case 4.} $5r+s> 1$, $\frac{1}{8}<r+s-\sqrt{s}\sqrt{r+s}\leq \frac{1}{2}$ and $2(1-r-s)(r+s-\sqrt{s}\sqrt{r+s})\leq r$. Let us first show that the condition $2(1-r-s)(r+s-\sqrt{s}\sqrt{r+s})\leq r$ implies
\begin{equation}
r+s+\sqrt{s}\sqrt{r+s} \geq \sqrt{2(r+s-\sqrt{s}\sqrt{r+s})}-(r+s-\sqrt{s}\sqrt{r+s}). \label{eq:happy-to-figure-out}
\end{equation}
By $2(1-r-s)(r+s-\sqrt{s}\sqrt{r+s})\leq r$, we have
\begin{eqnarray*}
r+s &\geq& 1-\frac{r}{2(r+s-\sqrt{s}\sqrt{r+s})} \\
&=& \frac{r}{2(r+s+\sqrt{s}\sqrt{r+s})} \\
&=& \frac{1}{2}\left(1-\sqrt{\frac{s}{r+s}}\right),
\end{eqnarray*}
and thus we have $r+s\geq \frac{1}{2}\left(1-\sqrt{\frac{s}{r+s}}\right)$. 
Multiply both sides by $4(r+s)$ to obtain $4(r+s)^2\geq 2(r+s-\sqrt{s}\sqrt{r+s})$. 
We then  {take the} square roots  {of} both sides  {to obtain} $2(r+s)\geq \sqrt{2(r+s-\sqrt{s}\sqrt{r+s})}$, which can then be rearranged into (\ref{eq:happy-to-figure-out}). In addition to (\ref{eq:happy-to-figure-out}), by $\frac{1}{8}<r+s-\sqrt{s}\sqrt{r+s}\leq \frac{1}{2}$, we also have
\begin{equation}
r+s-\sqrt{s}\sqrt{r+s}\leq \sqrt{2(r+s-\sqrt{s}\sqrt{r+s})}-(r+s-\sqrt{s}\sqrt{r+s}) < 3(r+s-\sqrt{s}\sqrt{r+s}). \label{eq:this-one-is-easy}
\end{equation}
Now we will analyze $\max_{t\in T_2}(f_2(t)\wedge g_2(t))$ under (\ref{eq:happy-to-figure-out}) and (\ref{eq:this-one-is-easy}). Since $f_2(t)$ is a quadratic function that achieves maximum at $t=r+s-\sqrt{s}\sqrt{r+s}$ and $g_2(t)$ is a quadratic function that achieves maximum at $t=3(r+s-\sqrt{s}\sqrt{r+s})$, by (\ref{eq:happy-to-figure-out}) and (\ref{eq:this-one-is-easy}), we know that for $t\in T_2$, $f_2(t)\wedge g_2(t)=g_2(t)$ is increasing on the left hand side of $\sqrt{2(r+s-\sqrt{s}\sqrt{r+s})}-(r+s-\sqrt{s}\sqrt{r+s})$, and $f_2(t)\wedge g_2(t)=f_2(t)$ is decreasing on the righthand side of  {it}.
Therefore,
\begin{eqnarray*}
\max_{t\in T_2}(f_2(t)\wedge g_2(t)) &=& g_2\left(\sqrt{2(r+s-\sqrt{s}\sqrt{r+s})}-(r+s-\sqrt{s}\sqrt{r+s})\right) \\
&=& 2\sqrt{2(r+s-\sqrt{s}\sqrt{r+s})}-2(r+s-\sqrt{s}\sqrt{r+s}) \\
&=& \beta^*(r,s),
\end{eqnarray*}
and we can conclude that $\beta<\beta^*(r,s)$ implies (\ref{eq:general-upper-HC-abstract}).

\textit{Case 5.} $r+s-\sqrt{s}\sqrt{r+s}>\frac{1}{2}$. In this case, we have
$$\max_{t\in T_3}(1\wedge g_3(t))=1\wedge \min_{t\in T_3}g_3(t)=1\wedge g_3(r+s-\sqrt{s}\sqrt{r+s})=1,$$
under the condition $r+s-\sqrt{s}\sqrt{r+s}>\frac{1}{2}$. Since we also have $\beta^*(r,s)=1$, we can conclude that $\beta<\beta^*(r,s)$ implies (\ref{eq:general-upper-HC-abstract}).

Combining the results of the five cases above, we conclude that $\beta<\beta^*(r,s)$ is a sufficient condition for (\ref{eq:general-upper-HC-abstract}), and thus the proof is complete.
\end{proof}

\begin{proof}[Proof of Theorem \ref{thm:general-HC} (lower bound)]
Let us write the null distribution (\ref{eq:general-comb-null}) and the alternative distribution (\ref{eq:general-comb-alt}) as $P$ and $(1-\epsilon)P+\epsilon Q$, where the densities of $P$ and $Q$ are given by
$p(u,v)$ and $q(u,v)$ defined in (\ref{eq:general-P-den}) and (\ref{eq:general-Q-den}). It is sufficient to show that $H(P,(1-\epsilon)P+\epsilon Q)^2=o(n^{-1})$. By the same argument used in the lower bound proof of Theorem \ref{thm:general-separate}, $H(P,(1-\epsilon)P+\epsilon Q)^2$ can be written as the sum of (\ref{eq:Hell-V1}) and (\ref{eq:Hell-V2}). 
By Equation (88) of \cite{cai2014optimal}, the first term (\ref{eq:Hell-V1}) can be bounded by $n^{-2\beta}$, which is $o(n^{-1})$ as long as $\beta>\frac{1}{2}$. For (\ref{eq:Hell-V2}), we have
\begin{eqnarray}
\nonumber && \mathbb{E}_P\left[\left(\sqrt{1+\epsilon\left(\frac{q}{p}-1\right)}-1\right)^2\indc{q< p}\right] \\
\nonumber &\leq& \mathbb{E}_P\left(\sqrt{1+\epsilon\frac{q}{p}}-1\right)^2 \\
\label{eq:He-bound-U-V-CW} &\leq& \mathbb{E}_{(U,V)\sim P}\left(\sqrt{1+2\epsilon e^{(|\sqrt{r}U+\sqrt{s}V|-\sqrt{r+s}|V|)\sqrt{2\log n}}}-1\right)^2,
\end{eqnarray}
where the last inequality is by Lemma \ref{lem:LR-approx}. Let us define the function
$$\alpha(u,v)=2|\sqrt{r}u+\sqrt{s}v+\sqrt{s}\sqrt{r+s}|-2\sqrt{r+s}|v+\sqrt{r+s}|,$$
and then we can write (\ref{eq:He-bound-U-V-CW}) as $\mathbb{E}\left(\sqrt{1+2n^{-\beta+\alpha(U,V)}}-1\right)^2$, where $U,V\stackrel{iid}{\sim} N(0,(2\log n)^{-1})$. By Lemma \ref{lem:Yihong-Wu}, we have
\begin{eqnarray*}
\mathbb{E}\left(\sqrt{1+2n^{-\beta+\alpha(U,V)}}-1\right)^2 &\leq& 4\mathbb{E}n^{(\alpha(U,V)-\beta)\wedge(2\alpha(U,V)-2\beta)} \\
\label{eq:intetral-U-V-lower} &=& \frac{4\log n}{\pi}\int\int n^{(\alpha(u,v)-\beta)\wedge(2\alpha(u,v)-2\beta)-u^2-v^2}dudv.
\end{eqnarray*}
Then, by Lemma \ref{lem:Laplace}, a sufficient condition for (\ref{eq:He-bound-U-V-CW}) to be $o(n^{-1})$ is
\begin{equation}
\max_{u,v}\left[(\alpha(u,v)-\beta)\wedge(2\alpha(u,v)-2\beta)-u^2-v^2\right] < -1. \label{eq:half-baked-U-V}
\end{equation}
By the same argument that leads to the equivalence between (\ref{eq:half-baked}) and (\ref{eq:almost-done}), (\ref{eq:half-baked-U-V}) is also equivalent to
\begin{equation}
\beta > \frac{1}{2} + \max_{u,v}\left[\alpha(u,v)-u^2-v^2+\frac{(u^2+v^2)\wedge 1}{2}\right]. \label{eq:almost-done-U-V}
\end{equation}
We also define
$$\bar{\alpha}(u,v)=2\sqrt{r}|u|-2(\sqrt{r+s}-\sqrt{s})|v+\sqrt{r+s}|.$$
Since $\alpha(u,v)\leq \bar{\alpha}(u,v)$, a sufficient condition of (\ref{eq:almost-done-U-V}) is
\begin{equation}
\beta > \frac{1}{2} + \max_{u,v}\left[\bar{\alpha}(u,v)-u^2-v^2+\frac{(u^2+v^2)\wedge 1}{2}\right]. \label{eq:almost-done-U-V-bar}
\end{equation}
Now it suffices to show (\ref{eq:almost-done-U-V-bar}) is equivalent to $\beta>\beta^*(r,s)$. To solve the righthand side of (\ref{eq:almost-done-U-V-bar}), we write
\begin{eqnarray}
\nonumber && \frac{1}{2} + \max_{u,v}\left[\bar{\alpha}(u,v)-u^2-v^2+\frac{(u^2+v^2)\wedge 1}{2}\right] \\
\label{eq:how-to-combine} &=& \left(\frac{1}{2}+\max_{u^2+v^2\leq 1}\left[\bar{\alpha}(u,v)-\frac{u^2+v^2}{2}\right]\right)\vee\left(1+\max_{u^2+v^2\geq 1}\left[\bar{\alpha}(u,v)-u^2-v^2\right]\right).
\end{eqnarray}
By Proposition \ref{prop:opt1} and Proposition \ref{prop:opt2}, when $3s>r$ and $r+s-\sqrt{s}\sqrt{r+s}\leq\frac{1}{8}$, we have
$$\frac{1}{2}+\max_{u^2+v^2\leq 1}\left[\bar{\alpha}(u,v)-\frac{u^2+v^2}{2}\right]=\frac{1}{2}+2(r+s-\sqrt{s}\sqrt{r+s}),$$
and when $3\leq r$ and $5r+s\leq 1$, we have
$$\frac{1}{2}+\max_{u^2+v^2\leq 1}\left[\bar{\alpha}(u,v)-\frac{u^2+v^2}{2}\right]=\frac{1+3r-s}{2}.$$
We then use Proposition \ref{prop:opt3} and Proposition \ref{prop:opt4}. When $r+s-\sqrt{s}\sqrt{r+s}>\frac{1}{2}$, we have
$$1+\max_{u^2+v^2\geq 1}\left[\bar{\alpha}(u,v)-u^2-v^2\right]=1.$$
Finally, by Proposition \ref{prop:opt5}, when $5r+s> 1$, $\frac{1}{8}<r+s-\sqrt{s}\sqrt{r+s}\leq \frac{1}{2}$ and $2(1-r-s)(r+s-\sqrt{s}\sqrt{r+s})>r$, we have
$$\max_{u^2+v^2=1}\bar{\alpha}(u,v)=2\sqrt{r}\sqrt{1-r-s},$$
and when $5r+s> 1$, $\frac{1}{8}<r+s-\sqrt{s}\sqrt{r+s}\leq \frac{1}{2}$ and $2(1-r-s)(r+s-\sqrt{s}\sqrt{r+s})\leq r$, we have
$$\max_{u^2+v^2=1}\bar{\alpha}(u,v)=2\sqrt{2(r+s-\sqrt{s}\sqrt{r+s})}-2(r+s-\sqrt{s}\sqrt{r+s}).$$
Combine the above cases through (\ref{eq:how-to-combine}), and we obtain that
$$\frac{1}{2} + \max_{u,v}\left[\bar{\alpha}(u,v)-u^2-v^2+\frac{(u^2+v^2)\wedge 1}{2}\right]=\beta^*(r,s),$$
and therefore, $\beta>\beta^*(r,s)$ implies (\ref{eq:almost-done-U-V}),
which completes the proof.
\end{proof}

\begin{proof}[Proof of Theorem \ref{thm:main-general}]
Similar to the proof of Theorem \ref{thm:equal-sum}, the upper bound conclusion directly follows the arguments used in the proofs of Proposition \ref{prop:equal-diff} and Theorem \ref{thm:general-HC}. For the lower bound, we can use the same argument in the proof of Theorem \ref{thm:main-equal} to reduce to the testing problem (\ref{eq:general-comb-null})-(\ref{eq:general-comb-alt}) with $\epsilon$ replaced by $n^{-(\beta-\delta)}$ for some sufficiently small constant $\delta>0$ that satisfies $\delta<\beta-\beta^*(r,s)$. Then, apply Theorem \ref{thm:general-HC}, and we obtain the desired conclusion.
\end{proof}

\subsection{Proofs of Theorem \ref{thm:exact} and Theorem \ref{thm:Bonf}}

Let us first state a proposition on the properties of $t^*(r,s)$ defined in Section \ref{sec:equality}.

\begin{proposition}\label{prop:t-star}
The following properties of $t^*(r,s)$ are satisfied.
\begin{enumerate}
\item $t^*(r,s)\leq  \sqrt{2(r+s-\sqrt{s}\sqrt{r+s})}-(r+s-\sqrt{s}\sqrt{r+s})$.
\item If $r+s-\sqrt{s}\sqrt{r+s}>\frac{1}{2}$, then $t^*(r,s)<r+s-\sqrt{s}\sqrt{r+s}$.
\item If $r+s-\sqrt{s}\sqrt{r+s}<\frac{1}{2}$, then $t^*(r,s)>r+s-\sqrt{s}\sqrt{r+s}$.
\item If $t^*(r,s)\geq 3(r+s-\sqrt{s}\sqrt{r+s})$, then $r+s-\sqrt{s}\sqrt{r+s}\leq\frac{1}{8}$.
\end{enumerate}
\end{proposition}
\begin{proof}
Define
$$f(t)=\begin{cases}
f_1(t), & t>r+s+\sqrt{s}\sqrt{r+s}, \\
f_2(t), & -(r+s)+\sqrt{s}\sqrt{r+s} < t\leq r+s+\sqrt{s}\sqrt{r+s},
\end{cases}$$
where $f_1(t)=\frac{t^2+r(r+s)}{r}$ and $f_2(t)=\frac{(t+r+s-\sqrt{s}\sqrt{r+s})^2}{2(r+s-\sqrt{s}\sqrt{r+s})}$. Then, $t^*(r,s)$ is a solution to the equation $f(t)=1$. In particular, $\sqrt{r(1-r-s)}$ is a solution to $f_1(t)=1$ and $\sqrt{2(r+s-\sqrt{s}\sqrt{r+s})}-(r+s-\sqrt{s}\sqrt{r+s})$ is a solution to $f_2(t)=1$. It is not hard to check that $f(t)$ is an increasing function, and the condition $2(r+s)(r+s+\sqrt{s}\sqrt{r+s})\leq r$ is equivalent to $f(r+s+\sqrt{s}\sqrt{r+s})\leq 1$. To prove the first conclusion, it suffices to show $\sqrt{r(1-r-s)}\leq \sqrt{2(r+s-\sqrt{s}\sqrt{r+s})}-(r+s-\sqrt{s}\sqrt{r+s})$ when $f(r+s+\sqrt{s}\sqrt{r+s})\leq 1$. It is direct to check that $f_1(r+s+\sqrt{s}\sqrt{r+s})=f_2(r+s+\sqrt{s}\sqrt{r+s})$ and $f_1'(r+s+\sqrt{s}\sqrt{r+s})=f_2'(r+s+\sqrt{s}\sqrt{r+s})$. Moreover, since $r+s-\sqrt{s}\sqrt{r+s}=\frac{\sqrt{r+s}}{\sqrt{r+s}+\sqrt{s}}r\geq\frac{r}{2}$, we have $f_1''(t)=\frac{1}{r}\geq \frac{1}{2(r+s-\sqrt{s}\sqrt{r+s})}=f_2''(t)$. This means $f_1$ grows faster than $f_2$ for $t\geq r+s+\sqrt{s}\sqrt{r+s}$, and thus will $f_1(t)$ reach $1$ no later than $f_2(t)$ does, which implies $\sqrt{r(1-r-s)}\leq \sqrt{2(r+s-\sqrt{s}\sqrt{r+s})}-(r+s-\sqrt{s}\sqrt{r+s})$.

To prove the second conclusion, we notice that $t^*(r,s)=\sqrt{r(1-r-s)}$ when
$r\geq 2(r+s)(r+s+\sqrt{s}\sqrt{r+s})$.
This condition can be equivalently written as
\begin{equation}
2(1-r-s)(r+s-\sqrt{s}\sqrt{r+s})\geq r, \label{eq:oh-again}
\end{equation}
because of the identity $\frac{r}{2(r+s+\sqrt{s}\sqrt{r+s})}=1-\frac{r}{2(r+s-\sqrt{s}\sqrt{r+s})}$. Following the same way that we derive (\ref{eq:happy-to-figure-out}) from $2(1-r-s)(r+s-\sqrt{s}\sqrt{r+s})\leq r$, we can also show that (\ref{eq:oh-again}) implies
\begin{equation}
r+s+\sqrt{s}\sqrt{r+s} \leq \sqrt{2(r+s-\sqrt{s}\sqrt{r+s})}-(r+s-\sqrt{s}\sqrt{r+s}). \label{eq:to-be-contradicted}
\end{equation}
However, under the condition that $r+s-\sqrt{s}\sqrt{r+s}>\frac{1}{2}$, we have
\begin{equation}
\sqrt{2(r+s-\sqrt{s}\sqrt{r+s})}-(r+s-\sqrt{s}\sqrt{r+s})< r+s-\sqrt{s}\sqrt{r+s}, \label{eq:tired}
\end{equation}
which contradicts (\ref{eq:to-be-contradicted}). Therefore, the condition $r+s-\sqrt{s}\sqrt{r+s}>\frac{1}{2}$ implies that we can only have
\begin{equation}
t^*(r,s)=\sqrt{2(r+s-\sqrt{s}\sqrt{r+s})}-(r+s-\sqrt{s}\sqrt{r+s}). \label{eq:t-star-exact-recov}
\end{equation}
Hence, $t^*(r,s)< r+s-\sqrt{s}\sqrt{r+s}$ holds because of (\ref{eq:tired}) and (\ref{eq:t-star-exact-recov}).

Now we prove the third conclusion. When $2(r+s)(r+s+\sqrt{s}\sqrt{r+s})>r$, since $t^*(r,s)=\sqrt{2(r+s-\sqrt{s}\sqrt{r+s})}-(r+s-\sqrt{s}\sqrt{r+s})$, $t^*(r,s)>r+s-\sqrt{s}\sqrt{r+s}$ is equivalent to $r+s-\sqrt{s}\sqrt{r+s}<\frac{1}{2}$. Now consider the other case $2(r+s)(r+s+\sqrt{s}\sqrt{r+s})\leq r$, and then $t^*(r,s)=\sqrt{r(1-r-s)}$. Note that the condition $2(r+s)(r+s+\sqrt{s}\sqrt{r+s})\leq r$ can be equivalently written as (\ref{eq:oh-again}), and thus we have $1-r-s\geq \frac{r}{2(r+s-\sqrt{s}\sqrt{r+s})}$. Then, by $r+s-\sqrt{s}\sqrt{r+s}<\frac{1}{2}$, we get $1-r-s>r$, which then implies $\sqrt{r(1-r-s)}>r$. Since $r+s-\sqrt{s}\sqrt{r+s}=\sqrt{r+s}(\sqrt{r+s}-\sqrt{s})=\frac{\sqrt{r+s}}{\sqrt{r+s}+\sqrt{s}}r\leq r$, we get $\sqrt{r(1-r-s)}>r+s-\sqrt{s}\sqrt{r+s}$ as desired. So we conclude that $t^*(r,s)>r+s-\sqrt{s}\sqrt{r+s}$ holds.

Finally, by the first conclusion, $t^*(r,s)\geq 3(r+s-\sqrt{s}\sqrt{r+s})$ implies
$$\sqrt{2(r+s-\sqrt{s}\sqrt{r+s})}-(r+s-\sqrt{s}\sqrt{r+s})\geq 3(r+s-\sqrt{s}\sqrt{r+s}).$$
This directly leads to $r+s-\sqrt{s}\sqrt{r+s}\leq\frac{1}{8}$.
\end{proof}

\begin{proof}[Proof of Theorem \ref{thm:exact} (upper bound)]
For the test $\dot{\psi}$ defined in (\ref{eq:general-test-comb}), its Type-I error is vanishing by the same arguments used in (\ref{eq:type-1-3.1})-(\ref{eq:type-1-3.1-part2}). For the Type-II error, we follow (\ref{eq:type-2-3.1}) and (\ref{eq:type-2-3.1-part}), and thus it suffices to prove
\begin{equation}
P^{(n)}_{(\theta,\eta,z,\sigma)}\left(\dot{T}_n^-(t^*(r,s))\leq\sqrt{2(1+\delta)\log \log n}\right)\rightarrow 0, \label{eq:want-to-prove-exact-equal}
\end{equation}
uniformly over any $z,\sigma\in\{-1,1\}^n$ that satisfy $\frac{1}{n}\sum_{i=1}^n\indc{z_i\neq \sigma_i}>0$. For any $t\in\mathbb{R}$, the $\dot{T}_n^-(t)$ in (\ref{eq:want-to-prove-exact-equal}) is defined by
$$\dot{T}_n^-(t)=\frac{\sum_{i=1}^n\indc{C^-(X_i,Y_i,\theta,\eta)>\sqrt{2}t\log n}-nS_{(r,s)}(t\sqrt{2\log n})}{\sqrt{n S_{(r,s)}(t\sqrt{2\log n})(1-S_{(r,s)}(t\sqrt{2\log n}))}}.$$
By Lemma \ref{lem:comp-tail-0} and the definition of $t^*(r,s)$, we have
\begin{equation}
\frac{1}{\log n}n^{-1}\lesssim\mathbb{P}\left(|\sqrt{r}U+\sqrt{s}V|-\sqrt{r+s}|V|>t^*(r,s)\sqrt{2\log n}\right)\lesssim \frac{1}{\sqrt{\log n}}n^{-1}, \label{eq:range-of-S}
\end{equation}
and thus
$$\frac{1}{\log n}\lesssim nS_{(r,s)}(t^*(r,s)\sqrt{2\log n})\lesssim \frac{1}{\sqrt{\log n}}.$$
Moreover, since $\frac{1}{n}\sum_{i=1}^n\indc{z_i\neq \sigma_i}>0$, there exists some $i_0\in[n]$ such that $z_{i_0}\neq \sigma_{i_0}$. By Lemma \ref{lem:comp-tail-1}, we have
\begin{eqnarray*}
&& \mathbb{P}\left(C^-(X_{i_0},Y_{i_0},\theta,\eta)>\sqrt{2}t^*(r,s)\log n\right) \\
&=& \mathbb{P}_{(U^2,V^2)\sim \chi_{1,2r\log n}^2\otimes \chi_{1,2s\log n}^2}\left(|\sqrt{r}U+\sqrt{s}V|-\sqrt{r+s}|V|>t^*(r,s)\sqrt{2\log n}\right) \\
&\rightarrow& 1,
\end{eqnarray*}
where the last line uses the fact that $t^*(r,s)< r+s-\sqrt{s}\sqrt{r+s}$, which is implied by $r+s-\sqrt{s}\sqrt{r+s}>\frac{1}{2}$ according to Proposition \ref{prop:t-star}. This implies that
$$\indc{C^-(X_{i_0},Y_{i_0},\theta,\eta)>\sqrt{2}t^*(r,s)\log n}=1,$$
with probability tending to $1$. Finally, by (\ref{eq:range-of-S}), we have
\begin{equation}
\dot{T}_n^-(t^*(r,s))\geq \frac{\indc{C^-(X_{i_0},Y_{i_0},\theta,\eta)>\sqrt{2}t^*(r,s)\log n}-nS_{(r,s)}(t^*(r,s)\sqrt{2\log n})}{\sqrt{nS_{(r,s)}(t^*(r,s)\sqrt{2\log n})(1-S_{(r,s)}(t^*(r,s)\sqrt{2\log n}))}}\gtrsim (\log n)^{1/4}, \label{eq:ipad-os}
\end{equation}
with probability tending to $1$. By the fact that $\sqrt{\log\log n}=o((\log n)^{1/4})$, (\ref{eq:want-to-prove-exact-equal}) holds, and the proof is complete.
\end{proof}

\begin{proof}[Proof of Theorem \ref{thm:exact} (lower bound)]
Define $P_0$ to be the joint distribution of $\{(X_i,Y_i)\}_{1\leq i\leq n}$, under which we have
$$(X_i,Y_i)\stackrel{iid}{\sim} \frac{1}{2}N(\theta,I_p)\otimes N(\eta,I_q)+\frac{1}{2}N(-\theta,I_p)\otimes N(-\eta,I_q),\quad i\in[n].$$
For each $i\in[n]$, we define $P_i$ to be the product measure identical to $P_0$ except for its $i$th coordinate takes
$$\frac{1}{2}N(\theta,I_p)\otimes N(-\eta,I_q)+\frac{1}{2}N(-\theta,I_p)\otimes N(\eta,I_q).$$
Then, for any testing function $\psi$, we have
$$\sup_{\substack{z\in\{-1,1\}^n\\ \sigma\in\{z,-z\}}}P^{(n)}_{(\theta,\eta,z,\sigma)}\psi\geq P_0\psi,$$
and
$$\sup_{\substack{z\in\{-1,1\}^n\\\sigma\in\{-1,1\}^n\\\ell(z,\sigma)> 0}}P^{(n)}_{(\theta,\eta,z,\sigma)}(1-\psi)\geq \frac{1}{n}\sum_{i=1}^n P_i(1-\psi).$$
Let us use $p_i$ for the density function of $P_i$ for all $i\in[n]\cup\{0\}$, and we have
\begin{eqnarray*}
R_n^{\rm exact}(\theta,\eta) &\geq& \inf_{\psi}\left(P_0\psi + \frac{1}{n}\sum_{i=1}^nP_i(1-\psi)\right) \\
&=& \int p_0\wedge\left(\frac{1}{n}\sum_{i=1}^np_i\right) \\
&\geq& \int_{\frac{1}{n}\sum_{i=1}^np_i>p_0/2}\frac{1}{2}p_0 \\
&=& \frac{1}{2}P_0\left(\frac{1}{n}\sum_{i=1}^n\frac{p_i(X,Y)}{p_0(X,Y)}>\frac{1}{2}\right).
\end{eqnarray*}
Define
$$
U_i = \frac{\frac{\|\eta\|}{\|\theta\|}\theta^TX_i-\frac{\|\theta\|}{\|\eta\|}\eta^TY_i}{\sqrt{\|\theta\|^2+\|\eta\|^2}}, \quad
V_i = \frac{\theta^TX_i+\eta^TY_i}{\sqrt{\|\theta\|^2+\|\eta\|^2}}.
$$
Then, under $P_0$, we have
$$(U_i,V_i)\stackrel{iid}{\sim} \frac{1}{2}N(0,1)\otimes N(-\sqrt{2(r+s)\log n},1)+\frac{1}{2}N(0,1)\otimes N(\sqrt{2(r+s)\log n},1).$$
We also have
\begin{eqnarray*}
\theta^TX_i-\eta^TY_i &=& \sqrt{2r\log n}U_i + \sqrt{2s\log n}V_i, \\
\theta^TX_i+\eta^TY_i &=& \sqrt{2(r+s)\log n}V_i.
\end{eqnarray*}
This implies for each $i\in[n]$,
\begin{eqnarray*}
\frac{p_i(X,Y)}{p_0(X,Y)} &=& \frac{e^{-\frac{1}{2}\|X_i-\theta\|^2-\frac{1}{2}\|Y_i+\eta\|^2}+e^{-\frac{1}{2}\|X_i+\theta\|^2-\frac{1}{2}\|Y_i-\eta\|^2}}{e^{-\frac{1}{2}\|X_i-\theta\|^2-\frac{1}{2}\|Y_i-\eta\|^2}+e^{-\frac{1}{2}\|X_i+\theta\|^2-\frac{1}{2}\|Y_i+\eta\|^2}} \\
&=& \frac{e^{\theta^TX_i-\eta^TY_i}+e^{-\theta^TX_i+\eta^TY_i}}{e^{\theta^TX_i+\eta^TY_i}+e^{-\theta^TX_i-\eta^TY_i}} \\
&=& \frac{e^{\sqrt{2r\log n}U_i + \sqrt{2s\log n}V_i}+e^{-\sqrt{2r\log n}U_i - \sqrt{2s\log n}V_i}}{e^{\sqrt{2(r+s)\log n}V_i}+e^{-\sqrt{2(r+s)\log n}V_i}} \\
&=& \frac{\phi(U_i-\sqrt{2r\log n})\phi(V_i-\sqrt{2s\log n})+\phi(U_i+\sqrt{2r\log n})\phi(V_i+\sqrt{2s\log n})}{\phi(U_i)\phi(V_i-\sqrt{2(r+s)\log n})+\phi(U_i)\phi(V_i+\sqrt{2(r+s)\log n})},
\end{eqnarray*}
where $\phi(\cdot)$ is the density function of $N(0,1)$.
Use the notation $r(U_i,V_i)=\frac{p_i(X,Y)}{p_0(X,Y)}$, and it suffices to study the statistic $\frac{1}{n}\sum_{i=1}^nr(U_i,V_i)$  under $P_0$.

To this end, we define the event $G=\cap_{i=1}^nG_i$, where
$$G_i=\left\{\sqrt{r}|U_i|-(\sqrt{r+s}-\sqrt{s})|V_i|\leq t^*(r,s)\sqrt{2\log n}\right\}.$$
Then, by Lemma \ref{lem:comp-tail-0} and the definition of $t^*(r,s)$, we have
\begin{eqnarray*}
P_0(G) &\geq& 1-\sum_{i=1}^nP_0(G_i) \\
&=& 1-n\mathbb{P}_{(U^2,V^2)\sim \chi_1^2\otimes \chi_{1,2(r+s)\log n}^2}\left(\sqrt{r}|U|-(\sqrt{r+s}-\sqrt{s})|V|>t^*(r,s)\sqrt{2\log n}\right) \\
&\geq& 1-O\left(\frac{1}{\sqrt{\log n}}\right),
\end{eqnarray*}
which means the event $G$ holds with high probability under $P_0$. Therefore,
\begin{equation}
\frac{1}{n}\sum_{i=1}^nr(U_i,V_i)=\frac{1}{n}\sum_{i=1}^nr(U_i,V_i)\indc{G_i}, \label{eq:just-trunc}
\end{equation}
with high probability, and we can analyze the truncated version $\frac{1}{n}\sum_{i=1}^nr(U_i,V_i)\indc{G_i}$ instead. We first calculate the mean,
\begin{eqnarray*}
&& \mathbb{E}_{P_0}\left(\frac{1}{n}\sum_{i=1}^nr(U_i,V_i)\indc{G_i}\right) \\
&=& \mathbb{E}_{P_0}\left(\frac{1}{n}\sum_{i=1}^n\frac{p_i(X,Y)}{p_0(X,Y)}\indc{G_i}\right) \\
&=& \frac{1}{n}\sum_{i=1}^nP_i(G_i) \\
&=& \mathbb{P}_{(U,V)\sim N(\sqrt{2r\log n},1)\otimes N(\sqrt{2s\log n},1)}\left(\sqrt{r}|U|-(\sqrt{r+s}-\sqrt{s})|V|\leq t^*(r,s)\sqrt{2\log n}\right).
\end{eqnarray*}
Write $U=Z_1+\sqrt{2r\log n}$ and $V=Z_2+\sqrt{2s\log n}$ with independent $Z_1,Z_2\sim N(0,1)$, and we can write the above probability as
\begin{eqnarray*}
&& \mathbb{P}\left(\sqrt{r}|Z_1+\sqrt{2r\log n}|+(\sqrt{r+s}-\sqrt{s})|Z_2+\sqrt{2s\log n}|\leq t^*(r,s)\sqrt{2\log n}\right) \\
&\geq& \mathbb{P}\left(\frac{\sqrt{r}|Z_1|+(\sqrt{r+s}-\sqrt{s})|Z_2|}{\sqrt{2\log n}}+r+s-\sqrt{s}\sqrt{r+s}\leq t^*(r,s)\right) \\
&\rightarrow& 1.
\end{eqnarray*}
The last line above uses the inequality $t^*(r,s)>r+s-\sqrt{s}\sqrt{r+s}$ and the fact that $\frac{\sqrt{r}|Z_1|+(\sqrt{r+s}-\sqrt{s})|Z_2|}{\sqrt{2\log n}}=o_{\mathbb{P}}(1)$. Note that $t^*(r,s)>r+s-\sqrt{s}\sqrt{r+s}$ is implied by $r+s-\sqrt{s}\sqrt{r+s}<\frac{1}{2}$ according to Proposition \ref{prop:t-star}. We therefore have
\begin{equation}
\mathbb{E}_{P_0}\left(\frac{1}{n}\sum_{i=1}^nr(U_i,V_i)\indc{G_i}\right) \rightarrow 1, \label{eq:mean-trunc}
\end{equation}
as $n\rightarrow \infty$.

Next, we calculate the variance of $\frac{1}{n}\sum_{i=1}^nr(U_i,V_i)\indc{G_i}$. We have
\begin{eqnarray*}
\Var_{P_0}\left(\frac{1}{n}\sum_{i=1}^nr(U_i,V_i)\indc{G_i}\right) &=& \frac{1}{n}\Var_{P_0}\left(r(U_1,V_1)\indc{G_1}\right) \\
&\leq& \frac{1}{n}\mathbb{E}_{P_0}\left(r(U_1,V_1)\indc{G_1}\right)^2.
\end{eqnarray*}
By the definition of $r(U_1,V_1)$, we have
\begin{eqnarray*}
&& \mathbb{E}_{P_0}\left(r(U_1,V_1)\indc{G_1}\right)^2 \\
&=& \int_{\frac{\sqrt{r}|u|-(\sqrt{r+s}-\sqrt{s})|v|}{t^*(r,s)}\leq \sqrt{2\log n}}f(u,v)dudv \\
&\leq& \int_{\frac{\sqrt{r}|u|-(\sqrt{r+s}-\sqrt{s})|v|}{t^*(r,s)}\leq \sqrt{2\log n}}\left(f(u,v)+f(-u,v)\right)dudv,
\end{eqnarray*}
where
$$f(u,v) = \frac{\left[\phi(u-\sqrt{2r\log n})\phi(v-\sqrt{2s\log n})+\phi(u+\sqrt{2r\log n})\phi(v+\sqrt{2s\log n})\right]^2}{2[\phi(u)\phi(v-\sqrt{2(r+s)\log n})+\phi(u)\phi(v+\sqrt{2(r+s)\log n})]}.$$
Note that the $f(u,v)+f(-u,v)$ is a function of $(|u|,|v|)$ because of symmetry. Moreover, when $u\geq 0$ and $v\geq 0$, we have
\begin{eqnarray*}
\phi(u-\sqrt{2r\log n}) &\geq& \phi(u+\sqrt{2r\log n}), \\
\phi(v-\sqrt{2s\log n}) &\geq& \phi(v+\sqrt{2s\log n}).
\end{eqnarray*}
These inequalities imply
$$f(u,v)+f(-u,v)\leq \frac{4\phi(u-\sqrt{2r\log n})^2\phi(v-\sqrt{2s\log n})^2}{\phi(u)\phi(v-\sqrt{2(r+s)\log n})}.$$
Hence,
\begin{eqnarray*}
&& \int_{\frac{\sqrt{r}|u|-(\sqrt{r+s}-\sqrt{s})|v|}{t^*(r,s)}\leq \sqrt{2\log n}}\left(f(u,v)+f(-u,v)\right)dudv \\
&=& 4\int_{\frac{\sqrt{r}|u|-(\sqrt{r+s}-\sqrt{s})|v|}{t^*(r,s)}\leq \sqrt{2\log n},u\geq 0,v\geq 0}\left(f(u,v)+f(-u,v)\right)dudv \\
&\leq& 16\int_{\frac{\sqrt{r}|u|-(\sqrt{r+s}-\sqrt{s})|v|}{t^*(r,s)}\leq \sqrt{2\log n},u\geq 0,v\geq 0} \frac{\phi(u-\sqrt{2r\log n})^2\phi(v-\sqrt{2s\log n})^2}{\phi(u)\phi(v-\sqrt{2(r+s)\log n})}dudv \\
&\leq& \frac{16}{2\pi}n^{4(r+s-\sqrt{s}\sqrt{r+s})}\int_{\frac{\sqrt{r}u-(\sqrt{r+s}-\sqrt{s})v}{t^*(r,s)}\leq \sqrt{2\log n}} e^{-\frac{(u-2\sqrt{2r\log n})^2}{2}-\frac{(v-(2\sqrt{2s\log n}-\sqrt{2(r+s)\log n}))^2}{2}}dudv \\
&=& 16n^{4(r+s-\sqrt{s}\sqrt{r+s})}\mathbb{P}\left(\sqrt{r}Z_1-(\sqrt{r+s}-\sqrt{s})Z_2\leq \left(t^*(r,s)-3(r+s-\sqrt{s}\sqrt{r+s})\right)\sqrt{2\log n}\right),
\end{eqnarray*}
where $Z_1,Z_2\stackrel{iid}{\sim} N(0,1)$.
We consider two cases. In the first case, we have $t^*(r,s)\geq 3(r+s-\sqrt{s}\sqrt{r+s})$, and by Proposition \ref{prop:t-star}, this implies $r+s-\sqrt{s}\sqrt{r+s}\leq \frac{1}{8}$. Thus, we have
$$\Var_{P_0}\left(\frac{1}{n}\sum_{i=1}^nr(U_i,V_i)\indc{G_i}\right)\leq 16n^{4(r+s-\sqrt{s}\sqrt{r+s})-1}\rightarrow 0.$$
In the second case, we have $t^*(r,s)< 3(r+s-\sqrt{s}\sqrt{r+s})$, and then a standard Gaussian tail bound (\ref{eq:Gaussian-tail}) gives
\begin{eqnarray*}
&& \mathbb{P}\left(\sqrt{r}Z_1-(\sqrt{r+s}-\sqrt{s})Z_2\leq \left(t^*(r,s)-3(r+s-\sqrt{s}\sqrt{r+s})\right)\sqrt{2\log n}\right) \\
&\lesssim& n^{-\frac{(3(r+s-\sqrt{s}\sqrt{r+s})-t^*(r,s))^2}{2(r+s-\sqrt{s}\sqrt{r+s})}} \\
&\leq& n^{-\frac{(4(r+s-\sqrt{s}\sqrt{r+s})-\sqrt{2(r+s-\sqrt{s}\sqrt{r+s})})^2}{2(r+s-\sqrt{s}\sqrt{r+s})}},
\end{eqnarray*}
where the last inequality uses the property $t^*(r,s)\leq \sqrt{2(r+s-\sqrt{s}\sqrt{r+s})}-(r+s-\sqrt{s}\sqrt{r+s})$ established by Proposition \ref{prop:t-star}. Let us write $A=r+s-\sqrt{s}\sqrt{r+s}$, and then we have
$$\Var_{P_0}\left(\frac{1}{n}\sum_{i=1}^nr(U_i,V_i)\indc{G_i}\right)\lesssim n^{4A-1-\frac{(4A-\sqrt{2A})^2}{2A}}=n^{-2(\sqrt{2A}-1)^2}\rightarrow 0,$$
as long as $A=r+s-\sqrt{s}\sqrt{r+s}<\frac{1}{2}$. Combine the two cases, and we conclude that
\begin{equation}
\Var_{P_0}\left(\frac{1}{n}\sum_{i=1}^nr(U_i,V_i)\indc{G_i}\right) \rightarrow 0, \label{eq:var-trunc}
\end{equation}
as $n\rightarrow \infty$. By (\ref{eq:mean-trunc}) and (\ref{eq:var-trunc}), $\frac{1}{n}\sum_{i=1}^nr(U_i,V_i)\indc{G_i}\rightarrow 1$ in probability. Finally, by (\ref{eq:just-trunc}), we also have $\frac{1}{n}\sum_{i=1}^nr(U_i,V_i)\rightarrow 1$ in probability, and thus $P_0\left(\frac{1}{n}\sum_{i=1}^n\frac{p_i(X,Y)}{p_0(X,Y)}>\frac{1}{2}\right)\rightarrow 1$, which completes the proof.
\end{proof}

\begin{proof}[Proof of Theorem \ref{thm:Bonf}]
Similar to (\ref{eq:type-1-3.1}), we need to bound $\sup_{z\in\{-1,1\}^n}P_{(\theta,\eta,z,z)}^{(n)}\psi_{\rm Bonferroni}$ and $\sup_{z\in\{-1,1\}^n}P_{(\theta,\eta,z,-z)}^{(n)}\psi_{\rm Bonferroni}$ in order to control the Type-I error. For the first term, we have
\begin{eqnarray*}
&& \sup_{z\in\{-1,1\}^n}P_{(\theta,\eta,z,z)}^{(n)}\psi_{\rm Bonferroni} \\
&\leq& \sup_{z\in\{-1,1\}^n}P_{(\theta,\eta,z,z)}^{(n)}\left(\max_{1\leq i\leq n}C^-(X_i,Y_i,\theta,\eta)>2t^*(r,s)\log n\right) \\
&\leq& \sup_{z\in\{-1,1\}^n}\sum_{i=1}^nP_{(\theta,\eta,z,z)}^{(n)}\left(C^-(X_i,Y_i,\theta,\eta)>2t^*(r,s)\log n\right) \\
&=& n\mathbb{P}_{(U^2,V^2)\sim \chi_{1}^2\otimes \chi_{1,2(r+s)\log n}^2}\left(|\sqrt{r}U+\sqrt{s}V|-\sqrt{r+s}|V|>t^*(r,s)\sqrt{2\log n}\right) \\
&\lesssim& \frac{1}{\sqrt{\log n}}\rightarrow 0,
\end{eqnarray*}
where the last inequality is by Lemma \ref{lem:comp-tail-0} and the definition of $t^*(r,s)$. The same argument also applies to $\sup_{z\in\{-1,1\}^n}P_{(\theta,\eta,z,-z)}^{(n)}\psi_{\rm Bonferroni}$, and therefore the Type-I error is vanishing.

For the Type-II error, we have
\begin{eqnarray}
\nonumber && \sup_{\substack{z\in\{-1,1\}^n \\ \sigma\in\{-1,1\}^n \\ \ell(z,\sigma)>0}}P_{(\theta,\eta,z,\sigma)}^{(n)}(1-\psi_{\rm Bonferroni}) \\
\label{eq:bonf-t2-1} &\leq& \sup_{\substack{z\in\{-1,1\}^n \\ \sigma\in\{-1,1\}^n \\ \frac{1}{n}\sum_{i=1}^n\indc{z_i\neq \sigma_i}>0}}P_{(\theta,\eta,z,\sigma)}^{(n)}\left(\max_{1\leq i\leq n}C^-(X_i,Y_i,\theta,\eta)\leq 2t^*(r,s)\log n\right) \\
\label{eq:bonf-t2-2} && + \sup_{\substack{z\in\{-1,1\}^n \\ \sigma\in\{-1,1\}^n \\ \frac{1}{n}\sum_{i=1}^n\indc{z_i\neq -\sigma_i}>0}}P_{(\theta,\eta,z,\sigma)}^{(n)}\left(\max_{1\leq i\leq n}C^+(X_i,Y_i,\theta,\eta)\leq 2t^*(r,s)\log n\right)
\end{eqnarray}
We give a bound for (\ref{eq:bonf-t2-1}). For any $z,\sigma\in\{-1,1\}^n$ such that $\frac{1}{n}\sum_{i=1}^n\indc{z_i\neq \sigma_i}>0$, there exists some $i_0\in[n]$ such that $z_{i_0}\neq \sigma_{i_0}$. Then,
\begin{eqnarray*}
&& P_{(\theta,\eta,z,\sigma)}^{(n)}\left(\max_{1\leq i\leq n}C^-(X_i,Y_i,\theta,\eta)\leq 2t^*(r,s)\log n\right) \\
&\leq& P_{(\theta,\eta,z,\sigma)}^{(n)}\left(C^-(X_{i_0},Y_{i_0},\theta,\eta)\leq 2t^*(r,s)\log n\right) \\
&=& \mathbb{P}_{(U^2,V^2)\sim \chi_{1,2r\log n}^2\otimes \chi_{1,2s\log n}^2}\left(|\sqrt{r}U+\sqrt{s}V|-\sqrt{r+s}|V|\leq t^*(r,s)\sqrt{2\log n}\right) \\
&\rightarrow& 0,
\end{eqnarray*}
where the last line is by Lemma \ref{lem:comp-tail-1} and the condition $t^*(r,s)< r+s-\sqrt{s}\sqrt{r+s}$. Note that $t^*(r,s)< r+s-\sqrt{s}\sqrt{r+s}$ is implied by $r+s-\sqrt{s}\sqrt{r+s}>\frac{1}{2}$ according to Proposition \ref{prop:t-star}. The same analysis also applies to (\ref{eq:bonf-t2-2}), and we thus conclude that the Type-II error is vanishing. The proof is complete.
\end{proof}

\subsection{Proofs of Theorem \ref{thm:ada-Bonf} and Theorem \ref{thm:HC-adaptive}}


\begin{proof}[Proof of Theorem \ref{thm:ada-Bonf}]
Let us first introduce some notation. Define
\begin{eqnarray*}
{C}_0^- &=& \max_{i\in\mathcal{D}_0}C^-(X_i,Y_i,\theta,\eta), \\
{C}_1^- &=& \max_{i\in\mathcal{D}_1}C^-(X_i,Y_i,\theta,\eta), \\
{C}_0^+ &=& \max_{i\in\mathcal{D}_0}C^+(X_i,Y_i,\theta,\eta), \\
{C}_1^+ &=& \max_{i\in\mathcal{D}_1}C^+(X_i,Y_i,\theta,\eta).
\end{eqnarray*}
Then, the Bonferroni test defined by (\ref{eq:test-exact-Bonf}) can be written as
$$\psi_{\rm Bonferroni}=\indc{(C_0^-\vee C_1^-)\wedge (C_0^+\vee C_1^+)>2t^*(r,s)\log n}.$$ Our primary goal in the proof is to bound the difference between $\wh{C}^-\wedge \wh{C}^+$ and $(C_0^-\vee C_1^-)\wedge (C_0^+\vee C_1^+)$. Define
$$M_n^{(0)}=\max_{i\in\mathcal{D}_0}\frac{|(\wh{\theta}^{(1)}-\theta)^TX_i|}{\|\wh{\theta}^{(1)}-\theta\|}\vee \max_{i\in\mathcal{D}_0}\frac{|(\wh{\theta}^{(1)}+\theta)^TX_i|}{\|\wh{\theta}^{(1)}+\theta\|}\vee\max_{i\in\mathcal{D}_0}\frac{|(\wh{\eta}^{(1)}-\eta)^TY_i|}{\|\wh{\eta}^{(1)}-\eta\|}\vee\max_{i\in\mathcal{D}_0}\frac{|(\wh{\eta}^{(1)}+\eta)^TY_i|}{\|\wh{\eta}^{(1)}+\eta\|},$$
$$M_n^{(1)}=\max_{i\in\mathcal{D}_1}\frac{|(\wh{\theta}^{(0)}-\theta)^TX_i|}{\|\wh{\theta}^{(0)}-\theta\|}\vee \max_{i\in\mathcal{D}_1}\frac{|(\wh{\theta}^{(0)}+\theta)^TX_i|}{\|\wh{\theta}^{(0)}+\theta\|}\vee\max_{i\in\mathcal{D}_1}\frac{|(\wh{\eta}^{(0)}-\eta)^TY_i|}{\|\wh{\eta}^{(0)}-\eta\|}\vee\max_{i\in\mathcal{D}_1}\frac{|(\wh{\eta}^{(0)}+\eta)^TY_i|}{\|\wh{\eta}^{(0)}+\eta\|},$$
and $M_n=M_n^{(0)}\vee M_n^{(1)}$.

Since $|C^-(X_i,Y_i,\wh{\theta},\wh{\eta})-C^-(X_i,Y_i,\theta,\eta)|\leq 2|(\wh{\theta}-\theta)^TX_i|+2|(\wh{\eta}-\eta)^TY_i|$, we have
\begin{eqnarray*}
|\wh{C}_0^--C_0^-| &\leq& 2\|\wh{\theta}^{(1)}-\theta\|\max_{i\in\mathcal{D}_0}\frac{|(\wh{\theta}^{(1)}-\theta)^TX_i|}{\|\wh{\theta}^{(1)}-\theta\|} + 2\|\wh{\eta}^{(1)}-\eta\|\max_{i\in\mathcal{D}_0}\frac{|(\wh{\eta}^{(1)}-\eta)^TY_i|}{\|\wh{\eta}^{(1)}-\eta\|} \\
&\leq& 2M_n\left(\|\wh{\theta}^{(1)}-\theta\|+\|\wh{\eta}^{(1)}-\eta\|\right).
\end{eqnarray*}
Note that we can write $C^-(X_i,Y_i,\theta,\eta)=C^-(X_i,Y_i,-\theta,-\eta)$, and thus we also have
$$|\wh{C}_0^--C_0^-|\leq 2M_n\left(\|\wh{\theta}^{(1)}+\theta\|+\|\wh{\eta}^{(1)}+\eta\|\right).$$
Combine the two bounds above, we have
$$
|\wh{C}_0^--C_0^-|\leq 2M_n\left[\left(\|\wh{\theta}^{(1)}-\theta\|+\|\wh{\eta}^{(1)}-\eta\|\right)\wedge \left(\|\wh{\theta}^{(1)}+\theta\|+\|\wh{\eta}^{(1)}+\eta\|\right)\right].
$$
Using the same argument, we also get
$$|\wh{C}_1^--C_1^-|\leq 2M_n\left[\left(\|\wh{\theta}^{(0)}-\theta\|+\|\wh{\eta}^{(0)}-\eta\|\right)\wedge \left(\|\wh{\theta}^{(0)}+\theta\|+\|\wh{\eta}^{(0)}+\eta\|\right)\right].$$
With the above two bounds, we have
\begin{eqnarray*}
|\wh{C}_0^-\vee \wh{C}_1^--C_0^-\vee C_1^-| &\leq& |\wh{C}_0^--C_0^-|\vee|\wh{C}_1^--C_1^-| \\
&\leq& 2M_n\left[\left(\|\wh{\theta}^{(1)}-\theta\|+\|\wh{\eta}^{(1)}-\eta\|\right)\wedge \left(\|\wh{\theta}^{(1)}+\theta\|+\|\wh{\eta}^{(1)}+\eta\|\right)\right] \\
&& + 2M_n\left[\left(\|\wh{\theta}^{(0)}-\theta\|+\|\wh{\eta}^{(0)}-\eta\|\right)\wedge \left(\|\wh{\theta}^{(0)}+\theta\|+\|\wh{\eta}^{(0)}+\eta\|\right)\right].
\end{eqnarray*}
A similar argument leads to
\begin{eqnarray*}
|\wh{C}_0^+\vee \wh{C}_1^+-C_0^+\vee C_1^+| &\leq& 2M_n\left[\left(\|\wh{\theta}^{(1)}-\theta\|+\|\wh{\eta}^{(1)}-\eta\|\right)\wedge \left(\|\wh{\theta}^{(1)}+\theta\|+\|\wh{\eta}^{(1)}+\eta\|\right)\right] \\
&& + 2M_n\left[\left(\|\wh{\theta}^{(0)}-\theta\|+\|\wh{\eta}^{(0)}-\eta\|\right)\wedge \left(\|\wh{\theta}^{(0)}+\theta\|+\|\wh{\eta}^{(0)}+\eta\|\right)\right].
\end{eqnarray*}
Observe the property that $C^+(X_i,Y_i,\wh{\theta},\wh{\eta})=C^-(X_i,Y_i,-\wh{\theta},\wh{\eta})=C^-(X_i,Y_i,\wh{\theta},-\wh{\eta})$. Use this property repeatedly, we get
\begin{eqnarray*}
|\wh{C}_0^-\vee \wh{C}_1^--C_0^+\vee C_1^+| &\leq& 2M_n\left[\left(\|\wh{\theta}^{(1)}-\theta\|+\|\wh{\eta}^{(1)}+\eta\|\right)\wedge \left(\|\wh{\theta}^{(1)}+\theta\|+\|\wh{\eta}^{(1)}-\eta\|\right)\right] \\
&& + 2M_n\left[\left(\|\wh{\theta}^{(0)}-\theta\|+\|\wh{\eta}^{(0)}+\eta\|\right)\wedge \left(\|\wh{\theta}^{(0)}+\theta\|+\|\wh{\eta}^{(0)}-\eta\|\right)\right],
\end{eqnarray*}
and
\begin{eqnarray*}
|\wh{C}_0^+\vee \wh{C}_1^+-C_0^-\vee C_1^-| &\leq& 2M_n\left[\left(\|\wh{\theta}^{(1)}-\theta\|+\|\wh{\eta}^{(1)}+\eta\|\right)\wedge \left(\|\wh{\theta}^{(1)}+\theta\|+\|\wh{\eta}^{(1)}-\eta\|\right)\right] \\
&& + 2M_n\left[\left(\|\wh{\theta}^{(0)}-\theta\|+\|\wh{\eta}^{(0)}+\eta\|\right)\wedge \left(\|\wh{\theta}^{(0)}+\theta\|+\|\wh{\eta}^{(0)}-\eta\|\right)\right].
\end{eqnarray*}
By combining the four bounds above, we obtain
\begin{eqnarray}
\nonumber && |(\wh{C}_0^-\vee \wh{C}_1^-)\wedge (\wh{C}_0^+\vee \wh{C}_1^+) - (C_0^-\vee C_1^-)\wedge (C_0^+\vee C_1^+)| \\
\nonumber &\leq& \left(|\wh{C}_0^-\vee \wh{C}_1^--C_0^-\vee C_1^-| \vee |\wh{C}_0^+\vee \wh{C}_1^+-C_0^+\vee C_1^+|\right) \\
\nonumber && \wedge \left(|\wh{C}_0^-\vee \wh{C}_1^--C_0^+\vee C_1^+| \vee |\wh{C}_0^+\vee \wh{C}_1^+-C_0^-\vee C_1^-|\right) \\
\label{eq:very-disgusting1} &\leq& \left(2M_n\left[\left(\|\wh{\theta}^{(1)}-\theta\|+\|\wh{\eta}^{(1)}-\eta\|\right)\wedge \left(\|\wh{\theta}^{(1)}+\theta\|+\|\wh{\eta}^{(1)}+\eta\|\right)\right]\right. \\
\nonumber && \left.+ 2M_n\left[\left(\|\wh{\theta}^{(0)}-\theta\|+\|\wh{\eta}^{(0)}-\eta\|\right)\wedge \left(\|\wh{\theta}^{(0)}+\theta\|+\|\wh{\eta}^{(0)}+\eta\|\right)\right]\right) \\
\nonumber && \wedge\left(2M_n\left[\left(\|\wh{\theta}^{(1)}-\theta\|+\|\wh{\eta}^{(1)}+\eta\|\right)\wedge \left(\|\wh{\theta}^{(1)}+\theta\|+\|\wh{\eta}^{(1)}-\eta\|\right)\right]\right. \\
\nonumber && \left.+ 2M_n\left[\left(\|\wh{\theta}^{(0)}-\theta\|+\|\wh{\eta}^{(0)}+\eta\|\right)\wedge \left(\|\wh{\theta}^{(0)}+\theta\|+\|\wh{\eta}^{(0)}-\eta\|\right)\right]\right).
\end{eqnarray}
A symmetric argument leads to
\begin{eqnarray}
\nonumber && \left|(\wh{C}_0^-\vee\wh{C}_1^+)\wedge(\wh{C}_0^+\vee\wh{C}_1^-)-(C_0^-\vee C_1^-)\wedge (C_0^+\vee C_1^+)\right| \\
\nonumber &\leq& \left(\left|\wh{C}_0^-\vee\wh{C}_1^+-C_0^-\vee C_1^-\right|\vee\left|\wh{C}_0^+\vee\wh{C}_1^--C_0^+\vee C_1^+\right|\right) \\
\nonumber && \wedge \left(\left|\wh{C}_0^-\vee\wh{C}_1^+-C_0^+\vee C_1^+\right|\vee\left|\wh{C}_0^+\vee\wh{C}_1^--C_0^-\vee C_1^-\right|\right) \\
\label{eq:very-disgusting2} &\leq& \left(2M_n\left[\left(\|\wh{\theta}^{(1)}-\theta\|+\|\wh{\eta}^{(1)}-\eta\|\right)\wedge \left(\|\wh{\theta}^{(1)}+\theta\|+\|\wh{\eta}^{(1)}+\eta\|\right)\right]\right) \\
\nonumber && \left.+ 2M_n\left[\left(\|\wh{\theta}^{(0)}+\theta\|+\|\wh{\eta}^{(0)}-\eta\|\right)\wedge \left(\|\wh{\theta}^{(0)}-\theta\|+\|\wh{\eta}^{(0)}+\eta\|\right)\right]\right) \\
\nonumber && \wedge\left(2M_n\left[\left(\|\wh{\theta}^{(1)}+\theta\|+\|\wh{\eta}^{(1)}-\eta\|\right)\wedge \left(\|\wh{\theta}^{(1)}-\theta\|+\|\wh{\eta}^{(1)}+\eta\|\right)\right]\right. \\
\nonumber && \left.+ 2M_n\left[\left(\|\wh{\theta}^{(0)}+\theta\|+\|\wh{\eta}^{(0)}+\eta\|\right)\wedge \left(\|\wh{\theta}^{(0)}-\theta\|+\|\wh{\eta}^{(0)}-\eta\|\right)\right]\right).
\end{eqnarray}
The condition (\ref{eq:estimation-error-weak}) indicates that
\begin{equation}
L(\wh{\theta}^{(0)},\theta)\vee L(\wh{\theta}^{(1)},\theta)\vee L(\wh{\eta}^{(0)},\eta)\vee L(\wh{\eta}^{(1)},\eta)\leq n^{-\gamma}, \label{eq:msorrow}
\end{equation}
with high probability. Due to sample splitting, we have $M_n^{(0)}\vee M_n^{(0)}\leq 8\sqrt{2\log n}$ with high probability. Let $\gamma'$ be a constant that satisfies $0<\gamma'<\gamma$. Then, with (\ref{eq:msorrow}), either the righthand side of (\ref{eq:very-disgusting1}) or the righthand side of (\ref{eq:very-disgusting2}) is bounded by $n^{-\gamma'}$. In the first case, we have $\indc{\|\wh{\theta}^{(0)}-\wh{\theta}^{(1)}\|\leq 1, \|\wh{\eta}^{(0)}-\wh{\eta}^{(1)}\|\leq 1} + \indc{\|\wh{\theta}^{(0)}-\wh{\theta}^{(1)}\|> 1, \|\wh{\eta}^{(0)}-\wh{\eta}^{(1)}\|> 1}=1$, and then
$$\wh{C}^-\wedge\wh{C}^+=(\wh{C}_0^-\vee \wh{C}_1^-)\wedge (\wh{C}_0^+\vee \wh{C}_1^+).$$
In the second case, we have $\indc{\|\wh{\theta}^{(0)}-\wh{\theta}^{(1)}\|> 1, \|\wh{\eta}^{(0)}-\wh{\eta}^{(1)}\|\leq 1} + \indc{\|\wh{\theta}^{(0)}-\wh{\theta}^{(1)}\|\leq 1, \|\wh{\eta}^{(0)}-\wh{\eta}^{(1)}\|> 1}=1$, and then
$$\wh{C}^-\wedge\wh{C}^+=(\wh{C}_0^-\vee\wh{C}_1^+)\wedge(\wh{C}_0^+\vee\wh{C}_1^-).$$
Therefore, in both cases, we have
\begin{equation}
\left|\wh{C}^-\wedge\wh{C}^+-(C_0^-\vee C_1^-)\wedge (C_0^+\vee C_1^+)\right|\leq n^{-\gamma'}, \label{eq:cao-ni-ma}
\end{equation}
with high probability. The definition of $\wh{t}$ also implies
\begin{equation}
|\wh{t}-t^*(r,s)|\leq n^{-\gamma'}, \label{eq:MLGB}
\end{equation}
with high probability.

To this end, we define $G$ to be the intersection of the events (\ref{eq:cao-ni-ma}) and (\ref{eq:MLGB}). Then,
\begin{eqnarray*}
&& \sup_{\substack{z\in\{-1,1\}^n\\ \sigma\in\{z,-z\}}}P^{(n)}_{(\theta,\eta,z,\sigma)}\psi_{\rm ada-Bonferroni} + \sup_{\substack{z\in\{-1,1\}^n\\\sigma\in\{-1,1\}^n\\\ell(z,\sigma)> 0}}P^{(n)}_{(\theta,\eta,z,\sigma)}(1-\psi_{\rm ada-Bonferroni}) \\
&\leq& \sup_{\substack{z\in\{-1,1\}^n\\ \sigma\in\{z,-z\}}}P^{(n)}_{(\theta,\eta,z,\sigma)}\psi_{\rm ada-Bonferroni}\indc{G} + \sup_{\substack{z\in\{-1,1\}^n\\\sigma\in\{-1,1\}^n\\\ell(z,\sigma)> 0}}P^{(n)}_{(\theta,\eta,z,\sigma)}(1-\psi_{\rm ada-Bonferroni})\indc{G} \\
&& + \sup_{\substack{z\in\{-1,1\}^n\\\sigma\in\{-1,1\}^n}}P^{(n)}_{(\theta,\eta,z,\sigma)}(G^c).
\end{eqnarray*}
The last term is vanishing because both (\ref{eq:cao-ni-ma}) and (\ref{eq:MLGB}) hold with high probability. Since
\begin{eqnarray*}
&& P^{(n)}_{(\theta,\eta,z,\sigma)}\psi_{\rm ada-Bonferroni}\indc{G} \\
&\leq& P^{(n)}_{(\theta,\eta,z,\sigma)}\left((C_0^-\vee C_1^-)\wedge (C_0^+\vee C_1^+)>2t^*(r,s)\log n\right),
\end{eqnarray*}
and
\begin{eqnarray*}
&& P^{(n)}_{(\theta,\eta,z,\sigma)}(1-\psi_{\rm ada-Bonferroni})\indc{G} \\
&\leq& P^{(n)}_{(\theta,\eta,z,\sigma)}\left((C_0^-\vee C_1^-)\wedge (C_0^+\vee C_1^+)\leq 2(t^*(r,s)+\delta)\log n\right),
\end{eqnarray*}
for some arbitrarily small constant $\delta>0$, the same argument in the proof of Theorem \ref{thm:Bonf} leads to the desired conclusion.
\end{proof}

\begin{proof}[Proof of Theorem \ref{thm:HC-adaptive}]
Let us introduce some notation. For any $a_1,b_1,a_2,b_2,t\in\mathbb{R}$, define
$$P(a_1,b_1,a_2,b_2,t)=\mathbb{P}\left(C^-(Z_1+a_1,Z_2+b_1,a_2,b_2)>t\sqrt{2\log n}\right),$$
where $Z_1,Z_2\stackrel{iid}{\sim}N(0,1)$. We first give a bound for the difference between $P(a_1,b_1,a_2,b_2,t)$ and $P(a_2,b_2,a_2,b_2,t)$. Note that
\begin{eqnarray}
\nonumber && P(a_1,b_1,a_2,b_2,t) \\
\label{eq:only-useful-later} &=& \mathbb{P}\left(\left\{-2b_2(Z_2+b_1)>t\sqrt{2\log n}\text{ and }2a_2(Z_1+a_1)>t\sqrt{2\log n}\right\}\right. \\
\nonumber && \left.\text{ or }\left\{-2a_2(Z_1+a_1)>t\sqrt{2\log n}\text{ and }2b_2(Z_2+b_1)>t\sqrt{2\log n}\right\}\right) \\
\nonumber &=& \mathbb{P}\left(-2b_2(Z_2+b_1)>t\sqrt{2\log n}\right)\mathbb{P}\left(2a_2(Z_1+a_1)>t\sqrt{2\log n}\right) \\
\nonumber && + \mathbb{P}\left(-2a_2(Z_1+a_1)>t\sqrt{2\log n}\right)\mathbb{P}\left(2b_2(Z_2+b_1)>t\sqrt{2\log n}\right) \\
\nonumber && - \mathbb{P}\left(|2b_2(Z_2+b_1)|<-t\sqrt{2\log n}\right)\mathbb{P}\left(|2a_2(Z_1+a_1)|<-t\sqrt{2\log n}\right).
\end{eqnarray}
We therefore have
\begin{eqnarray}
\nonumber && \left|P(a_1,b_1,a_2,b_2,t)-P(a_2,b_2,a_2,b_2,t)\right| \\
\label{eq:est-prob-diff-1} &\leq& \left|\mathbb{P}\left(-2b_2(Z_2+b_1)>t\sqrt{2\log n}\right)\mathbb{P}\left(2a_2(Z_1+a_1)>t\sqrt{2\log n}\right)\right. \\
\nonumber && \left.- \mathbb{P}\left(-2b_2(Z_2+b_2)>t\sqrt{2\log n}\right)\mathbb{P}\left(2a_2(Z_1+a_2)>t\sqrt{2\log n}\right)\right| \\
\label{eq:est-prob-diff-2} && +\left|\mathbb{P}\left(-2a_2(Z_1+a_1)>t\sqrt{2\log n}\right)\mathbb{P}\left(2b_2(Z_2+b_1)>t\sqrt{2\log n}\right)\right. \\
\nonumber && \left.- \mathbb{P}\left(-2a_2(Z_1+a_2)>t\sqrt{2\log n}\right)\mathbb{P}\left(2b_2(Z_2+b_2)>t\sqrt{2\log n}\right)\right| \\
\label{eq:est-prob-diff-3} && +\left|\mathbb{P}\left(|2b_2(Z_2+b_1)|<-t\sqrt{2\log n}\right)\mathbb{P}\left(|2a_2(Z_1+a_1)|<-t\sqrt{2\log n}\right)\right. \\
\nonumber && \left.- \mathbb{P}\left(|2b_2(Z_2+b_2)|<-t\sqrt{2\log n}\right)\mathbb{P}\left(|2a_2(Z_1+a_2)|<-t\sqrt{2\log n}\right)\right|.
\end{eqnarray}
We demonstrate how to bound (\ref{eq:est-prob-diff-1}). The bounds for (\ref{eq:est-prob-diff-2}) and (\ref{eq:est-prob-diff-3}) follow the same argument, and thus we omit the details. By triangle inequality, (\ref{eq:est-prob-diff-1}) can be bounded by
\begin{eqnarray*}
&& \mathbb{P}\left(-2b_2(Z_2+b_2)>t\sqrt{2\log n}\right) \\
&& \times \left|\mathbb{P}\left(2a_2(Z_1+a_1)>t\sqrt{2\log n}\right)-\mathbb{P}\left(2a_2(Z_1+a_2)>t\sqrt{2\log n}\right)\right| \\
&&  +\mathbb{P}\left(2a_2(Z_1+a_2)>t\sqrt{2\log n}\right) \\
&& \times \left|\mathbb{P}\left(-2b_2(Z_2+b_1)>t\sqrt{2\log n}\right)-\mathbb{P}\left(-2b_2(Z_2+b_2)>t\sqrt{2\log n}\right)\right| \\
&& + \left|\mathbb{P}\left(2a_2(Z_1+a_1)>t\sqrt{2\log n}\right)-\mathbb{P}\left(2a_2(Z_1+a_2)>t\sqrt{2\log n}\right)\right| \\
&& \times \left|\mathbb{P}\left(-2b_2(Z_2+b_1)>t\sqrt{2\log n}\right)-\mathbb{P}\left(-2b_2(Z_2+b_2)>t\sqrt{2\log n}\right)\right|.
\end{eqnarray*}
Then, for any $1\leq |a_1|, |a_2|, |b_1|, |b_2|\leq \log n$ and $|t|\leq \log n$, apply Lemma \ref{prop:standard-normal}, and we have
\begin{eqnarray*}
&& \left|\mathbb{P}\left(2a_2(Z_1+a_1)>t\sqrt{2\log n}\right)-\mathbb{P}\left(2a_2(Z_1+a_2)>t\sqrt{2\log n}\right)\right| \\
&=& \left|\int_{\frac{t\sqrt{2\log n}-2a_1a_2}{2|a_2|}}^{\infty}\phi(x)dx-\int_{\frac{t\sqrt{2\log n}-2a_2^2}{2|a_2|}}^{\infty}\phi(x)dx\right| \\
&\leq& 2|a_1-a_2|\left(\phi\left(\frac{t\sqrt{2\log n}-2a_1a_2}{2|a_2|}\right)\vee\phi\left(\frac{t\sqrt{2\log n}-2a_2^2}{2|a_2|}\right)\right) \\
&\leq& 4|a_1-a_2|\phi\left(\frac{t\sqrt{2\log n}-2a_2^2}{2|a_2|}\right) \\
&\leq& (\log n)^2|a_1-a_2|\mathbb{P}\left(2a_2(Z_1+a_2)>t\sqrt{2\log n}\right).
\end{eqnarray*}
With a similar argument, we also have
\begin{eqnarray*}
&& \left|\mathbb{P}\left(-2b_2(Z_2+b_1)>t\sqrt{2\log n}\right)-\mathbb{P}\left(-2b_2(Z_2+b_2)>t\sqrt{2\log n}\right)\right| \\
&\leq& (\log n)^2|b_1-b_2|\mathbb{P}\left(-2b_2(Z_2+b_2)>t\sqrt{2\log n}\right),
\end{eqnarray*}
and we therefore obtain the following bound for (\ref{eq:est-prob-diff-1}),
\begin{eqnarray*}
&& \left((\log n)^2\left(|a_1-a_2|+|b_1-b_2|\right) + (\log n)^4|a_1-a_2||b_1-b_2|\right) \\
&& \times \mathbb{P}\left(2a_2(Z_1+a_2)>t\sqrt{2\log n}\right)\mathbb{P}\left(-2b_2(Z_2+b_2)>t\sqrt{2\log n}\right).
\end{eqnarray*}
With similar bounds obtained for (\ref{eq:est-prob-diff-2}) and (\ref{eq:est-prob-diff-3}), we have
\begin{eqnarray*}
&& \left|P(a_1,b_1,a_2,b_2,t)-P(a_2,b_2,a_2,b_2,t)\right| \\
&\leq&  \left((\log n)^2\left(|a_1-a_2|+|b_1-b_2|\right) + (\log n)^4|a_1-a_2||b_1-b_2|\right)  \\
&& \times \left(\mathbb{P}\left(2a_2(Z_1+a_2)>t\sqrt{2\log n}\right)\mathbb{P}\left(-2b_2(Z_2+b_2)>t\sqrt{2\log n}\right) \right. \\
&& \left. + \mathbb{P}\left(-2a_2(Z_1+a_2)>t\sqrt{2\log n}\right)\mathbb{P}\left(2b_2(Z_2+b_2)>t\sqrt{2\log n}\right)\right) \\
&\leq& 2\left((\log n)^2\left(|a_1-a_2|+|b_1-b_2|\right) + (\log n)^4|a_1-a_2||b_1-b_2|\right)P(a_2,b_2,a_2,b_2,t),
\end{eqnarray*}
where the last inequality uses the fact that $\mathbb{P}(A)+\mathbb{P}(B)\leq 2\mathbb{P}(A\cup B)$. We summarize the bound into the following inequality,
\begin{eqnarray}
\nonumber && \frac{\left|P(a_1,b_1,a_2,b_2,t)-P(a_2,b_2,a_2,b_2,t)\right|}{P(a_2,b_2,a_2,b_2,t)} \\
&\leq& 2\left((\log n)^2\left(|a_1-a_2|+|b_1-b_2|\right) + (\log n)^4|a_1-a_2||b_1-b_2|\right), \label{eq:good-interface-1}
\end{eqnarray}
which holds uniformly over $1\leq |a_1|, |a_2|, |b_1|, |b_2|\leq \log n$ and $|t|\leq \log n$.

Next, we study the ratio between $P(a_1,b_1,a_1,b_1,t)$ and $P(a_2,b_2,a_2,b_2,t)$. By a union bound argument, we have
\begin{eqnarray*}
&& P(a_1,b_1,a_1,b_1,t) \\
&\leq& \mathbb{P}\left(-2b_1(Z_2+b_1)>t\sqrt{2\log n}, 2a_1(Z_1+a_1)>t\sqrt{2\log n}\right) \\
&& + \mathbb{P}\left(-2a_1(Z_1+a_1)>t\sqrt{2\log n},2b_1(Z_2+b_1)>t\sqrt{2\log n}\right) \\
&=& \mathbb{P}\left(-2b_1(Z_2+b_1)>t\sqrt{2\log n}\right)\mathbb{P}\left(2a_1(Z_1+a_1)>t\sqrt{2\log n}\right) \\
&& + \mathbb{P}\left(-2a_1(Z_1+a_1)>t\sqrt{2\log n}\right)\mathbb{P}\left(2b_1(Z_2+b_1)>t\sqrt{2\log n}\right).
\end{eqnarray*}
By $\mathbb{P}(A)+\mathbb{P}(B)\leq 2\mathbb{P}(A\cup B)$ we have
\begin{eqnarray*}
&& P(a_2,b_2,a_2,b_2,t) \\
&\geq& \frac{1}{2}\mathbb{P}\left(-2b_2(Z_2+b_2)>t\sqrt{2\log n}\right)\mathbb{P}\left(2a_2(Z_1+a_2)>t\sqrt{2\log n}\right) \\
&& + \frac{1}{2}\mathbb{P}\left(-2a_2(Z_1+a_2)>t\sqrt{2\log n}\right)\mathbb{P}\left(2b_2(Z_2+b_2)>t\sqrt{2\log n}\right).
\end{eqnarray*}
For any $1\leq |a_1|, |a_2|, |b_1|, |b_2|\leq \log n$ and $|t|\leq \log n$ that satisfy $|a_1-a_2|\leq n^{-c}$ and $|b_1-b_2|\leq n^{-c}$ with some constant $c>0$, we apply Lemma \ref{prop:standard-normal}, and obtain
\begin{eqnarray*}
\frac{\mathbb{P}\left(-2b_1(Z_2+b_1)>t\sqrt{2\log n}\right)}{\mathbb{P}\left(-2b_2(Z_2+b_2)>t\sqrt{2\log n}\right)} &\leq& 2, \\
\frac{\mathbb{P}\left(2a_1(Z_1+a_1)>t\sqrt{2\log n}\right)}{\mathbb{P}\left(2a_2(Z_1+a_2)>t\sqrt{2\log n}\right)} &\leq& 2, \\
\frac{\mathbb{P}\left(-2a_1(Z_1+a_1)>t\sqrt{2\log n}\right)}{\mathbb{P}\left(-2a_2(Z_1+a_2)>t\sqrt{2\log n}\right)} &\leq& 2, \\
\frac{\mathbb{P}\left(2b_1(Z_2+b_1)>t\sqrt{2\log n}\right)}{\mathbb{P}\left(2b_2(Z_2+b_2)>t\sqrt{2\log n}\right)} &\leq& 2.
\end{eqnarray*}
Hence, we have
\begin{equation}
\frac{P(a_1,b_1,a_1,b_1,t)}{P(a_2,b_2,a_2,b_2,t)} \leq 4, \label{eq:good-interface-2}
\end{equation}
uniformly over all $1\leq |a_1|, |a_2|, |b_1|, |b_2|\leq \log n$ and $|t|\leq \log n$ that satisfy $|a_1-a_2|\leq n^{-c}$ and $|b_1-b_2|\leq n^{-c}$.

Use Hoeffding's inequality, and we have
\begin{equation}
|\mathcal{D}_0|\wedge|\mathcal{D}_1|\wedge|\mathcal{D}_2|\geq \frac{n}{4},\label{eq:sample-size-large}
\end{equation}
with high probability. The condition (\ref{eq:estimation-error-strong}), together with (\ref{eq:sample-size-large}), implies $L(\wh{\theta},\theta)\vee L(\wh{\eta},\eta)\leq (n/4)^{-\gamma}$ with high probability. By the definitions of $a$ and $b$, we then have
\begin{equation}
(|a-\|\theta\||\wedge|a+\|\theta\||)\vee(|b-\|\eta\||\wedge|b+\|\eta\||)\leq (n/4)^{-\gamma}, \label{eq:prob-a-b}
\end{equation}
with high probability. With some standard calculation and the sample size bound (\ref{eq:sample-size-large}), we also have
\begin{equation}
(|\wh{a}-a|\wedge|\wh{a}+a|)\vee(|\wh{b}-b|\wedge|\wh{b}+b|)\leq \sqrt{\frac{\log n}{n}}, \label{eq:prob-a-b-hat}
\end{equation}
with high probability. 
Finally, by definitions of $\wh{r},\wh{s},r,s$ and (\ref{eq:sample-size-large})-(\ref{eq:prob-a-b-hat}), we have
\begin{equation}
|\wh{r}-r|\vee|\wh{s}-s|\leq n^{-\wt{\gamma}},  \label{eq:prob-r-s}
\end{equation}
with high probability for some constant $\wt{\gamma}>0$. Let us define $G$ to be the intersection of the events (\ref{eq:sample-size-large}), (\ref{eq:prob-a-b}), (\ref{eq:prob-a-b-hat}), and (\ref{eq:prob-r-s}). Then, we have
\begin{equation}
\sup_{z,\sigma\in\{-1,1\}^n}P_{(\theta,\eta,z,\sigma)}^{(n)}(G)\rightarrow 1. \label{eq:good-event-adaptive}
\end{equation}

With the help of (\ref{eq:good-interface-1}), (\ref{eq:good-interface-2}), and (\ref{eq:good-event-adaptive}), we are ready to analyze $R_n(\psi_{\rm ada-HC},\theta,\eta)$. It is easy to check by the definition that
\begin{eqnarray}
\label{eq:pp1} S_{(r,s)}(t) &=& P(\|\theta\|,\|\eta\|,\|\theta\|,\|\eta\|,t), \\
\label{eq:pp2} S_{(\wh{r},\wh{s})}(t) &=& P(\wh{a},\wh{b},\wh{a},\wh{b},t).
\end{eqnarray}
By the calibration (\ref{eq:general-cali}), we know that under the event $G$, we have $1\leq |a|, |b|, |\wh{a}|, |\wh{b}|, \|\theta\|, \|\eta\|\leq \log n$  {for sufficiently large values of $n$}. Moreover,  the definitions of $\wh{T}_n^-$ and $\wh{T}_n^+$ imply that it is sufficient to consider $|t|\leq \log n$. With $X_i\sim N(z_i\theta,I_p)$ and $Y_i\sim N(\sigma_i\eta,I_q)$, for any $i\in\mathcal{D}_2$, we have
\begin{eqnarray}
\label{eq:pp3} \mathbb{P}\left(C^-(\wh{X}_i,\wh{Y}_i,\wh{a},\wh{b})>t\sqrt{2\log n}\Big|
 {\{d_j\}_{j\in[n]},\{(X_j,Y_j)\}_{j\in\mathcal{D}_0\cup\mathcal{D}_1}}\right) &=& P(z_ia,\sigma_ib,\wh{a},\wh{b},t), \\
\label{eq:pp4} \mathbb{P}\left(C^+(\wh{X}_i,\wh{Y}_i,\wh{a},\wh{b})>t\sqrt{2\log n}\Big|
 {\{d_j\}_{j\in[n]},\{(X_j,Y_j)\}_{j\in\mathcal{D}_0\cup\mathcal{D}_1}}\right) &=& P(z_ia,\sigma_ib,\wh{a},-\wh{b},t).
\end{eqnarray}
In addition, due to the symmetry in the definition, we have
\begin{eqnarray}
\label{eq:pp5} && P(a,b,\wh{a},\wh{b},t) = P(-a,-b,\wh{a},\wh{b},t), \\
\label{eq:pp6} && P(a,b,\wh{a},-\wh{b},t) = P(a,-b,\wh{a},\wh{b},t) = P(-a,b,\wh{a},\wh{b},t).
\end{eqnarray}
Now we analyze the Type-I error. By triangle inequality, we have
\begin{eqnarray}
\label{eq:bound-T-ada1} \wh{T}_n^- &\leq& \sup_{|t|\leq \log n}\frac{\left|\sum_{i\in\mathcal{D}_2}\indc{C^-(\wh{X}_i,\wh{Y}_i,\wh{a},\wh{b})>t\sqrt{2\log n}}-|\mathcal{D}_2|P(a,b,\wh{a},\wh{b},t) \right|}{\sqrt{|\mathcal{D}_2|S_{(\wh{r},\wh{s})}(t)}} \\
\nonumber && + \sup_{|t|\leq\log n}\frac{|\mathcal{D}_2|\left|P(a,b,\wh{a},\wh{b},t)-S_{(\wh{r},\wh{s})}(t)\right|}{\sqrt{|\mathcal{D}_2|S_{(\wh{r},\wh{s})}(t)}}, \\
\label{eq:bound-T-ada2} \wh{T}_n^- &\leq& \sup_{|t|\leq \log n}\frac{\left|\sum_{i\in\mathcal{D}_2}\indc{C^-(\wh{X}_i,\wh{Y}_i,\wh{a},\wh{b})>t\sqrt{2\log n}}-|\mathcal{D}_2|P(-a,b,\wh{a},\wh{b},t) \right|}{\sqrt{|\mathcal{D}_2|S_{(\wh{r},\wh{s})}(t)}} \\
\nonumber && + \sup_{|t|\leq\log n}\frac{|\mathcal{D}_2|\left|P(-a,b,\wh{a},\wh{b},t)-S_{(\wh{r},\wh{s})}(t)\right|}{\sqrt{|\mathcal{D}_2|S_{(\wh{r},\wh{s})}(t)}}, \\
\end{eqnarray}
\begin{eqnarray}
\label{eq:bound-T-ada3} \wh{T}_n^+ &\leq& \sup_{|t|\leq \log n}\frac{\left|\sum_{i\in\mathcal{D}_2}\indc{C^+(\wh{X}_i,\wh{Y}_i,\wh{a},\wh{b})>t\sqrt{2\log n}}-|\mathcal{D}_2|P(a,b,\wh{a},\wh{b},t) \right|}{\sqrt{|\mathcal{D}_2|S_{(\wh{r},\wh{s})}(t)}} \\
\nonumber && + \sup_{|t|\leq\log n}\frac{|\mathcal{D}_2|\left|P(a,b,\wh{a},\wh{b},t)-S_{(\wh{r},\wh{s})}(t)\right|}{\sqrt{|\mathcal{D}_2|S_{(\wh{r},\wh{s})}(t)}}, \\
\label{eq:bound-T-ada4} \wh{T}_n^+ &\leq& \sup_{|t|\leq \log n}\frac{\left|\sum_{i\in\mathcal{D}_2}\indc{C^+(\wh{X}_i,\wh{Y}_i,\wh{a},\wh{b})>t\sqrt{2\log n}}-|\mathcal{D}_2|P(-a,b,\wh{a},\wh{b},t) \right|}{\sqrt{|\mathcal{D}_2|S_{(\wh{r},\wh{s})}(t)}} \\
\nonumber && + \sup_{|t|\leq\log n}\frac{|\mathcal{D}_2|\left|P(-a,b,\wh{a},\wh{b},t)-S_{(\wh{r},\wh{s})}(t)\right|}{\sqrt{|\mathcal{D}_2|S_{(\wh{r},\wh{s})}(t)}}.
\end{eqnarray}
Therefore, we can use any one of the four bounds above to upper bound $\wh{T}_n^-\wedge \wh{T}_n^+$. Under the null hypothesis, we either have $z=\sigma$ or $z=-\sigma$. Assume $z=\sigma$ without loss of generality, and then by (\ref{eq:pp3})-(\ref{eq:pp6}), we should use the smaller bound between (\ref{eq:bound-T-ada1}) and (\ref{eq:bound-T-ada4}). 
By (\ref{eq:prob-a-b-hat}), $\wh{a}$ and $\wh{b}$ estimates $a$ and $b$ up to their signs. Let us suppose $|\wh{a}-a|\vee|\wh{b}-b|\leq\sqrt{\frac{\log n}{n}}$, and then by (\ref{eq:pp2}), we should use (\ref{eq:bound-T-ada1}) instead of  (\ref{eq:bound-T-ada4}) to bound $\wh{T}_n^-\wedge \wh{T}_n^+$. 
By (\ref{eq:good-interface-1}), the first term of (\ref{eq:bound-T-ada1}) can be bounded by
\begin{eqnarray*}
&& (1+o_{\mathbb{P}}(1))\sup_{|t|\leq \log n}\frac{\left|\sum_{i\in\mathcal{D}_2}\indc{C^-(\wh{X}_i,\wh{Y}_i,\wh{a},\wh{b})>t\sqrt{2\log n}}-|\mathcal{D}_2|P(a,b,\wh{a},\wh{b},t) \right|}{\sqrt{|\mathcal{D}_2|P(a,b,\wh{a},\wh{b},t)}} \\
&\leq& (1+o_{\mathbb{P}}(1))\sup_{|t|\leq \log n}\frac{\left|\sum_{i\in\mathcal{D}_2}\indc{C^-(\wh{X}_i,\wh{Y}_i,\wh{a},\wh{b})>t\sqrt{2\log n}}-|\mathcal{D}_2|P(a,b,\wh{a},\wh{b},t) \right|}{\sqrt{|\mathcal{D}_2|P(a,b,\wh{a},\wh{b},t)(1-P(a,b,\wh{a},\wh{b},t))}} \\
&=& (1+o_{\mathbb{P}}(1))\sqrt{2\log\log|\mathcal{D}_2|}=(1+o_{\mathbb{P}}(1))\sqrt{2\log\log n},
\end{eqnarray*}
where the last line is by \cite{shorack2009empirical,donoho2004higher} and (\ref{eq:sample-size-large}). By (\ref{eq:good-interface-1}), the second term of (\ref{eq:bound-T-ada1}) can be bounded by
\begin{eqnarray*}
&& \sup_{|t|\leq\log n}\frac{|\mathcal{D}_2|\left|P(a,b,\wh{a},\wh{b},t)-S_{(\wh{r},\wh{s})}(t)\right|}{\sqrt{|\mathcal{D}_2|S_{(\wh{r},\wh{s})}(t)}} \\
&\lesssim& \sup_{|t|\leq\log n} \frac{|\mathcal{D}_2|\frac{(\log n)^{5/2}}{\sqrt{n}}S_{(\wh{r},\wh{s})}(t)}{\sqrt{|\mathcal{D}_2|S_{(\wh{r},\wh{s})}(t)}} \lesssim (\log n)^{5/2}.
\end{eqnarray*}
Thus, we have $\wh{T}_n^-\wedge \wh{T}_n^+\lesssim (\log n)^{5/2}$. In the case when $z=-\sigma$, $|\wh{a}+a|\leq\sqrt{\frac{\log n}{n}}$, or $|\wh{b}+b|\leq\sqrt{\frac{\log n}{n}}$, we can use one of the four bounds (\ref{eq:bound-T-ada1})-(\ref{eq:bound-T-ada4}) to get the same conclusion. To summarize, whenever $G$ holds, $\wh{T}_n^-\wedge \wh{T}_n^+\lesssim (\log n)^{5/2}$ with high probability under the null, and
the Type-I error can be bounded by
$$\sup_{\substack{z\in\{-1,1\}^n\\ \sigma\in\{z,-z\}}}P^{(n)}_{(\theta,\eta,z,\sigma)}\psi_{\rm ada-HC}\leq \sup_{\substack{z\in\{-1,1\}^n\\ \sigma\in\{z,-z\}}}P^{(n)}_{(\theta,\eta,z,\sigma)}\psi_{\rm ada-HC}\indc{G}+\sup_{\substack{z\in\{-1,1\}^n\\ \sigma\in\{z,-z\}}}P^{(n)}_{(\theta,\eta,z,\sigma)}(G^c)\rightarrow 0.$$

To analyze the Type-II error, we define $G_{++}=G\cap\{a\geq 0, b\geq 0\}$, $G_{+-}=G\cap\{a\geq 0, b< 0\}$, $G_{-+}=G\cap\{a< 0, b\geq 0\}$, and $G_{--}=G\cap\{a< 0, b< 0\}$. Then,
\begin{eqnarray*}
&& \sup_{\substack{z\in\{-1,1\}^n\\\sigma\in\{-1,1\}^n\\\ell(z,\sigma)> \epsilon}}P^{(n)}_{(\theta,\eta,z,\sigma)}(1-\psi_{\rm ada-HC}) \\
&\leq& \sup_{\substack{z\in\{-1,1\}^n\\\sigma\in\{-1,1\}^n\\\ell(z,\sigma)> \epsilon}}P^{(n)}_{(\theta,\eta,z,\sigma)}(1-\psi_{\rm ada-HC})\indc{G_{++}} + \sup_{\substack{z\in\{-1,1\}^n\\\sigma\in\{-1,1\}^n\\\ell(z,\sigma)> \epsilon}}P^{(n)}_{(\theta,\eta,z,\sigma)}(1-\psi_{\rm ada-HC})\indc{G_{+-}} \\
&& + \sup_{\substack{z\in\{-1,1\}^n\\\sigma\in\{-1,1\}^n\\\ell(z,\sigma)> \epsilon}}P^{(n)}_{(\theta,\eta,z,\sigma)}(1-\psi_{\rm ada-HC})\indc{G_{-+}} + \sup_{\substack{z\in\{-1,1\}^n\\\sigma\in\{-1,1\}^n\\\ell(z,\sigma)> \epsilon}}P^{(n)}_{(\theta,\eta,z,\sigma)}(1-\psi_{\rm ada-HC})\indc{G_{--}} \\
&& +\sup_{\substack{z\in\{-1,1\}^n\\\sigma\in\{-1,1\}^n\\\ell(z,\sigma)> \epsilon}}P^{(n)}_{(\theta,\eta,z,\sigma)}(G^c).
\end{eqnarray*}
The first four terms in the bound can be bounded in the same way, and we only show how to bound the first term.
By the definition of $\ell(z,z^*)$, we have
\begin{align}
\nonumber & \sup_{\substack{z\in\{-1,1\}^n\\\sigma\in\{-1,1\}^n\\\ell(z,\sigma)> \epsilon}}P^{(n)}_{(\theta,\eta,z,\sigma)}(1-\psi_{\rm ada-HC})\indc{G_{++}}  \\
\label{eq:xiaren-1} \leq& \sup_{\substack{z\in\{-1,1\}^n\\\sigma\in\{-1,1\}^n\\ \frac{1}{n}\sum_{i=1}^n\indc{z_i\neq \sigma_i}> \epsilon}}P^{(n)}_{(\theta,\eta,z,\sigma)}\left(\wh{T}_n^-\leq (\log n)^3,G_{++}\right) \\
\label{eq:xiaren-2} & + \sup_{\substack{z\in\{-1,1\}^n\\\sigma\in\{-1,1\}^n\\ \frac{1}{n}\sum_{i=1}^n\indc{z_i\neq -\sigma_i}> \epsilon}}P^{(n)}_{(\theta,\eta,z,\sigma)}\left(\wh{T}_n^+\leq (\log n)^3,G_{++}\right).
\end{align}
We then bound (\ref{eq:xiaren-1}). The bound for (\ref{eq:xiaren-2}) follows a similar argument and thus we omit the details. We have
$$P^{(n)}_{(\theta,\eta,z,\sigma)}\left(\wh{T}_n^-\leq (\log n)^3,G_{++}\right)\leq \inf_{|t|\leq \log n}P^{(n)}_{(\theta,\eta,z,\sigma)}\left(|\wh{T}_n^-(t)|\leq (\log n)^3,G_{++}\right),$$
where
$$\wh{T}_n^-(t)=\frac{\sum_{i\in\mathcal{D}_2}\indc{C^-(\wh{X}_i,\wh{Y}_i,\wh{a},\wh{b})>t\sqrt{2\log n}}-|\mathcal{D}_2|S_{(\wh{r},\wh{s})}(t)}{\sqrt{|\mathcal{D}_2|S_{(\wh{r},\wh{s})}(t)}}.$$
Note that under $G_{++}$, we have
\begin{eqnarray}
\label{eq:prob-a-b-hat-sign} |\wh{a}-a|\vee|\wh{b}-b| &\leq& \sqrt{\frac{\log n}{n}}, \\ 
\label{eq:prob-a-b-sign} |a-\|\theta\||\vee|b-\|\eta\|| &\leq& (n/4)^{-\gamma}.
\end{eqnarray}
Define $m_0=\sum_{i\in\mathcal{D}_2}\indc{z_i=\sigma_i}$ and $m_1=\sum_{i\in\mathcal{D}_2}\indc{z_i\neq\sigma_i}$. Then, we can write $\wh{T}_n^-(t)$ as
$$\wh{T}_n^-(t)=R_n(t)+\frac{m_0P(a,b,\wh{a},\wh{b},t)-m_0P(\wh{a},\wh{b},\wh{a},\wh{b},t)}{\sqrt{|\mathcal{D}_2|S_{(\wh{r},\wh{s})}(t)}},$$
where
$$R_n(t)=\frac{\sum_{i\in\mathcal{D}_2}\indc{C^-(\wh{X}_i,\wh{Y}_i,\wh{a},\wh{b})>t\sqrt{2\log n}}-m_0P(a,b,\wh{a},\wh{b},t)-m_1P(\wh{a},\wh{b},\wh{a},\wh{b},t)}{\sqrt{|\mathcal{D}_2|S_{(\wh{r},\wh{s})}(t)}}.$$
The difference between $\wh{T}_n^-(t)$ and $R_n(t)$ can be ignored compared with the threshold $(\log n)^3$ because
$$\sup_{|t|\leq\log n}\frac{\left|m_0P(a,b,\wh{a},\wh{b},t)-m_0P(\wh{a},\wh{b},\wh{a},\wh{b},t)\right|}{\sqrt{|\mathcal{D}_2|S_{(\wh{r},\wh{s})}(t)}} \lesssim (\log n)^{5/2},$$
by (\ref{eq:good-interface-1}), (\ref{eq:sample-size-large}), and (\ref{eq:prob-a-b-hat-sign}). Now we study the mean and the variance  {of $R_n(t)$} conditioning on $\{d_i\}_{i\in[n]}$ and $\{(X_i,Y_i)\}_{i\in\mathcal{D}_0\cup\mathcal{D}_1}$. The conditional mean of the numerator of $R_n(t)$ is given by
$$m_1P(a,-b,\wh{a},\wh{b},t)-m_1P(\wh{a},\wh{b},\wh{a},\wh{b},t),$$
and the conditional variance of the numerator of $R_n(t)$ is bounded by
$$m_0P(a,b,\wh{a},\wh{b},t)+m_1P(a,-b,\wh{a},\wh{b},t).$$
By Chebyshev's inequality, $m_1\geq \frac{1}{6}n^{1-\beta}$ with high probability. Use (\ref{eq:good-interface-1}), (\ref{eq:good-interface-2}), (\ref{eq:prob-a-b-hat-sign}), and (\ref{eq:prob-a-b-sign}), and we have
\begin{eqnarray*}
&& P(a,-b,\wh{a},\wh{b},t) \asymp P(\wh{a},-\wh{b},\wh{a},\wh{b},t) \asymp P(a,-b,a,b,t) \asymp P(\|\theta\|,-\|\eta\|,\|\theta\|,\|\eta\|,t) \\
&& P(a,b,\wh{a},\wh{b},t) \asymp P(\wh{a},\wh{b},\wh{a},\wh{b},t) \asymp  P(a,b,a,b,t) \asymp P(\|\theta\|,\|\eta\|,\|\theta\|,\|\eta\|,t).
\end{eqnarray*}
Therefore, following the same argument in the proof of Theorem \ref{thm:general-HC}, we have
$$\frac{(\mathbb{E}(R_n(t)|\{d_i\}_{i\in[n]},\{(X_i,Y_i)\}_{i\in\mathcal{D}_0\cup\mathcal{D}_1}))^2}{\Var(R_n(t)|\{d_i\}_{i\in[n]},\{(X_i,Y_i)\}_{i\in\mathcal{D}_0\cup\mathcal{D}_1})}\rightarrow\infty,$$
at a polynomial rate with high probability as long as $\beta<\beta^*(r,s)$, which implies (\ref{eq:xiaren-1}) is vanishing, and thus we have $\lim_{n\rightarrow 0}R_n(\psi_{\rm ada-HC},\theta,\eta)=0$.
\end{proof}

\subsection{Proofs of Lemma \ref{lem:LR-approx} and Lemma \ref{lem:max-order}}

\begin{proof}[Proof of Lemma \ref{lem:LR-approx}]
By the definitions of $p(u,v)$, and $q(u,v)$, we have
\begin{eqnarray*}
\frac{q(u,v)}{p(u,v)} &=& \frac{\phi(u-\sqrt{2r\log n})\phi(v-\sqrt{2(r+s)\log n})+\phi(u+\sqrt{2r\log n})\phi(v+\sqrt{2(r+s)\log n})}{\phi(u)\phi(v-\sqrt{2(r+s)\log n})+\phi(u)\phi(v+\sqrt{2(r+s)\log n})} \\
&=& \frac{e^{u\sqrt{2r\log n}+v\sqrt{2s\log n}}+e^{-u\sqrt{2r\log n}-v\sqrt{2s\log n}}}{e^{v\sqrt{2(r+s)\log n}}+e^{-v\sqrt{2(r+s)\log n}}}.
\end{eqnarray*}
Apply the inequality $e^{|x|}\leq e^x+e^{-x}\leq 2e^{|x|}$ to both the numerator and the denominator, and we have
$$\frac{1}{2}\leq\frac{q(u,v)/p(u,v)}{e^{\sqrt{2\log n}\left(|\sqrt{r}u+\sqrt{s}v|-\sqrt{r+s}|v|\right)}}\leq 2.$$
This leads to the desired conclusion.
\end{proof}

\begin{proof}[Proof of Lemma \ref{lem:max-order}]
We note that by Lemma \ref{lem:comp-tail-0} and the definition of $t^*(r,s)$, we have
$$\mathbb{P}\left(|\sqrt{r}U+\sqrt{s}V|-\sqrt{r+s}|V|>t^*(r,s)\sqrt{2\log n}\right)\lesssim \frac{1}{\sqrt{\log n}}n^{-1}.$$
Therefore, using a union bound argument, we have
\begin{eqnarray*}
&& \mathbb{P}\left(\frac{\max_{1\leq i\leq n}\left(|\sqrt{r}U_i+\sqrt{s}V_i|-\sqrt{r+s}|V_i|\right)}{\sqrt{2\log n}}> t^*(r,s)\right) \\
&\leq& \sum_{i=1}^n\mathbb{P}\left(|\sqrt{r}U_i+\sqrt{s}V_i|-\sqrt{r+s}|V_i|>t^*(r,s)\sqrt{2\log n}\right) \\
&\lesssim& \frac{1}{\sqrt{\log n}}\rightarrow 0.
\end{eqnarray*}
Apply Lemma \ref{lem:comp-tail-0} and the definition of $t^*(r,s)$ again, and we have for any constant $\delta>0$, there exists some $\delta'>0$, such that
$$\mathbb{P}\left(|\sqrt{r}U+\sqrt{s}V|-\sqrt{r+s}|V|>(t^*(r,s)-\delta)\sqrt{2\log n}\right)\gtrsim n^{-(1-\delta')}.$$
Then, we have
\begin{eqnarray*}
&& \mathbb{P}\left(\frac{\max_{1\leq i\leq n}\left(|\sqrt{r}U_i+\sqrt{s}V_i|-\sqrt{r+s}|V_i|\right)}{\sqrt{2\log n}}\leq t^*(r,s)-\delta\right) \\
&\leq& \prod_{i=1}^n\mathbb{P}\left(\frac{|\sqrt{r}U_i+\sqrt{s}V_i|-\sqrt{r+s}|V_i|}{\sqrt{2\log n}}\leq t^*(r,s)-\delta\right) \\
&=& \prod_{i=1}^n\left(1-\mathbb{P}\left(\frac{|\sqrt{r}U_i+\sqrt{s}V_i|-\sqrt{r+s}|V_i|}{\sqrt{2\log n}}> t^*(r,s)-\delta\right)\right) \\
&\leq& \left(1-n^{-(1-\delta')}\right)^n \rightarrow 0.
\end{eqnarray*}
The proof is complete.
\end{proof}

\subsection{Proofs of Proposition \ref{prop:comp-p-value} and Proposition \ref{prop:parameter-estimation}}\label{sec:pf-org2}

\begin{proof}[Proof of Proposition \ref{prop:comp-p-value}]
By the definition of $C^-(\wh{X}_i,\wh{Y}_i,\wh{a},\wh{b})$, we have
$$
\max_{i\in\mathcal{D}_2}|C^-(\wh{X}_i,\wh{Y}_i,\wh{a},\wh{b})| \leq 2|\wh{a}|\max_{i\in\mathcal{D}_2}|\wh{X}_i| + 2|\wh{b}|\max_{i\in\mathcal{D}_2}|\wh{Y}_i|.
$$
Note that $|\wh{a}|\leq\|\theta\|=O(\sqrt{\log n})$ and $|\wh{b}|\leq\|\eta\|=O(\sqrt{\log n})$. Due to data splitting, we can write $\wh{X}_i=z_i\wh{\theta}^T\theta/\|\wh{\theta}\|+W_i$ for each $i\in\mathcal{D}_2$ with $W_i\sim N(0,1)$ independent of $\mathcal{D}_2$. Then,
$$\max_{i\in\mathcal{D}_2}|\wh{X}_i|\leq \|\theta\|+\max_{i\in\mathcal{D}_2}|W_i|,$$ 
where we have $\|\theta\|=O(\sqrt{\log n})$ and $\max_{i\in\mathcal{D}_2}|W_i|\leq C\sqrt{\log n}$ with probability tending to $1$ with some sufficiently large constant $C>0$. 
The same argument also applies to $\max_{i\in\mathcal{D}_2}|\wh{Y}_i|$. Therefore, $\max_{i\in\mathcal{D}_2}\frac{|C^-(\wh{X}_i,\wh{Y}_i,\wh{a},\wh{b})|}{\sqrt{2\log n}}=O(\sqrt{\log n})$ with high probability, which implies that
$$\inf_{\substack{z\in\{-1,1\}^n\\\sigma\in\{-1,1\}^n}}P^{(n)}_{(\theta,\eta,z,\sigma)}\left(\max_{i\in\mathcal{D}_2}\frac{|C^-(\wh{X}_i,\wh{Y}_i,\wh{a},\wh{b})|}{\sqrt{2\log n}}\leq\log n\right)\rightarrow 1.$$
Hence, with high probability, we have
\begin{eqnarray*}
\wh{T}_n^- &=& \sup_{|t|\leq \log n}\frac{\left|\sum_{i\in\mathcal{D}_2}\indc{C^-(\wh{X}_i,\wh{Y}_i,\wh{a},\wh{b})>t\sqrt{2\log n}}-|\mathcal{D}_2|S_{(\wh{r},\wh{s})}(t)\right|}{\sqrt{|\mathcal{D}_2|S_{(\wh{r},\wh{s})}(t)}} \\
&=& \sup_{t\in\mathbb{R}}\frac{\left|\sum_{i\in\mathcal{D}_2}\indc{C^-(\wh{X}_i,\wh{Y}_i,\wh{a},\wh{b})>t\sqrt{2\log n}}-|\mathcal{D}_2|S_{(\wh{r},\wh{s})}(t)\right|}{\sqrt{|\mathcal{D}_2|S_{(\wh{r},\wh{s})}(t)}} \\
&=& \max_{1\leq i\leq |\mathcal{D}_2|}\frac{\sqrt{|\mathcal{D}_2|}\left|\frac{i}{|\mathcal{D}_2|}-\wh{p}_{(i,\mathcal{D}_2)}^-\right|}{\sqrt{\wh{p}_{(i,\mathcal{D}_2)}^-}}.
\end{eqnarray*}
The same conclusion also applies to $\wh{T}_n^+$ by the same argument, and the proof is complete.
\end{proof}

\begin{proof}[Proof of Proposition \ref{prop:parameter-estimation}]
For $X_i\sim N(z_i\theta,I_p)$, we can write $X_i=z_i\theta+W_i$ with $W_i\sim N(0,I_p)$ independently for all $i\in[n]$. Then,
\begin{eqnarray*}
\left\|\frac{1}{n}\sum_{i=1}^nX_iX_i^T-\theta\theta^T-I_p\right\|_{\rm op} &\leq& \left\|\frac{1}{n}\sum_{i=1}^nW_iW_i^T-I_p\right\|_{\rm op} + 2\left\|\theta\left(\frac{1}{n}\sum_{i=1}^nz_iW_i^T\right)\right\|_{\rm op} \\
&\leq& \left\|\frac{1}{n}\sum_{i=1}^nW_iW_i^T-I_p\right\|_{\rm op} + 2\|\theta\|\left\|\frac{1}{n}\sum_{i=1}^nz_iW_i\right\|.
\end{eqnarray*}
By a standard covariance matrix concentration bound \citep{koltchinskii2014concentration},
$$\left\|\frac{1}{n}\sum_{i=1}^nW_iW_i^T-I_p\right\|_{\rm op}\leq C\sqrt{\frac{p}{n}},$$
with probability at least $1-e^{-C'p}$. Since $\left\|\frac{1}{\sqrt{n}}\sum_{i=1}^nz_iW_i\right\|^2\sim \chi_p^2$, a chi-square tail bound \citep{laurent2000adaptive} gives
$$\left\|\frac{1}{n}\sum_{i=1}^nz_iW_i\right\|\leq C\sqrt{\frac{p}{n}},$$
with probability at least $1-e^{-C'p}$. Therefore,
\begin{equation}
\left\|\frac{1}{n}\sum_{i=1}^nX_iX_i^T-\theta\theta^T-I_p\right\|_{\rm op} \leq C(1+2\|\theta\|)\sqrt{\frac{p}{n}},\label{eq:trivial-cov-bound}
\end{equation}
with probability at least $1-e^{-C'p}$. By (\ref{eq:trivial-cov-bound}) and Davis-Kahan theorem, we have
$$\|\wh{u}_1-\theta/\|\theta\|\|\wedge\|\wh{u}_1+\theta/\|\theta\|\|\leq C_1\frac{1+\|\theta\|}{\|\theta\|^2}\sqrt{\frac{p}{n}}.$$
Weyl's inequality and (\ref{eq:trivial-cov-bound}) give $|\wh{\lambda}_1  {-1}-\|\theta\|^2|\leq C(1+2\|\theta\|)\sqrt{\frac{p}{n}}$, which leads to
$$\left|\sqrt{\wh{\lambda}_1  {-1}}-\|\theta\| \right|\leq C\frac{1+2\|\theta\|}{\|\theta\|}\sqrt{\frac{p}{n}}.$$
Then, by triangle inequality, we have
$$L(\wh{\theta},\theta)\leq \left|\sqrt{\wh{\lambda}_1 {-1}}-\|\theta\| \right|+\|\theta\|\left(\|\wh{u}_1-\theta/\|\theta\|\|\wedge\|\wh{u}_1+\theta/\|\theta\|\|\right).$$
Combining the bounds, we obtain the desired result.
\end{proof}

\subsection{Proofs of Technical Lemmas}\label{sec:pf-last}

\begin{proof}[Proof of Lemma \ref{prop:standard-normal}]
Conclusion 1 is a standard Gaussian tail estimate from Page 116 of \cite{ross2007second}. For the second conclusion, we consider $t_1\leq t_2$ without loss of generality. If $t_1$ and $t_2$ are on the same side of $0$, we will have $\left|\int_{t_1}^{t_2}\phi(x)dx\right|\leq |t_1-t_2|\sup_{x\in[t_1,t_2]}\phi(x)\leq |t_1-t_2|\left(\phi(t_1)\vee\phi(t_2)\right)$. Otherwise, $\left|\int_{t_1}^{t_2}\phi(x)dx\right|\leq |t_1-t_2|\phi(0)$. The condition $|t_1-t_2|\leq 1$ implies $\phi(0)\leq 2\left(\phi(t_1)\vee\phi(t_2)\right)$, which leads to the desired conclusion. For Conclusion 3, we first consider $t\leq 2$, and then $\frac{\phi(t)/(1\vee t)}{1-\Phi(t)}\leq \frac{\phi(0)}{1-\Phi(2)}\leq 20$. For $t>2$, we use the first conclusion and then we have $\frac{\phi(t)/(1\vee t)}{1-\Phi(t)}\leq 1-t^{-2}\leq 20$. 
For Conclusion 4, we have $\phi(t_1)/\phi(t_2)\leq e^{\frac{1}{2}|t_1-t_2||t_1+t_2|}\leq 2$, since $|t_1-t_2||t_1+t_2|\rightarrow 0$ when $|t_1|,|t_2|\leq (\log n)^2$ and $|t_1-t_2|\leq n^{-c}$.
Finally, we prove Conclusion 5. We have $\frac{1-\Phi(t_1)}{1-\Phi(t_2)}\leq \frac{1-\Phi(t_1)}{1-\Phi(t_1)-2|t_1-t_2|\left(\phi(t_1)\vee\phi(t_2)\right)}$, where the inequality is by Conclusion 2. By Conclusions 3 and 4, $2|t_1-t_2|\left(\phi(t_1)\vee\phi(t_2)\right)/(1-\Phi(t_1))\rightarrow 0$ when $|t_1|,|t_2|\leq (\log n)^2$ and $|t_1-t_2|\leq n^{-c}$. This implies $\frac{1-\Phi(t_1)}{1-\Phi(t_2)}\leq 2$, and the proof is complete.
\end{proof}

\begin{proof}[Proof of Lemma \ref{lem:easy-tail}]
Let $Z\sim N(0,1)$, and then we can write
$$\mathbb{P}\left(U^2\leq 2t\log n\right)=\mathbb{P}\left(-\sqrt{2\log n}(\sqrt{t}+\sqrt{r})\leq Z\leq -\sqrt{2\log n}(\sqrt{r}-\sqrt{t})\right).$$
An application of (\ref{eq:Gaussian-tail}) leads to first result.

For the second result, note that we can write $U=Z_1+\sqrt{2r\log n}$ and $V=Z_2+\sqrt{2s\log n}$ with independent $Z_1,Z_2\sim N(0,1)$. Then, we have
\begin{eqnarray}
\nonumber && \mathbb{P}\left(|U|-|V|>t\sqrt{2\log n}\right) \\
\label{eq:union-sharp} &\asymp& \mathbb{P}\left(U>|V|+t\sqrt{2\log n}\right) + \mathbb{P}\left(U<-|V|-t\sqrt{2\log n}\right) \\
\nonumber &=& \mathbb{P}\left(V<U-t\sqrt{2\log n}, V>-U+t\sqrt{2\log n}\right) \\
\nonumber && + \mathbb{P}\left(V<-U-t\sqrt{2\log n}, V>U+t\sqrt{2\log n}\right) \\
\nonumber &=& \mathbb{P}\left(Z_2-Z_1<\sqrt{2\log n}(\sqrt{r}-\sqrt{s}-t), Z_2+Z_1> \sqrt{2\log n}(t-\sqrt{r}-\sqrt{s})\right) \\
\nonumber && + \mathbb{P}\left(Z_2+Z_1<\sqrt{2\log n}(-\sqrt{r}-\sqrt{s}-t), Z_2-Z_1>\sqrt{2\log n}(\sqrt{r}-\sqrt{s}+t)\right) \\
\label{eq:independent-sharp} &=& \mathbb{P}\left(Z_2-Z_1<\sqrt{2\log n}(\sqrt{r}-\sqrt{s}-t)\right)\mathbb{P}\left(Z_2+Z_1> \sqrt{2\log n}(t-\sqrt{r}-\sqrt{s})\right) \\
\nonumber && + \mathbb{P}\left(Z_2+Z_1<\sqrt{2\log n}(-\sqrt{r}-\sqrt{s}-t)\right)\mathbb{P}\left(Z_2-Z_1>\sqrt{2\log n}(t+\sqrt{r}-\sqrt{s})\right),
\end{eqnarray}
where we have used $(\mathbb{P}(A)+\mathbb{P}(B))/2\leq \mathbb{P}(A\cup B)\leq \mathbb{P}(A)+\mathbb{P}(B)$ in (\ref{eq:union-sharp}) and the fact that $Z_2-Z_1$ and $Z_2+Z_1$ are independent in (\ref{eq:independent-sharp}). Use (\ref{eq:Gaussian-tail}), and we have
$$\mathbb{P}\left(Z_2-Z_1<\sqrt{2\log n}(\sqrt{r}-\sqrt{s}-t)\right)\asymp\begin{cases}
\frac{1}{\sqrt{\log n}}n^{-\frac{1}{2}(t-\sqrt{r}+\sqrt{s})^2}, & t>\sqrt{r}-\sqrt{s}, \\
1, & t\leq \sqrt{r}-\sqrt{s},
\end{cases}$$
and
$$\mathbb{P}\left(Z_2+Z_1> \sqrt{2\log n}(t-\sqrt{r}-\sqrt{s})\right)\asymp \begin{cases}
\frac{1}{\sqrt{\log n}}n^{-\frac{1}{2}(t-\sqrt{r}-\sqrt{s})^2}, & t>\sqrt{r}+\sqrt{s}, \\
1, & t\leq \sqrt{r}+\sqrt{s}.
\end{cases}$$
The product of the above two probabilities is of order
$$\begin{cases}
\frac{1}{\log n}n^{-\left[(t-\sqrt{r})^2+s\right]}, & t>\sqrt{r}+\sqrt{s}, \\
\frac{1}{\sqrt{\log n}}n^{-\frac{1}{2}(t-\sqrt{r}+\sqrt{s})^2}, & \sqrt{r}-\sqrt{s}<t\leq\sqrt{r}+\sqrt{s}, \\
1, & t\leq \sqrt{r}-\sqrt{s}.
\end{cases}$$
The quantity $\mathbb{P}\left(Z_2+Z_1<\sqrt{2\log n}(-\sqrt{r}-\sqrt{s}-t)\right)\mathbb{P}\left(Z_2-Z_1>\sqrt{2\log n}(t+\sqrt{r}-\sqrt{s})\right)$ can be analyzed in the same way, and it is of a smaller order. Thus, the proof is complete.
\end{proof}

\begin{proof}[Proof of Lemma \ref{lem:comp-tail-0}]
We first study $\mathbb{P}\left(|\sqrt{r}U+\sqrt{s}V|-\sqrt{r+s}|V|>t\sqrt{2\log n}\right)$.
Consider independent random variables $W_1\sim N(0,r+s-\sqrt{s}\sqrt{r+s})$, $W_2\sim N(0,r+s-\sqrt{s}\sqrt{r+s})$, and $W_3\sim N(0,\sqrt{s}\sqrt{r+s})$. It is easy to check that
$$(\sqrt{r}U+\sqrt{s}V,\sqrt{r+s}V)\stackrel{d}{=}(W_1+W_3+\sqrt{s}\sqrt{r+s}\sqrt{2\log n}, W_2+W_3+(r+s)\sqrt{2\log n}).$$
Therefore,
\begin{eqnarray*}
&& \mathbb{P}\left(|\sqrt{r}U+\sqrt{s}V|-\sqrt{r+s}|V|>t\sqrt{2\log n}\right) \\
&=& \mathbb{P}\left(|W_1+W_3+\sqrt{s}\sqrt{r+s}\sqrt{2\log n}|-|W_2+W_3+(r+s)\sqrt{2\log n}|>t\sqrt{2\log n}\right) \\
&\asymp& \mathbb{P}\left(W_1+W_3+\sqrt{s}\sqrt{r+s}\sqrt{2\log n}>t\sqrt{2\log n}+|W_2+W_3+(r+s)\sqrt{2\log n}|\right) \\
&& + \mathbb{P}\left(-W_1-W_3-\sqrt{s}\sqrt{r+s}\sqrt{2\log n}>t\sqrt{2\log n}+|W_2+W_3+(r+s)\sqrt{2\log n}|\right) \\
&=& \mathbb{P}\left(W_1+W_3+\sqrt{s}\sqrt{r+s}\sqrt{2\log n}>t\sqrt{2\log n}+W_2+W_3+(r+s)\sqrt{2\log n},\right. \\
&& \left.W_1+W_3+\sqrt{s}\sqrt{r+s}\sqrt{2\log n}>t\sqrt{2\log n}-W_2-W_3-(r+s)\sqrt{2\log n}\right) \\
&& + \mathbb{P}\left(-W_1-W_3-\sqrt{s}\sqrt{r+s}\sqrt{2\log n}>t\sqrt{2\log n}+W_2+W_3+(r+s)\sqrt{2\log n},\right. \\
&& \left.-W_1-W_3-\sqrt{s}\sqrt{r+s}\sqrt{2\log n}>t\sqrt{2\log n}-W_2-W_3-(r+s)\sqrt{2\log n}\right) \\
&=&  \mathbb{P}\left(W_1-W_2>(t+r+s-\sqrt{s}\sqrt{r+s})\sqrt{2\log n}\right) \\
&& \times \mathbb{P}\left(W_1+W_2+2W_3>(t-(r+s)-\sqrt{s}\sqrt{r+s})\sqrt{2\log n}\right) \\
&& + \mathbb{P}\left(-W_1-W_2-2W_3>(t+r+s+\sqrt{s}\sqrt{r+s})\sqrt{2\log n}\right) \\
&& \times \mathbb{P}\left(-W_1+W_2>(t-(r+s)+\sqrt{s}\sqrt{r+s})\sqrt{2\log n}\right),
\end{eqnarray*}
where we have used the fact that $W_1-W_2$, $W_1+W_2$, and $W_3$ are independent. For the four probabilities above, we have
\begin{eqnarray*}
&& \mathbb{P}\left(W_1-W_2>(t+r+s-\sqrt{s}\sqrt{r+s})\sqrt{2\log n}\right) \\
&=& \mathbb{P}\left(N(0,1)>\frac{(t+r+s-\sqrt{s}\sqrt{r+s})\sqrt{2\log n}}{\sqrt{2(r+s-\sqrt{s}\sqrt{r+s})}}\right) \\
&\asymp& \begin{cases}
\frac{1}{\sqrt{\log n}}n^{-\frac{(t+r+s-\sqrt{s}\sqrt{r+s})^2}{2(r+s-\sqrt{s}\sqrt{r+s})}}, & t>-r-s+\sqrt{s}\sqrt{r+s}, \\
1, & t\leq -r-s+\sqrt{s}\sqrt{r+s},
\end{cases}
\end{eqnarray*}
\begin{eqnarray*}
&& \mathbb{P}\left(W_1+W_2+2W_3>(t-(r+s)-\sqrt{s}\sqrt{r+s})\sqrt{2\log n}\right) \\
&=& \mathbb{P}\left(N(0,1)>\frac{(t-(r+s)-\sqrt{s}\sqrt{r+s})\sqrt{2\log n}}{\sqrt{2(r+s+\sqrt{s}\sqrt{r+s})}}\right) \\
&\asymp& \begin{cases}
\frac{1}{\sqrt{\log n}}n^{-\frac{(t-r-s-\sqrt{s}\sqrt{r+s})^2}{2(r+s+\sqrt{s}\sqrt{r+s})}}, & t>r+s+\sqrt{s}\sqrt{r+s}, \\
1, & t\leq r+s+\sqrt{s}\sqrt{r+s},
\end{cases}
\end{eqnarray*}
\begin{eqnarray*}
&& \mathbb{P}\left(-W_1-W_2-2W_3>(t+r+s+\sqrt{s}\sqrt{r+s})\sqrt{2\log n}\right) \\
&=& \mathbb{P}\left(N(0,1)>\frac{(t+r+s+\sqrt{s}\sqrt{r+s})\sqrt{2\log n}}{\sqrt{2(r+s+\sqrt{s}\sqrt{r+s})}}\right) \\
&\asymp& \begin{cases}
\frac{1}{\sqrt{\log n}}n^{-\frac{(t+r+s+\sqrt{s}\sqrt{r+s})^2}{2(r+s+\sqrt{s}\sqrt{r+s})}}, & t>-r-s-\sqrt{s}\sqrt{r+s}, \\
1, & t\leq -r-s-\sqrt{s}\sqrt{r+s},
\end{cases}
\end{eqnarray*}
and
\begin{eqnarray*}
&& \mathbb{P}\left(-W_1+W_2>(t-(r+s)+\sqrt{s}\sqrt{r+s})\sqrt{2\log n}\right) \\
&=& \mathbb{P}\left(N(0,1)>\frac{(t-(r+s)+\sqrt{s}\sqrt{r+s})\sqrt{2\log n}}{\sqrt{2(r+s-\sqrt{s}\sqrt{r+s})}}\right) \\
&\asymp& \begin{cases}
\frac{1}{\sqrt{\log n}}n^{-\frac{(t-r-s+\sqrt{s}\sqrt{r+s})^2}{2(r+s-\sqrt{s}\sqrt{r+s})}}, & t>r+s-\sqrt{s}\sqrt{r+s}, \\
1, & t\leq r+s-\sqrt{s}\sqrt{r+s}.
\end{cases}
\end{eqnarray*}
Putting the pieces together, we get
\begin{eqnarray*}
&& \mathbb{P}\left(|\sqrt{r}U+\sqrt{s}V|-\sqrt{r+s}|V|>t\sqrt{2\log n}\right) \\
&\asymp& \begin{cases}
\frac{1}{\log n}n^{-\frac{t^2+r(r+s)}{r}}, & t>r+s+\sqrt{s}\sqrt{r+s}, \\
\frac{1}{\sqrt{\log n}}n^{-\frac{(t+r+s-\sqrt{s}\sqrt{r+s})^2}{2(r+s-\sqrt{s}\sqrt{r+s})}},& -(r+s)+\sqrt{s}\sqrt{r+s}<t\leq r+s+\sqrt{s}\sqrt{r+s}, \\
1, & t\leq -(r+s)+\sqrt{s}\sqrt{r+s}.
\end{cases}
\end{eqnarray*}

Next, we analyze $\mathbb{P}\left(\sqrt{r}|U|-(\sqrt{r+s}-\sqrt{s})|V|>t\sqrt{2\log n}\right)$. Note that
\begin{equation}
\sqrt{r}|U|-(\sqrt{r+s}-\sqrt{s})|V|=\left(|\sqrt{r}U+\sqrt{s}V|-\sqrt{r+s}|V|\right)\vee\left(|\sqrt{r}U-\sqrt{s}V|-\sqrt{r+s}|V|\right). \label{eq:great-formula}
\end{equation}
Then, by the fact that $\mathbb{P}(A)\vee \mathbb{P}(B)\leq \mathbb{P}(A\cup B)\leq 2\left(\mathbb{P}(A)\vee \mathbb{P}(B)\right)$, we have
\begin{eqnarray*}
&& \mathbb{P}\left(\sqrt{r}|U|-(\sqrt{r+s}-\sqrt{s})|V|>t\sqrt{2\log n}\right) \\
&\asymp& \mathbb{P}\left(|\sqrt{r}U+\sqrt{s}V|-\sqrt{r+s}|V|>t\sqrt{2\log n}\right) \\
&& \vee \mathbb{P}\left(|\sqrt{r}U-\sqrt{s}V|-\sqrt{r+s}|V|>t\sqrt{2\log n}\right).
\end{eqnarray*}
We have derived the asymptotics of $\mathbb{P}\left(|\sqrt{r}U+\sqrt{s}V|-\sqrt{r+s}|V|>t\sqrt{2\log n}\right)$, and thus it is sufficient to analyze $\mathbb{P}\left(|\sqrt{r}U-\sqrt{s}V|-\sqrt{r+s}|V|>t\sqrt{2\log n}\right)$.
 {Since $|\sqrt{r}U-\sqrt{s}V|-\sqrt{r+s}|V|$ has the same distribution as that of $|\sqrt{r}U+\sqrt{s}V|-\sqrt{r+s}|V|$ under our assumptions, the desired bound is exactly the same as before.}
Hence, the proof is complete.
\end{proof}

\begin{proof}[Proof of Lemma \ref{lem:comp-tail-1}]
Consider independent random variables $W_1\sim N(0,r+s-\sqrt{s}\sqrt{r+s})$, $W_2\sim N(0,r+s-\sqrt{s}\sqrt{r+s})$, and $W_3\sim N(0,\sqrt{s}\sqrt{r+s})$. It is easy to check that
$$(\sqrt{r}U+\sqrt{s}V,\sqrt{r+s}V)\stackrel{d}{=}(W_1+W_3+(r+s)\sqrt{2\log n}, W_2+W_3+\sqrt{s}\sqrt{r+s}\sqrt{2\log n}).$$
Therefore,
\begin{eqnarray*}
&& \mathbb{P}\left(|\sqrt{r}U+\sqrt{s}V|-\sqrt{r+s}|V|>t\sqrt{2\log n}\right) \\
&=& \mathbb{P}\left(|W_1+W_3+(r+s)\sqrt{2\log n}|-|W_2+W_3+\sqrt{s}\sqrt{r+s}\sqrt{2\log n}|>t\sqrt{2\log n}\right) \\
&\asymp& \mathbb{P}\left(W_1+W_3+(r+s)\sqrt{2\log n}>t\sqrt{2\log n}+|W_2+W_3+\sqrt{s}\sqrt{r+s}\sqrt{2\log n}|\right) \\
&& + \mathbb{P}\left(-W_1-W_3-(r+s)\sqrt{2\log n}>t\sqrt{2\log n}+|W_2+W_3+\sqrt{s}\sqrt{r+s}\sqrt{2\log n}|\right) \\
&=& \mathbb{P}\left(W_1+W_3+(r+s)\sqrt{2\log n}>t\sqrt{2\log n}+W_2+W_3+\sqrt{s}\sqrt{r+s}\sqrt{2\log n},\right. \\
&& \left. W_1+W_3+(r+s)\sqrt{2\log n}>t\sqrt{2\log n}-W_2-W_3-\sqrt{s}\sqrt{r+s}\sqrt{2\log n}\right) \\
&& + \mathbb{P}\left(-W_1-W_3-(r+s)\sqrt{2\log n}>t\sqrt{2\log n}+W_2+W_3+\sqrt{s}\sqrt{r+s}\sqrt{2\log n},\right. \\
&& \left.-W_1-W_3-(r+s)\sqrt{2\log n}>t\sqrt{2\log n}-W_2-W_3-\sqrt{s}\sqrt{r+s}\sqrt{2\log n}\right) \\
&=&  \mathbb{P}\left(W_1-W_2 > (t-(r+s)+\sqrt{s}\sqrt{r+s})\sqrt{2\log n}\right) \\
&& \times \mathbb{P}\left(W_1+W_2+2W_3> (t-(r+s)-\sqrt{s}\sqrt{r+s})\sqrt{2\log n}\right) \\
&& + \mathbb{P}\left(-W_1-W_2-2W_3> (t+r+s+\sqrt{s}\sqrt{r+s})\sqrt{2\log n}\right) \\
&& \times \mathbb{P}\left(-W_1+W_2> (t+r+s-\sqrt{s}\sqrt{r+s})\sqrt{2\log n}\right),
\end{eqnarray*}
where we have used the fact that $W_1-W_2$, $W_1+W_2$, and $W_3$ are independent. For the four probabilities above, we have
\begin{eqnarray*}
&& \mathbb{P}\left(W_1-W_2 > (t-(r+s)+\sqrt{s}\sqrt{r+s})\sqrt{2\log n}\right) \\
&=& \mathbb{P}\left(N(0,1)>\frac{(t-(r+s)+\sqrt{s}\sqrt{r+s})\sqrt{2\log n}}{\sqrt{2(r+s-\sqrt{s}\sqrt{r+s})}}\right) \\
&\asymp& \begin{cases}
\frac{1}{\sqrt{\log n}}n^{-\frac{(t-(r+s)+\sqrt{s}\sqrt{r+s})^2}{2(r+s-\sqrt{s}\sqrt{r+s})}}, & t>r+s-\sqrt{s}\sqrt{r+s}, \\
1, & t\leq r+s-\sqrt{s}\sqrt{r+s},
\end{cases}
\end{eqnarray*}
\begin{eqnarray*}
&& \mathbb{P}\left(W_1+W_2+2W_3> (t-(r+s)-\sqrt{s}\sqrt{r+s})\sqrt{2\log n}\right) \\
&=& \mathbb{P}\left(N(0,1)>\frac{(t-(r+s)-\sqrt{s}\sqrt{r+s})\sqrt{2\log n}}{\sqrt{2(r+s+\sqrt{s}\sqrt{r+s})}}\right) \\
&\asymp& \begin{cases}
\frac{1}{\sqrt{\log n}}n^{-\frac{(t-(r+s)-\sqrt{s}\sqrt{r+s})^2}{2(r+s+\sqrt{s}\sqrt{r+s})}}, & t>r+s+\sqrt{s}\sqrt{r+s}, \\
1, & t\leq r+s+\sqrt{s}\sqrt{r+s},
\end{cases}
\end{eqnarray*}
\begin{eqnarray*}
&& \mathbb{P}\left(-W_1-W_2-2W_3> (t+r+s+\sqrt{s}\sqrt{r+s})\sqrt{2\log n}\right) \\
&=& \mathbb{P}\left(N(0,1)>\frac{(t+r+s+\sqrt{s}\sqrt{r+s})\sqrt{2\log n}}{\sqrt{2(r+s+\sqrt{s}\sqrt{r+s})}}\right) \\
&\asymp& \begin{cases}
\frac{1}{\sqrt{\log n}}n^{-\frac{(t+r+s+\sqrt{s}\sqrt{r+s})^2}{2(r+s+\sqrt{s}\sqrt{r+s})}}, & t>-r-s-\sqrt{s}\sqrt{r+s}, \\
1, & t\leq -r-s-\sqrt{s}\sqrt{r+s},
\end{cases}
\end{eqnarray*}
and
\begin{eqnarray*}
&& \mathbb{P}\left(-W_1+W_2> (t+r+s-\sqrt{s}\sqrt{r+s})\sqrt{2\log n}\right) \\
&=& \mathbb{P}\left(N(0,1)>\frac{(t+r+s-\sqrt{s}\sqrt{r+s})\sqrt{2\log n}}{\sqrt{2(r+s-\sqrt{s}\sqrt{r+s})}}\right) \\
&\asymp& \begin{cases}
\frac{1}{\sqrt{\log n}}n^{-\frac{(t+r+s-\sqrt{s}\sqrt{r+s})^2}{2(r+s-\sqrt{s}\sqrt{r+s})}}, & t>-r-s+\sqrt{s}\sqrt{r+s}, \\
1, & t\leq -r-s+\sqrt{s}\sqrt{r+s}.
\end{cases}
\end{eqnarray*}
Putting the pieces together, we get 
\begin{eqnarray*}
&& \mathbb{P}\left(|\sqrt{r}U+\sqrt{s}V|-\sqrt{r+s}|V|>t\sqrt{2\log n}\right) \\
&\asymp& \begin{cases}
\frac{1}{\log n}n^{-\frac{(t-r)^2+rs}{r}}, & t>r+s+\sqrt{s}\sqrt{r+s}, \\
\frac{1}{\sqrt{\log n}}n^{-\frac{(t-(r+s)+\sqrt{s}\sqrt{r+s})^2}{2(r+s-\sqrt{s}\sqrt{r+s})}}, & r+s-\sqrt{s}\sqrt{r+s}<t \leq r+s+\sqrt{s}\sqrt{r+s}, \\
1, & t\leq r+s-\sqrt{s}\sqrt{r+s}.
\end{cases}
\end{eqnarray*}
Therefore, the desired conclusion is obtained.

Finally, we consider $t<r+s-\sqrt{s}\sqrt{r+s}$. Write $U=Z_1+\sqrt{2r\log n}$ and $V=Z_2+\sqrt{2s\log n}$ with independent $Z_1,Z_2\sim N(0,1)$. Then, by using (\ref{eq:great-formula}), we have
\begin{eqnarray*}
&&  \mathbb{P}\left(|\sqrt{r}U+\sqrt{s}V|-\sqrt{r+s}|V|>t\sqrt{2\log n}\right) \\
&\geq& \mathbb{P}\left(\sqrt{r}|U|-(\sqrt{r+s}-\sqrt{s})|V|>t\sqrt{2\log n}\right) \\
&\geq& \mathbb{P}\left(\sqrt{r}(Z_1+\sqrt{2r\log n})-(\sqrt{r+s}-\sqrt{s})|Z_2|-(\sqrt{r+s}-\sqrt{s})\sqrt{2s\log n}>t\sqrt{2\log n}\right) \\
&=& \mathbb{P}\left(\frac{\sqrt{t}Z_1-(\sqrt{r+s}-\sqrt{s})|Z_2|}{\sqrt{2\log n}}>t-(r+s-\sqrt{s}\sqrt{r+s})\right) \rightarrow 1,
\end{eqnarray*}
since $\frac{\sqrt{t}Z_1-(\sqrt{r+s}-\sqrt{s})|Z_2|}{\sqrt{2\log n}}=o_{\mathbb{P}}(1)$. The proof is complete.
\end{proof}

\bibliographystyle{dcu}
\bibliography{reference}

\end{document}